\pdfoutput=1
\RequirePackage{ifpdf}
\ifpdf % We are running pdfTeX in pdf mode
\documentclass[pdftex]{sigma}
\else
\documentclass{sigma}
\fi

\usepackage[all]{xy}

\numberwithin{equation}{section}

\newtheorem{Theorem}{Theorem}[section]
\newtheorem{Lemma}[Theorem]{Lemma}
\newtheorem{Proposition}[Theorem]{Proposition}
 { \theoremstyle{definition}
\newtheorem{Definition}[Theorem]{Definition}
\newtheorem{Example}[Theorem]{Example}
\newtheorem{Remark}[Theorem]{Remark} }

\let\ssize\scriptstyle
\newif\ifFIRST\newdimen\MAXright\MAXright0pt
\def\sdynkin{\bgroup\eightpoint\dynkin}
\def\endsdynkin{\enddynkin\egroup}
\def\dynkin{\bgroup\FIRSTtrue\hskip.5em\setbox1\hbox{$\diagup$}%
\setbox2\hbox{$\diagdown$}%
\setbox0\hbox to2\wd1{\hrulefill}%
\setbox3\hbox{$\bullet$}%
\setbox4\hbox{$\times$}%
\setbox7\hbox{$\circ$}%       (L.K.)
\def\whiteroot##1{\ifFIRST\setbox5\hbox{$##1$}\ifdim\wd5>1.3em%       (L.K.)
\hskip-.5em\hskip.5\wd5\fi\fi\FIRSTfalse%                             (L.K.)
\hskip-.25em\raise1.5\wd3\hbox to0pt{\hss\hskip.45em$%                (L.K.)
\ssize##1$\hss}\copy7\hskip-.25em\setbox6\hbox{$##1$}%                (L.K.)
\MAXright\wd6}%                                                       (L.K.)
\def\root##1{\ifFIRST\setbox5\hbox{$##1$}\ifdim\wd5>1.3em%
\hskip-.5em\hskip.5\wd5\fi\fi\FIRSTfalse%
\hskip-.25em\raise1.5\wd3\hbox to0pt{\hss\hskip.45em$%
\ssize##1$\hss}\copy3\hskip-.25em\setbox6\hbox{$##1$}%
\MAXright\wd6}%
\def\whitedroot##1{\ifFIRST\setbox5\hbox{$##1$}\ifdim\wd5>1.3em% (L.K.)
\hskip-.5em\hskip.5\wd5\fi\fi\FIRSTfalse% (L.K.)
\hskip-.25em\lower1.8\wd3\hbox to0pt{\hss\hskip.45em$%  (L.K.)
\ssize##1$\hss}\copy7\hskip-.25em\setbox6\hbox{$##1$}% (L.K.)
\MAXright\wd6}%
\def\whiterroot##1{\hskip-.25em\copy7\hbox to0pt{\hskip.3em$\ssize##1$\hss}%
\hskip-.25em\setbox6\hbox{\hskip.6em$##1##1$}%
\MAXright\wd6}%
\def\droot##1{\ifFIRST\setbox5\hbox{$##1$}\ifdim\wd5>1.3em%
\hskip-.5em\hskip.5\wd5\fi\fi\FIRSTfalse%
\hskip-.25em\lower1.8\wd3\hbox to0pt{\hss\hskip.45em$%
\ssize##1$\hss}\copy3\hskip-.25em\setbox6\hbox{$##1$}%
\MAXright\wd6}%
\def\rroot##1{\hskip-.25em\copy3\hbox to0pt{\hskip.3em$\ssize##1$\hss}%
\hskip-.25em\setbox6\hbox{\hskip.6em$##1##1$}%
\MAXright\wd6}%
\def\norroot##1{\hskip-.36em\copy4\hbox to0pt{\hskip.3em$\ssize##1$\hss}%
\hskip-.48em\setbox6\hbox{\hskip.6em$##1##1$}%
\MAXright\wd6}%
\def\noroot##1{\ifFIRST\setbox5\hbox{$##1$}\ifdim\wd5>1.3em%
\hskip-.5em\hskip.5\wd5\fi\fi\FIRSTfalse%
\hskip-.36em\raise1.5\wd3\hbox to0pt{\hss\hskip.6em$%
\ssize##1$\hss}\copy4\hskip-.38em\setbox6\hbox{$##1$}%
\MAXright\wd6}%
\def\nodroot##1{\ifFIRST\setbox5\hbox{$##1$}\ifdim\wd5>1.3em%
\hskip-.5em\hskip.5\wd5\fi\fi\FIRSTfalse%
\hskip-.36em\lower1.8\wd3\hbox to0pt{\hss\hskip.6em$%
\ssize##1$\hss}\copy4\hskip-.38em\setbox6\hbox{$##1$}%
\MAXright\wd6}%
\def\nolink{\hskip\wd0}%      (L.K.)
\def\link{\raise.22em\copy0}%
\def\llink##1{\raise.32em\copy0\hskip-\wd0%
\raise.12em\copy0\hskip-.5\wd0\hbox to0pt{\hss$##1$\hss}\hskip.5\wd0}%
\def\lllink##1{\raise.22em\copy0\hskip-\wd0\raise.32em\copy0\hskip-\wd0%
\raise.12em\copy0\hskip-.5\wd0\hbox to0pt{\hss$##1$\hss}\hskip.5\wd0}%
\def\rootupright##1{\hbox to0pt{\raise.45em\copy1\hskip-.25em\raise1.3\ht1%
\hbox{\copy3\hskip.3em$\ssize##1$}\hss}%
\setbox6\hbox{\hskip.6em\copy1\copy1$##1##1$}%
\ifdim\MAXright<\wd6\MAXright\wd6\fi}%
\def\whiterootupright##1{\hbox to0pt{\raise.45em\copy1\hskip-.25em\raise1.3\ht1% (L.K.)
\hbox{\copy7\hskip.3em$\ssize##1$}\hss}% (L.K.)
\setbox6\hbox{\hskip.6em\copy1\copy1$##1##1$}% (L.K.)
\ifdim\MAXright<\wd6\MAXright\wd6\fi}% (L.K.)
\def\norootupright##1{\hbox to0pt{\raise.45em\copy1\hskip-.36em\raise1.3\ht1%
\hbox{\copy4\hskip.3em$\ssize##1$}\hss}%
\setbox6\hbox{\hskip.6em\copy1\copy1$##1##1$}%
\ifdim\MAXright<\wd6\MAXright\wd6\fi}%
\def\rootdownright##1{\hbox to0pt{\raise-.5em\copy2\hskip-.25em\raise-1.35\ht1%
\hbox{\copy3\hskip.3em$\ssize##1$}\hss}\setbox6%
\hbox{\hskip.6em\copy2\copy2$##1##1$}%
\ifdim\MAXright<\wd6\MAXright\wd6\fi}%
\def\whiterootdownright##1{\hbox to0pt{\raise-.5em\copy2\hskip-.25em\raise-1.35\ht1% (L.K.)
\hbox{\copy7\hskip.3em$\ssize##1$}\hss}\setbox6% (L.K.)
\hbox{\hskip.6em\copy2\copy2$##1##1$}% (L.K.)
\ifdim\MAXright<\wd6\MAXright\wd6\fi}% (L.K.)
\def\rootdown##1{\hbox to0pt{\hskip-.05em\vrule height.25em depth.65em%
\hskip-.25em\raise-.95em\hbox{\copy3\hskip.3em$\ssize##1$}\hss}%
\setbox6\hbox{$##1$}%
\ifdim\MAXright<\wd6\MAXright\wd6\fi}%
\def\whiterootdown##1{\hbox to0pt{\hskip-.05em\vrule height.25em depth.65em% (L.K.)
\hskip-.25em\raise-.95em\hbox{\copy7\hskip.3em$\ssize##1$}\hss}% (L.K.)
\setbox6\hbox{$##1$}% (L.K.)
\ifdim\MAXright<\wd6\MAXright\wd6\fi}% (L.K.)
\def\dots{\hskip.5em\cdots\hskip.5em}}%
\def\enddynkin{\ifdim\MAXright>1em\hskip.5\MAXright\else\hskip.5em\fi\egroup}%
 
This macro is used for creation of Dynkin diagrams in AMSTeX.
The usage:
In math-mode, between \dynkin and \enddynkin,
\root#1 makes a node with label #1 over this node,
\rroot#1 and \droot do the same with labels on right and down,
\noroot and \norroot do the same with cross instead of bullet,
\link links two adjacent \root's, \llink#1 links two adjacent nodes
with double-line with #1 (probably '>' or '<') in the middle, \lllink is does
same with three lines, \dots does \cdots between the two adjacent nodes,
\rootdown#1 makes a linked node down labeled by #1, \rootupright#1 and
\rootdownright#1 are labeled linked nodes in the indicated directions 
which do not count for the jorizontal dimensions.

\sdynkin ... \endsdynkin works exactly as above but in smaller
version (eightpoint)

Examples:
$\dynkin \root{a_1}\link\root{a_2}\dots\root{a_{n-1}}\link\root{a_n}
\enddynkin$
or
$\dynkin \root{}\lllink>\root{}\enddynkin$
or
$\dynkin \root{a}\link\root{b}\rootupright{c}\rootdownright{d}\enddynkin$
or
$\dynkin \root{}\link\root{}\rootupright{}\link\root{}\enddynkin$.
or (white roots!)
{$\dynkin \whiteroot{0}\link\noroot{0}\dots\root{0}\link\whiteroot{0}\rootupright{0}\whiterootdownright{0}\enddynkin$}. % (L.K.)

\def\Ytddj{                 %3221
\begin{picture}(12,16) 
\put(0,16){\line(1,0){12}}
\put(0,12){\line(1,0){12}}
\put(0,8){\line(1,0){8}}
\put(0,4){\line(1,0){8}}
\put(0,0){\line(1,0){4}}
\put(0,16){\line(0,-1){16}}
\put(4,16){\line(0,-1){16}}
\put(8,16){\line(0,-1){12}}
\put(12,16){\line(0,-1){4}}
\end{picture}
}

\def\Yctj{                 %431
\begin{picture}(16,12) 
\put(0,12){\line(1,0){16}}
\put(0,8){\line(1,0){16}}
\put(0,4){\line(1,0){12}}
\put(0,0){\line(1,0){4}}
\put(0,12){\line(0,-1){12}}
\put(4,12){\line(0,-1){12}}
\put(8,12){\line(0,-1){8}}
\put(12,12){\line(0,-1){8}}
\put(16,12){\line(0,-1){4}}
\end{picture}
}

\def\sYj{                  %1
\begin{picture}(3,3)
\put(0,3){\line(1,0){3}}
\put(0,0){\line(1,0){3}}
\put(0,3){\line(0,-1){3}}
\put(3,3){\line(0,-1){3}}
\end{picture}}

\def\sYjj{                 %11
\begin{picture}(3,6) 
\put(0,6){\line(1,0){3}}
\put(0,3){\line(1,0){3}}
\put(0,0){\line(1,0){3}}
\put(0,6){\line(0,-1){6}}
\put(3,6){\line(0,-1){6}}
\end{picture}
}

\def\sYjjj{                 %111
\begin{picture}(3,9) 
\put(0,9){\line(1,0){3}}
\put(0,6){\line(1,0){3}}
\put(0,3){\line(1,0){3}}
\put(0,0){\line(1,0){3}}
\put(0,9){\line(0,-1){9}}
\put(3,9){\line(0,-1){9}}
\end{picture}
}

\def\sYd{                 %2
\begin{picture}(6,3) 
\put(0,3){\line(1,0){6}}
\put(0,0){\line(1,0){6}}
\put(0,3){\line(0,-1){3}}
\put(3,3){\line(0,-1){3}}
\put(6,3){\line(0,-1){3}}
\end{picture}
}

\def\sYdj{                 %21
\begin{picture}(6,6) 
\put(0,6){\line(1,0){6}}
\put(0,3){\line(1,0){6}}
\put(0,0){\line(1,0){3}}
\put(0,6){\line(0,-1){6}}
\put(3,6){\line(0,-1){6}}
\put(6,6){\line(0,-1){3}}
\end{picture}
}

\def\sYdjj{                 %211
\begin{picture}(6,9) 
\put(0,9){\line(1,0){6}}
\put(0,6){\line(1,0){6}}
\put(0,3){\line(1,0){3}}
\put(0,0){\line(1,0){3}}
\put(0,9){\line(0,-1){9}}
\put(3,9){\line(0,-1){9}}
\put(6,9){\line(0,-1){3}}
\end{picture}
}

\def\sYdd{                 %22
\begin{picture}(6,6) 
\put(0,6){\line(1,0){6}}
\put(0,3){\line(1,0){6}}
\put(0,0){\line(1,0){6}}
\put(0,6){\line(0,-1){6}}
\put(3,6){\line(0,-1){6}}
\put(6,6){\line(0,-1){6}}
\end{picture}
}

\def\sYddj{                 %221
\begin{picture}(6,9) 
\put(0,9){\line(1,0){6}}
\put(0,6){\line(1,0){6}}
\put(0,3){\line(1,0){6}}
\put(0,0){\line(1,0){3}}
\put(0,9){\line(0,-1){9}}
\put(3,9){\line(0,-1){9}}
\put(6,9){\line(0,-1){6}}
\end{picture}
}

\def\sYt{                 %3
\begin{picture}(9,3) 
\put(0,3){\line(1,0){9}}
\put(0,0){\line(1,0){9}}
\put(0,3){\line(0,-1){3}}
\put(3,3){\line(0,-1){3}}
\put(6,3){\line(0,-1){3}}
\put(9,3){\line(0,-1){3}}
\end{picture}
}

\def\sYtj{                 %31
\begin{picture}(9,6) 
\put(0,6){\line(1,0){9}}
\put(0,3){\line(1,0){9}}
\put(0,0){\line(1,0){3}}
\put(0,6){\line(0,-1){6}}
\put(3,6){\line(0,-1){6}}
\put(6,6){\line(0,-1){3}}
\put(9,6){\line(0,-1){3}}
\end{picture}
}

\def\sYtjj{                 %311
\begin{picture}(9,9) 
\put(0,9){\line(1,0){9}}
\put(0,6){\line(1,0){9}}
\put(0,3){\line(1,0){3}}
\put(0,0){\line(1,0){3}}
\put(0,9){\line(0,-1){9}}
\put(3,9){\line(0,-1){9}}
\put(6,9){\line(0,-1){3}}
\put(9,9){\line(0,-1){3}}
\end{picture}
}

\def\sYtd{                 %32
\begin{picture}(9,6) 
\put(0,6){\line(1,0){9}}
\put(0,3){\line(1,0){9}}
\put(0,0){\line(1,0){6}}
\put(0,6){\line(0,-1){6}}
\put(3,6){\line(0,-1){6}}
\put(6,6){\line(0,-1){6}}
\put(9,6){\line(0,-1){3}}
\end{picture}
}

\def\sYtdj{                 %321
\begin{picture}(9,9) 
\put(0,9){\line(1,0){9}}
\put(0,6){\line(1,0){9}}
\put(0,3){\line(1,0){6}}
\put(0,0){\line(1,0){3}}
\put(0,9){\line(0,-1){9}}
\put(3,9){\line(0,-1){9}}
\put(6,9){\line(0,-1){6}}
\put(9,9){\line(0,-1){3}}
\end{picture}
}

\begin{document}

\newcommand{\arXivNumber}{1705.09469}

\renewcommand{\PaperNumber}{012}

\FirstPageHeading

\ShortArticleName{$k$-Dirac Complexes}

\ArticleName{$\boldsymbol{k}$-Dirac Complexes}

\Author{Tom\'a\v{s} SALA\v{C}}

\AuthorNameForHeading{T.~Sala\v{c}}

\Address{Mathematical Institute, Charles University, Sokolovsk\'a 49/83, Prague, Czech Republic}
\Email{\href{mailto:salac@karlin.mff.cuni.cz}{salac@karlin.mff.cuni.cz}}
\URLaddress{\url{http://www.karlin.mff.cuni.cz/~salac/}}

\ArticleDates{Received June 01, 2017, in f\/inal form February 06, 2018; Published online February 16, 2018}

\Abstract{This is the f\/irst paper in a series of two papers. In this paper we construct complexes of invariant dif\/ferential operators which live on homogeneous spaces of $|2|$-graded parabolic geometries of some particular type. We call them $k$-Dirac complexes. More explicitly, we will show that each $k$-Dirac complex arises as the direct image of a relative BGG sequence and so this f\/its into the scheme of the Penrose transform. We will also prove that each $k$-Dirac complex is formally exact, i.e., it induces a long exact sequence of inf\/inite (weighted) jets at any f\/ixed point. In the second part of the series we use this information to show that each $k$-Dirac complex is exact at the level of formal power series at any point and that it descends to a resolution of the $k$-Dirac operator studied in Clif\/ford analysis.}


\Classification{58J10; 32N05; 32L25; 35A22; 53C28; 58A20}

\section{Introduction}

Let $h$ be a non-degenerate, symmetric and $\mathbb{C}$-bilinear form on~$\mathbb{C}^{2m}$. The Grassmannian variety~$M$ of totally isotropic $k$-dimensional subspaces in $\mathbb{C}^{2m}$ is a homogeneous space of a~$|2|$-graded parabolic geometry. We assume throughout the paper that $n:=m-k\ge k\ge2$. We will show that on~$M$ there is a complex of invariant dif\/ferential operators which we call the $k$-\textit{Dirac complex}.
The main result of this article is (see Theorem~\ref{theorem formal exactness}) that the complex is formally exact (as explained above in the abstract) in the sense of~\cite{Sp}.

This result is crucial for an application in \cite{S} where it is shown that the complex is exact with formal power series at any f\/ixed point and that it descends (as outlined in the recent series \cite{CSaI, CSaII,CSaIII}) to a resolution of the $k$-Dirac operator studied in Clif\/ford analysis (see~\cite{CSSS,SSSL}). As a~potential application of the resolution, there is an open problem of characterizing the domains of monogenicity, i.e., an open set~$\mathcal{U}$ is a domain of monogenicity if there is no open set~$\mathcal{U}'$ with~$\mathcal{U}\subsetneq\mathcal{U}'$ such that each monogenic function\footnote{A monogenic function is a null solution of the $k$-Dirac operator.} in~$\mathcal{U}$ extends to a monogenic function in $\mathcal{U}'$. Recall from \cite[Section~4]{H} that the Dolbeault resolution together with some $L^2$ estimates are crucial in a proof of the statement that any pseudoconvex domain is a domain of holomorphy.

The $k$-Dirac complexes belong to the singular inf\/initesimal character and so the BGG machinery introduced in~\cite{CaSS} is not available. However, we will show that each $k$-Dirac complex arises as the direct image of a relative BGG sequence (see \cite{CSo,CSoI} for a recent publication on this topic) and so, this paper f\/its into the scheme of the Penrose transform (see \cite{BE,WW}). In particular, we will work here only in the setting of complex parabolic geometries.

The machinery of the Penrose transform is a main tool used in \cite{B}. The main result of that article is a construction of families of locally exact complexes of invariant dif\/ferential operators on quaternionic manifolds. One of these quaternionic complexes can (see \cite{BAS,BS,CSS}) be identif\/ied with a resolution of the $k$-\textit{Cauchy--Fueter operator} which has been intensively studied in Clif\/ford analysis (see again \cite{CSSS,SSSL}). In order to prove the local exactness of this quaternionic complex, one uses that an almost quaternionic structure is a $|1|$-graded parabolic geometry and the theory of constant coef\/f\/icient operators from~\cite{Na}.

Unfortunately, the parabolic geometry on $M$ is $|2|$-graded and so there is canonical a 2-step f\/iltration of the tangent bundle of $M$ given by a bracket generating distribution. With such a~structure, it is more natural to work with weighted jets (see~\cite{Mo}) rather than usual jets and we use this concept also here, i.e., we prove the formal exactness of the $k$-Dirac complexes with respect to the weighted jets.
Nevertheless, we will prove in~\cite{S} that the formal exactness of the $k$-Dirac complex (or more precisely the exactness of~(\ref{les of graded jets II}) for each $\ell+j\ge0$) is enough to conclude that it descends to a resolution of the $k$-Dirac operator.

We consider here only the even case $\mathbb{C}^{2m}$ as, due to the representation theory, the Penrose transform does not work in odd dimension $\mathbb{C}^{2m+1}$ and it seems that this case has to be treated by completely dif\/ferent methods. The assumption $n\ge k$ is (see \cite{CSSS}) called the stable range. This assumption is needed only in Proposition~\ref{thm direct images} where we compute direct images of sheaves that appear in the relative BGG sequences. Hence, it is reasonable to expect that (see also~\cite{K}) the machinery of the Penrose transform provides formally exact complexes also in the unstable range $n<k$.

For the application in \cite{S}, we need to show that the $k$-Dirac complexes constructed in this paper give rise to complexes from \cite{TS} which live on the corresponding real parabolic geometries. This turns out to be rather easy since any linear $\mathrm{G}$-invariant operator is determined by a certain $\mathrm{P}$-equivariant homomorphism. As this correspondence works also in the smooth setting, passing from the holomorphic setting to the smooth setting is rather elementary.

The abstract approach of the Penrose transform is not very helpful when one is interested in local formulae of dif\/ferential operators.
Local formulae of the operators in the $k$-Dirac complexes can be found in~\cite{TS}. Notice that in this article we construct only one half of each complex from~\cite{TS}. This is due to the fact that the complex space of spinors decomposes into two irreducible $\mathfrak{so}(2n,\mathbb{C})$-sub-representations. The other half of each $k$-Dirac complex can (see Remark~\ref{remark PT for other Dirac}) be easily obtained by adapting results from this paper.

Finally, let us mention few more articles which deal with the $k$-Dirac complexes.
The null solutions of the f\/irst operator in the $k$-Dirac complex were studied in \cite{TSI,TSII}. The singular Hasse graphs and the corresponding homomorphisms of generalized Verma modules were computed in~\cite{F}.

\subsection*{Notation}
\begin{itemize}\itemsep=0pt
\item $M(n,k,\mathbb{C})$ matrices of size $n\times k$ with complex coef\/f\/icients;
\item $M(n,\mathbb{C}):=M(n,n,\mathbb{C})$;
\item $A(n,\mathbb{C}):=\{A\in M(n,\mathbb{C})\,|\, A^T=-A\}$;
\item $1_n$ identity $n\times n$-matrix;
\item $[v_1,\dots,v_\ell]$ linear span of vectors $v_1,\dots,v_\ell$.
\end{itemize}

%\tableofcontents

\section{Preliminaries}\label{section Preliminaries}

In Section \ref{section Preliminaries} we will review some well known material. Namely, in Section~\ref{section review} we will summarize some theory of complex parabolic geometries. We will recall in Section~\ref{section wdo} the concept of weighted jets on f\/iltered manifolds and in Section~\ref{section ideal sheaf} the def\/inition of the normal bundle associated to analytic subvariety and the formal neighborhood. In Section~\ref{section PT} we will give a~short summary of the Penrose transform.

See \cite{CS} for a thorough introduction into the theory of parabolic geometries. The concept of weighted jets was originally introduced in the smooth setting by Tohru Morimoto, see for example~\cite{Mo}. Sections~\ref{section ideal sheaf} and~\ref{section PT} were taken mostly from \cite{BE,WW}.

\subsection{Review of parabolic geometries}\label{section review}
Let $\mathfrak{g}$ be a complex semi-simple Lie algebra, $\mathfrak{h}$ be a Cartan subalgebra, $\triangle$ be the associated set of roots, $\triangle^+$ be a set of positive roots and $\triangle^0=\{\alpha_1,\dots,\alpha_m\}$ be the associated set of (pairwise distinct) simple roots.
We will denote by $\mathfrak{g}_\alpha$ the root space associated to $\alpha\in\triangle$ and we will write $\alpha>0$ if $\alpha\in\triangle^+$ and $\alpha<0$ if $-\alpha\in\triangle^+$. Given $\alpha\in\triangle$, there are unique integers $\lambda_1,\dots,\lambda_m$ such that $\alpha=\lambda_1\alpha_1+\dots+\lambda_m\alpha_m$. If $\Sigma\subset\triangle^0$, then the integer $ht_\Sigma(\alpha):=\sum\limits_{i\colon \alpha_i\in\Sigma}\lambda_i$ is called the $\Sigma$-\textit{height} of $\alpha$. We put $\mathfrak{g}_i:=\bigoplus\limits_{\alpha\colon ht_\Sigma(\alpha)=i}\mathfrak{g}_{\alpha}$. Then there is an integer $k\ge0$ such that $\mathfrak{g}_k\ne\{0\}$, $\mathfrak{g}_\ell=\{0\}$ whenever $|\ell|>k$ and
\begin{gather}\label{k-gradation}
\mathfrak{g}=\mathfrak{g}_{-k}\oplus\mathfrak{g}_{-k+1}\oplus\dots\oplus\mathfrak{g}_{k-1}\oplus\mathfrak{g}_k.
\end{gather}
The direct sum decomposition (\ref{k-gradation}) is the $|k|$-\textit{grading} on $\mathfrak{g}$ associated to $\Sigma$.

Since $[\mathfrak{g}_\alpha,\mathfrak{g}_\beta]\subset\mathfrak{g}_{\alpha+\beta}$ for every $\alpha,\beta\in\triangle$, it follows that $[\mathfrak{g}_i,\mathfrak{g}_j]\subset\mathfrak{g}_{i+j}$ for each \smash{$i,j\in\mathbb{Z}$}. In particular, it follows that $\mathfrak{g}_0$ is a subalgebra and it can be shown that it is always reductive, i.e., $\mathfrak{g}_0=\mathfrak{g}_0^{ss}\oplus\mathfrak z(\mathfrak{g}_0)$ where $\mathfrak{g}_0^{ss}:=[\mathfrak{g}_0,\mathfrak{g}_0]$ is semi-simple and $\mathfrak z(\mathfrak{g}_0)$ is the center (see \cite[Corollary~2.1.6]{CS}). Moreover, each subspace $\mathfrak{g}_i$ is $\mathfrak{g}_0$-invariant and the Killing form of $\mathfrak{g}$ induces an isomorphism $\mathfrak{g}_{-i}\cong\mathfrak{g}_i^\ast$ of $\mathfrak{g}_0$-modules where $ ^\ast$ denotes the dual module.
We put
\begin{gather*}
\mathfrak{g}^i:=\bigoplus_{j\ge i} \mathfrak{g}_j, \qquad \mathfrak{g}_-:=\mathfrak{g}_{-k}\oplus\dots\oplus\mathfrak{g}_{-1}, \qquad \mathfrak{p}_\Sigma:=\mathfrak{g}^0\qquad \mathrm{and}\qquad \mathfrak{p}_+:=\mathfrak{g}^1.
\end{gather*}
Then $\mathfrak{p}_\Sigma$ is the \textit{parabolic subalgebra} associated to the $|k|$-grading and $\mathfrak{p}_\Sigma=\mathfrak{g}_0\oplus\mathfrak{p}_+$ is known as the Levi decomposition (see \cite[Section~2.2]{BE}). This means that $\mathfrak{p}_+$ is the nilradical\footnote{Recall that the nilradical is a maximal nilpotent ideal and that it is unique.} and that~$\mathfrak{g}_0$ is a maximal reductive subalgebra called the \textit{Levi factor}. It is clear that each subspace~$\mathfrak{g}^i$ is $\mathfrak{p}_\Sigma$-invariant and that $\mathfrak{g}_-$ is a nilpotent subalgebra. Moreover, it can be shown that, as a~Lie algebra, $\mathfrak{g}_-$ is generated by~$\mathfrak{g}_{-1}$.

The algebra $\mathfrak{b}:=\mathfrak{p}_{\triangle^0}$ is called the \textit{standard Borel subalgebra}. A subalgebra of $\mathfrak{g}$ is called \textit{standard parabolic} if it contains $\mathfrak{b}$ and in particular, $\mathfrak{p}_\Sigma$ is such an algebra. More generally, a~subalgebra of $\mathfrak{g}$ is called a \textit{Borel subalgebra} and a \textit{parabolic subalgebra} if it is conjugated to the standard Borel subalgebra and to a standard parabolic subalgebra, respectively. We will for brevity sometimes write~$\mathfrak{p}$ instead of~$\mathfrak{p}_\Sigma$.

Let $s_i$ be the simple ref\/lection associated to $\alpha_i$, $i=1,2,\dots,m$, $W_\mathfrak{g}$ be the Weyl group of $\mathfrak{g}$ and $W_\mathfrak{p}$ be the subgroup of $W_\mathfrak{g}$ generated by $\{s_i\colon \alpha_i\not\in\Sigma\}$. Then $W_\mathfrak{p}$ is isomorphic to the Weyl group of $\mathfrak{g}_0^{ss}$ and the directed graph that encodes the Bruhat order on $W_\mathfrak{g}$ contains a subgraph called the \textit{Hasse diagram $W^\mathfrak{p}$ attached to $\mathfrak{p}$} (see \cite[Section 4.3]{BE}). The vertices of $W^\mathfrak{p}$ consist of those $w\in W_\mathfrak{g}$ such that $w.\lambda$ is $\mathfrak{p}$-dominant for any $\mathfrak{g}$-dominant weight $\lambda$ where the dot stands for the af\/f\/ine action, namely, $w.\lambda=w(\lambda+\rho)-\rho$ where $\rho$ is the lowest form. It turns out that each right coset of $W_\mathfrak{p}$ in $W_\mathfrak{g}$ contains a unique element from $W^\mathfrak{p}$ and it can be shown that this is the element of minimal length (see \cite[Lemma 4.3.3]{BE}). This identif\/ies $W^\mathfrak{p}$ with $W_\mathfrak{p}\backslash W_\mathfrak{g}$.

We will need also the relative case. Assume that $\Sigma'\subset\triangle^0$ and put $\mathfrak{r}:=\mathfrak{p}_{\Sigma'}$. Then $\mathfrak{q}:=\mathfrak{r}\cap\mathfrak{p}=\mathfrak{p}_{\Sigma\cup\Sigma'}$ is a standard parabolic subalgebra and $\mathfrak l:=\mathfrak{g}_0^{ss}\cap\mathfrak{q}$ is a parabolic subalgebra of $\mathfrak{g}_0^{ss}$ (see \cite[Section 2.4]{BE}). The def\/inition of the Hasse diagram attached to $\mathfrak{p}$ applies also to the pair $(\mathfrak{g}_0^{ss},\mathfrak l)$, namely an element $w\in W_\mathfrak{p}$ (as \cite[Section 4.4]{BE}) belongs to the \textit{relative Hasse diagram} $W^\mathfrak{q}_\mathfrak{p}$ if it is the element of minimal length in its right coset of $W_\mathfrak{q}$ in $W_\mathfrak{p}$. Hence, $W^\mathfrak{q}_\mathfrak{p}$ is a subset of $W_\mathfrak{g}$ which can be naturally identif\/ied with $W_\mathfrak{q}\backslash W_\mathfrak{p}$.

There is (up to isomorphism) a unique connected and simply connected complex Lie group~$\mathrm{G}$ with Lie algebra $\mathfrak{g}$. Assume that $\Sigma=\{\alpha_{i_1},\dots,\alpha_{i_j}\}$. Let $\omega_1,\dots,\omega_m$ be the fundamental weights associated to the simple roots and
$\mathbb{V}$ be an irreducible $\mathfrak{g}$-module with highest weight $\lambda:=\omega_{i_1}+\dots+\omega_{i_j}$. Since any $\mathfrak{g}$-representation integrates to $\mathrm{G}$, $\mathbb{V}$ is also a $\mathrm{G}$-module. The action descends to the projective space $\mathbb P(\mathbb{V})$ and the stabilizer of the line spanned by a non-zero highest weight vector $v$ is the associated \textit{parabolic subgroup} $\mathrm{P}$. This is by def\/inition a closed subgroup of $\mathrm{G}$ and its Lie algebra is $\mathfrak{p}$. The homogeneous space $\mathrm{G}/\mathrm{P}$ is biholomorphic to the $\mathrm{G}$-orbit of $[v]\in\mathbb P(\mathbb{V})$ and since it is completely determined by~$\Sigma$, we denote it by crossing in the Dynkin diagram of $\mathfrak{g}$ the simple roots from $\Sigma$. We will for brevity put $M:=\mathrm{G}/\mathrm{P}$ and denote by $\textbf{p}\colon \mathrm{G}\rightarrow M$ the canonical projection.

On $\mathrm{G}$ lives a tautological $\mathfrak{g}$-valued 1-form $\omega$ which is known as the \textit{Maurer-Cartan form}. This form is $\mathrm{P}$-equivariant in the sense that for each $p\in\mathrm{P}\colon (r^p)^\ast\omega=\operatorname{Ad}(p^{-1})\circ\omega$ where $\operatorname{Ad}$ is the adjoint representation and $r^p$ is the principal action of $p$. If $\mathbb{V}$ is a subspace of $\mathfrak{g}$ and $g\in\mathrm{G}$, then $\omega^{-1}_g(\mathbb{V})$ is a subspace of~$T_g\mathrm{G}$ and the disjoint union $\bigcup_{g\in\mathrm{G}}\omega_g^{-1}(\mathbb{V})$ determines a distribution on~$\mathrm{G}$ which we for brevity denote by $\omega^{-1}(\mathbb{V})$. Since $T\textbf{p}=T\textbf{p}\circ Tr^p$, it follows that $T\textbf{p}(\omega^{-1}(\mathbb{V}))$ is a well-def\/ined distribution on $M$ provided that $\mathbb{V}$ is $\mathrm{P}$-invariant. In particular, this applies to~$\mathfrak{g}^i$ and we put $F_{i}:=T\textbf{p}(\omega^{-1}(\mathfrak{g}^{i}))$. Since $\ker(T\textbf{p})=\omega^{-1}(\mathfrak{p})$, it follows that the f\/iltration $\mathfrak{g}^{-k}/\mathfrak{p}\supset\dots\supset\mathfrak{g}^{-1}/\mathfrak{p}\supset\mathfrak{g}^0/\mathfrak{p}$ gives a~f\/iltration $TM=F_{-k}\supset F_{-k+1}\supset\dots\supset F_{-1}\supset F_0=\{0\}$ of the tangent bundle $TM$ where~$\{0\}$ is the zero section.
The \textit{graded tangent bundle} associated to the f\/iltration $\{F_{-i}\colon i=k,\dots,0\}$ is $gr(TM):=\bigoplus_{i=-k}^{-1}gr_i(TM)$ where $gr_i(TM):=F_{i}/F_{i+1}$. Since $M$ is the homogeneous model, we have the following:
\begin{itemize}\itemsep=0pt
 \item the f\/iltration is compatible\footnote{Filtrations which satisfy this property are called \textit{regular}.} with the Lie bracket of vector f\/ields in the sense that the commutator of a section of $F_i$ and a section of $F_j$ is a section of $F_{i+j}$,
 \item the Lie bracket descends to a vector bundle map $\mathcal{L}\colon \Lambda^2gr(TM)\rightarrow gr(TM)$, called the \textit{Levi form}, which is homogeneous of degree zero and
 \item $(gr(T_xM),\mathcal{L}_x)$, $x\in M$ is a nilpotent Lie algebra isomorphic to $\mathfrak{g}_-$.
\end{itemize}
Hence, $(gr(TM),\mathcal{L})$ is a locally trivial bundle of nilpotent Lie algebras with typical f\/iber $\mathfrak{g}_-$ and it follows that $F_{-1}$ is a bracket generating distribution.

We denote by $T^\ast M=\Lambda^{1,0}T^\ast M$ the vector bundle dual to $TM$, i.e., the f\/iber over $x\in M$ is the space of $\mathbb{C}$-linear maps $T_x M\rightarrow\mathbb{C}$. The f\/iltration of $TM$ induces a f\/iltration $T^\ast M=F_1\supset F_2\supset\dots\supset F_{k+1}=\{0\}$ where $F_{i}:=F_{-i+1}^\perp$ is the annihilator of $F_{-i+1}$. We put $gr_i(T^\ast M):= F_i/F_{i+1}$, $i=1,2,\dots,k$ so that $gr(T^\ast M)=\bigoplus_{i=1}^kgr_i(T^\ast M)$ is the associated graded vector bundle and $gr_i(T^\ast M)$ is isomorphic to the dual of~$gr_i(TM)$.

\subsection{Weighted dif\/ferential operators}\label{section wdo}
Let $M$ be the homogeneous space with the regular f\/iltration $\{F_{-j}\colon  j=0,\dots,k\}$ as in Section~\ref{section review}.
As $M$ is a complex manifold, $TM_\mathbb{C}:=TM\otimes\mathbb{C}=T^{1,0}M\oplus T^{0,1}M$ where the f\/irst and the second summand is the holomorphic and the anti-holomorphic part\footnote{The holomorphic and anti-holomorphic part is the $-i$ and the $+i$-eigenspace, respectively, for the canonical almost complex structure on~$TM_\mathbb{C}$.}, respectively. As each vector bundle~$F_{-j}$ is a~holomorhic sub-bundle of $TM$, we have $F_{-j}\otimes\mathbb{C}=F^{1,0}_{-j}\oplus F^{0,1}_{-j}$ as above.

Let $\mathcal{U}$ be an open subset of $M$ and $X$ be a holomorphic vector f\/ield on $\mathcal{U}$. The \textit{weighted order} $\mathfrak{wo}(X)$ of~$X$ is the smallest integer $j$ such that $X\in\Gamma(F_{-j}^{1,0}|_\mathcal{U})$. A dif\/ferential operator $D$ acting on the space $\mathcal{O}(\mathcal{U})$ of holomorphic functions on $\mathcal{U}$ is called a \textit{differential operator of weighted order at most} $r$ if for each $x\in \mathcal{U}$ there is an open neighborhood $\mathcal{U}_x$ of $x$ with a local framing\footnote{This means that the holomorphic vector f\/ields $X_1,\dots,X_p$ trivialize $T^{1,0}M$ over $\mathcal{U}_x$.} $\{X_1,\dots,X_p\}$ by holomorphic vector f\/ields such that
\begin{gather*}%\label{weighted differential operator}
D|_{\mathcal{U}_x}=\sum_{a\in\mathbb N_{0}^p} f_a X_1^{a_1}\cdots X_p^{a_p},
\end{gather*}
where $\mathbb N_0^p:=\{a=(a_1,\dots,a_p)\colon a_i\in\mathbb{Z},\ a_i\ge0,\ i=1,\dots,p\}$, $f_a\in\mathcal{O}(\mathcal{U}_x)$ and for all $a$ in the sum with $f_a$ non-zero: $\sum\limits_{i=1}^pa_i.\mathfrak{wo}(X_i)\le r$. We write $\mathfrak{wo}(D)\le r$.

Let $\mathcal{O}_x$ be the space of germs of holomorphic functions at $x$. We denote by $\mathfrak{F}_x^{i}$ the space of those germs $f\in\mathcal{O}_x$ such that $Df(x)=0$ for every dif\/ferential operator $D$ which is def\/ined on an open neighborhood of $x$ and $\mathfrak{wo}(D)\le i$. We put $\mathfrak{J}^i_x:=\mathcal{O}_x/\mathfrak{F}_{x}^{i+1}$, denote by $\mathfrak{j}^i_xf\in\mathfrak{J}^i_x$ the class of $f$ and call it the $i$-\textit{th weighted jet of} $f$. Then the disjoint union $\mathfrak{J}^i:=\cup_{x\in M}\mathfrak{J}_x^i$ is naturally a~holomorphic vector bundle over~$M$, the canonical vector bundle map $\mathfrak{J}^i\xrightarrow{\pi_i}\mathfrak{J}^{i-1}$ has constant rank and thus, its kernel $\mathfrak{gr}^i$ is again a~holomorphic vector bundle with f\/iber $\mathfrak{gr}^i_x$ over $x$. Notice that for each integer $i\ge0$ there is a short exact sequence $0\rightarrow\mathfrak{F}^{i+1}_x\rightarrow\mathfrak{F}^{i}_x\rightarrow\mathfrak{gr}^{i+1}_x\rightarrow0$ of vector spaces.

Assume that $V$ is a holomorphic vector bundle over $M$. We denote by $V^\ast$ the dual bundle, by $\langle-,-\rangle$ the canonical pairing between $V$ and $V^\ast$ and f\/inally, by $\mathcal{O}(V)_x$ the space of germs of holomorphic sections of $V$ at $x$. We def\/ine $\mathfrak{F}_{x}^iV$ as the space of germs $s\in\mathcal{O}(V)_x$ such that $\langle\lambda,s\rangle\in\mathfrak{F}_x^i$ for each $\lambda\in\mathcal{O}(V^\ast)_x$. We put $\mathfrak{J}^i_x:=\mathcal{O}(V)_x/\mathfrak{F}_{x}^{i+1}V$, denote by $\mathfrak{j}^i_xs\in\mathfrak{J}^i_xV$ the equivalence class of $s$ and call it the $i$-\textit{th weighted jet of} $s$. Then the disjoint union $\mathfrak{J}^i V:=\bigcup_{x\in M}\mathfrak{J}^i_xM$ is naturally a holomorphic vector bundle over $M$, the canonical bundle map $\mathfrak{J}^iV\xrightarrow{\pi_i}\mathfrak{J}^{i-1}V$ has constant rank and thus, its kernel $\mathfrak{gr}^i V$ is again a holomorphic vector bundle and we denote by $\mathfrak{gr}^i_xV$ its f\/iber over $x$. As above, there is for each integer $i\ge0$ a short exact sequence $0\rightarrow\mathfrak{F}^{i+1}_x V\rightarrow\mathfrak{F}^{i}_xV\rightarrow\mathfrak{gr}^{i+1}_xV\rightarrow0$ and just as in the smooth case, there is a canonical linear isomorphism $\mathfrak{gr}^i_x\otimes V_x\rightarrow\mathfrak{gr}^i_xV$.

\begin{Remark}\label{remark usual jets}
If the f\/iltration is trivial, i.e., $F_{-1}=TM$, then the concept of weighted jets agrees with that of usual jets. In this case we will use calligraphic letters instead of Gothic letters, i.e., we write $\mathcal F^i$ and $\mathcal J^i$ and $gr^i$ and $j^i_xf$ instead of $\mathfrak{F}^i$ and $\mathfrak{J}^i$ and $\mathfrak{gr}^i$ and $\mathfrak{j}^i_xf$, respectively. The vector bundle $gr^i$ is canonically isomorphic to the $i$-th symmetric power $S^iT^\ast M$.
\end{Remark}

Assume that there is a $\mathrm{P}$-module $\mathbb{V}$ such that $V$ is isomorphic to the $\mathrm{G}$-homogeneous vector bundle $\mathrm{G}\times_\mathrm{P}\mathbb{V}$. Let $e$ be the identity element of $\mathrm{G}$. Then we call the point $x_0:=e\mathrm{P}$ the origin of $M$ and we put
\begin{gather}
\label{notation weighted jets over origin}
\mathfrak{J}^i\mathbb{V}:=\mathfrak{J}^i_{x_0}V \qquad \mathrm{and}\qquad \mathfrak{gr}^i\mathbb{V}:=\mathfrak{gr}^i_{x_0}V.
\end{gather}
There are linear isomorphisms
\begin{gather}\label{isom of graded jets over origin}
\mathfrak{gr}^r\mathbb{V}\cong\mathfrak{gr}^{r}_{x_0}\otimes\mathbb{V} \cong\bigoplus_{i_1+2i_2+\dots+ki_k=r}S^{i_1}\mathfrak{g}_1\otimes S^{i_2}\mathfrak{g}_2\otimes\dots \otimes S^{i_k}\mathfrak{g}_k\otimes\mathbb{V}.
\end{gather}

We will be interested in the sub-bundle $ S^i gr_1(T^\ast M)\otimes V$ of $\mathfrak{gr}^iV$. Notice that the f\/iber of this sub-bundle over $x\in M$ is $\{\mathfrak{j}^i_xf\colon f\in\mathcal{O}(V)_x,\ j^{i-1}_xf=0\}$, i.e., the vector space of all weighted $i$-th jets of germs of holomorphic sections at $x$ whose usual $(i-1)$-th jet vanishes. The f\/iber of this bundle over $x_0$ is isomorphic to $S^i\mathfrak{g}_1\otimes\mathbb{V}$ and we denote it for brevity by $gr^i\mathbb{V}$.

Suppose that $\mathbb{W}$ is another $\mathrm{P}$-module and $W:=\mathrm{G}\times_\mathrm{P}\mathbb{W}$ be the associated homogeneous vector bundle. We say that the \textit{weighted order} of a linear dif\/ferential operator $D\colon \Gamma(V)\rightarrow\Gamma(W)$ is at most $r$ if for each $x\in M$, $ s\in\mathcal{O}(V)_x\colon \mathfrak{j}^r_xs=0\Rightarrow Ds(x)=0$. It is well known (see~\cite{Mo}) that $D$ induces for each $i\ge0$ a vector bundle map $\mathfrak{gr}^{i}V\rightarrow\mathfrak{gr}^{i-r}W$ where we agree that $\mathfrak{gr}^\ell W=0$ if $\ell <0$. The restriction of this map to the f\/ibers over the origin is a linear map
\begin{gather*}%\label{vector bundle map from diff op}
 \mathfrak{gr} D\colon \ \mathfrak{gr}^{i}\mathbb{V}\rightarrow\mathfrak{gr}^{i-r}\mathbb{W}.
\end{gather*}

\subsection{Ideal sheaf of an analytic subvariety}\label{section ideal sheaf}

Let us f\/irst recall some basics from the theory of sheaves (see for example~\cite{WW}). Suppose that~$\mathcal{F}$ and~$\mathcal{G}$ are sheaves on topological spaces $X$ and $Y$, respectively, and that $\iota\colon X\rightarrow Y$ is a~continuous map. We denote by~$\mathcal{F}_x$ the stalk of $\mathcal{F}$ at $x\in X$ and by $\mathcal{F}(\mathcal{U})$ or by $\Gamma(\mathcal{U},\mathcal{F})$ the space of sections of $\mathcal{F}$ over an open set $\mathcal{U}$. Then the pullback sheaf $\iota^{-1}\mathcal{G}$ is a sheaf on $X$ and the direct image $\iota_\ast\mathcal{F}$ is a sheaf on~$Y$. The $q$-th direct image $\iota^q_\ast\mathcal{F}$ is a sheaf on $Y$, it is def\/ined as the sheaf\/if\/ication of the pre-sheaf $\mathcal{V}\mapsto H^q(\iota^{-1}(\mathcal{V}),\mathcal{F})$ where $\mathcal{V}$ is open in~$Y$.

Suppose now that $X$ and $Y$ are complex manifolds with structure sheaves of holomorphic functions $\mathcal{O}_X$ and $\mathcal{O}_Y$, respectively, that $\iota$ is holomorphic and that $\mathcal{G}$ is a sheaf of $\mathcal{O}_Y$-modules. Then $\iota^{-1}\mathcal{G}$ is in general not a sheaf of $\mathcal{O}_X$-modules. To f\/ix this problem, we use that $\iota^{-1}\mathcal{O}_Y$ is naturally a sub-sheaf of $\mathcal{O}_X$ and
 def\/ine a new sheaf $\iota^\ast\mathcal{G}:=\mathcal{O}_X\otimes_{\iota^{-1}\mathcal{O}_Y}\iota^{-1}\mathcal{G}$. Then $\iota^\ast\mathcal{G}$ is by construction a sheaf of $\mathcal{O}_X$-modules.

Now we can continue with the def\/inition of the ideal sheaf. Suppose that the holomorphic map $\iota$ is an embedding. The restriction $TY|_X$ contains the tangent bundle $TX$ of $X$. The normal bundle $N_X$ of $X$ in $Y$ is simply the quotient bundle, i.e., it f\/its into the short exact sequence
$0\rightarrow TX\rightarrow TY|_X\rightarrow NX\rightarrow0$ of holomorhic vector bundles.
Dually, the co-normal bundle $N^\ast$ f\/its into the short exact sequence
$0\rightarrow N^\ast \rightarrow T^\ast Y|_X\rightarrow T^\ast X\rightarrow0$.

The structure sheaf $\mathcal{O}_Y$ contains a sub-sheaf called the \textit{ideal sheaf} $\mathcal{I}_X$. If $\mathcal{V}$ is an open subset of $Y$, then $\mathcal{I}_X(\mathcal{V})=\{f\in\mathcal{O}_Y(\mathcal{V})\colon f=0$ on $\mathcal{V}\cap X\}$. Notice that $\mathcal{I}_X(\mathcal{V})$ is an ideal in the ring~$\mathcal{O}_Y(\mathcal{V})$ and hence, for each positive integer $i$ there is the sheaf $\mathcal{I}^i_X$ whose space of sections over $\mathcal{V}$ is $(\mathcal{I}_X(\mathcal{V}))^i$. Then there are short exact sequences of sheaves
\begin{gather*}
0\rightarrow\mathcal{I}_X\rightarrow\mathcal{O}_Y\rightarrow \iota_\ast\mathcal{O}_X\rightarrow0
\end{gather*}
and
\begin{gather*}
0\rightarrow\mathcal{I}^{i+1}_X\rightarrow\mathcal{I}^i_X\rightarrow\iota_\ast \mathcal{O}(S^i N^\ast)\rightarrow0,
\end{gather*}
where $S^iN^\ast$ is the $i$-th symmetric power of $N^\ast$ and we agree that $\mathcal{I}^0_X=\mathcal{O}_Y$. We put $\mathcal{F}^i_X:=\iota^{-1}\mathcal{I}_X^i$. As $\iota^{-1}$ is an exact functor, we get short exact sequences
\begin{gather*}
0\rightarrow\mathcal{F}_{X}\rightarrow\iota^{-1}\mathcal{O}_Y\rightarrow \mathcal{O}_X\rightarrow0\label{ses with pullback ideal sheaf}
\end{gather*}
and
\begin{gather*}
0\rightarrow\mathcal{F}^{i+1}_X\rightarrow\mathcal{F}^i_X\rightarrow\mathcal{O}(S^i N^\ast)\rightarrow0
\end{gather*}
of sheaves on $X$. Here we use that the adjunction morphism $\iota^{-1}\iota_\ast\mathcal{F}\rightarrow\mathcal{F}$ is an isomorphism when $\mathcal{F}=\mathcal{O}_X$ or $\mathcal{O}(S^iN^\ast)$.

Put $\mathcal{O}_X^{(i)}:=\mathcal{O}_Y/\mathcal{I}^{i+1}_X$.
The pair $(X,\mathcal{O}^{(i)}_X)$ is called the $i$-\textit{th formal neighborhood of $X$ in} $Y$. Then $\iota^{-1}\mathcal{O}_X^{(i)}\cong\iota^{-1}\mathcal{O}_Y/\mathcal{F}^{(i+1)}$ and since the support of $\mathcal{O}_X^{(i)}$ is contained in $X$, the sheaf $\iota^{-1}\mathcal{O}_X^{(i)}$ contains basically the same information as the sheaf $\mathcal{O}_X^{(i)}$. These sheaves will be crucial in this article.

\begin{Remark}
The stalk of $\mathcal{F}^i_X$ at $x\in X$ is equal to $\{f\in\mathcal{F}_x\colon j^i_xf=0\}$. Hence, if $X=x$ is a~point, the stalk of $\mathcal{F}^i_X$ at $x$ is $\{f\in(\mathcal{O}_Y)_x\colon j^i_xf=0\}$. Since any sheaf over a point is completely determined by its stalk, there is no risk of confusion with the notation set in Remark~\ref{remark usual jets}.
\end{Remark}

\subsection{The Penrose transform}\label{section PT}
Let us f\/irst set notation. Suppose that $\lambda\in\mathfrak{h}^\ast$ is a $\mathfrak{g}$-integral and $\mathfrak{p}$-dominant weight. Then there is (see \cite[Remark 3.1.6]{BE}) an irreducible $\mathrm{P}$-module $\mathbb{V}_\lambda$ with lowest weight $-\lambda$. We denote by $V_\lambda:=\mathrm{G}\times_\mathrm{P}\mathbb{V}_\lambda$ the induced vector bundle and by $\mathcal{O}_\mathfrak{p}(\lambda)$ the associated sheaf of holomorphic sections.

Suppose that $\mathfrak{p}$, $\mathfrak{r}$ are standard parabolic subalgebras. Then $\mathfrak{q}:=\mathfrak{r}\cap\mathfrak{p}$ is also a standard parabolic subalgebra and we denote by $\mathrm{P}$ and $\mathrm{R}$ and $\mathrm{Q}$ the associated parabolic subgroups with Lie algebras $\mathfrak{p}$ and $\mathfrak{r}$ and $\mathfrak{q}$, respectively, as explained in Section \ref{section review}. Then $\mathrm{Q}=\mathrm{R}\cap\mathrm{P}$ and there is a double f\/ibration diagram
\begin{gather*}
\label{double fibration diagram}
\xymatrix{&\ar[dl]^\eta \mathrm{G}/\mathrm{Q}\ar[dr]^\tau&\\
\mathrm{G}/\mathrm{R}&&\mathrm{G}/\mathrm{P},}
\end{gather*}
where $\eta$ and $\tau$ are the canonical projections. The space $\mathrm{G}/\mathrm{R}$ is called the \textit{twistor space~$TS$} and $\mathrm{G}/\mathrm{Q}$ the \textit{correspondence space~$CS$}. Such a diagram is a starting point for the Penrose transform.

Next we need to f\/ix an $\mathfrak{r}$-dominant and integral weight $\lambda\in\mathfrak{h}^\ast$. Then there is a relative BGG sequence $\blacktriangle^\ast(\lambda)$ which is an exact sequence of holomorphic sections of associated vector bundles over $CS$ and linear $\mathrm{G}$-invariant dif\/ferential operators such that $\eta^{-1}\mathcal{O}_\mathfrak{r}(\lambda)$ is the kernel sheaf of the f\/irst operator in the sequence. In other words, there is a long exact sequence of sheaves
\begin{gather*}%\label{relative BGG resolution}
0\rightarrow\eta^{-1}\mathcal{O}_\mathfrak{r}(\lambda)\rightarrow\blacktriangle^\ast(\lambda).
\end{gather*}
The upshot of this is that although the pullback sheaf $\eta^{-1}\mathcal{O}_\mathfrak{r}(\lambda)$ is not a sheaf of holomorphic sections of an associated vector bundle over $CS$, it is naturally a sub-sheaf of $\mathcal{O}_\mathfrak{q}(\lambda)$ which is cut out by an invariant dif\/ferential equation. Moreover, the graph of the relative BGG sequence is \cite[Section~8.7]{BE} completely determined by the $W^\mathfrak{q}_\mathfrak{r}$-orbit of~$\lambda$.

Then we push down the relative BGG sequence by the direct image functor $\tau_\ast$. Computing higher direct images of sheaves in the relative BGG sequence is completely algorithmic and algebraic (see \cite[Section~5.3]{BE}). On the other hand, there is no general algorithm which computes direct images of dif\/ferential operators and it seems that this has to be treated in each case separately. Nevertheless, in this way one obtains a complex of operators on~$\mathrm{G}/\mathrm{P}$.

\section{Lie theory}\label{section LT}

In Section \ref{section LT} we will provide an algebraic background which is needed in the construction of the $k$-Dirac complexes via the Penrose transform. We will work with complex parabolic geometries which are associated to gradings on the simple Lie algebra $\mathfrak{g}=\mathfrak{so}(2m,\mathbb{C})$. Section \ref{section LT} is organized as follows: in Section \ref{section lie algebra} we will set notation and study the gradings on $\mathfrak{g}$. In Section \ref{section the relative Weyl group} we will compute the relative Hasse diagram $W^\mathfrak{q}_\mathfrak{r}$.

\subsection{Lie algebra $\mathfrak{g}$ and parabolic subalgebras}\label{section lie algebra}
Let
$\label{standard basis}
\{e_1,\dots,e_m,e^\ast_{k+1},\dots,e^\ast_m,e_1^\ast,\dots,e_k^\ast\}$
be the standard basis of $\mathbb{C}^{2m}$, $\delta$ be the Kronecker delta and $h$ be the complex bilinear form that satisf\/ies $h(e_i,e^\ast_j)=\delta_{ij}$, $h(e_i,e_j)=h(e^\ast_i,e^\ast_j)=0$ for all $i,j=1,\dots,m$.
A matrix belongs to the associated Lie algebra $\mathfrak{g}:=\mathfrak{so}(h)\cong\mathfrak{so}(2m,\mathbb{C})$ if and only if it is of the form
\begin{gather}\label{lie algebra g}
\left(
\begin{matrix}
A&Z_1&Z_2&W\\
X_1&B&D&-Z_2^T\\
X_2&C&-B^T&-Z_1^T\\
Y&-X_2^T&-X_1^T&-A^T
\end{matrix}
\right),
\end{gather}
 where $A\in M(k,\mathbb{C})$, $B\in M(n,\mathbb{C})$, $C,D\in A(n,\mathbb{C})$, $X_1,X_2,Z_1^T,Z_2^T\in M(n,k,\mathbb{C})$,
$W,Y\in A(k,\mathbb{C})$.

The subspace of diagonal matrices $\mathfrak{h}$ is a Cartan subalgebra of $\mathfrak{g}$. We denote by $\epsilon_i$ the linear form on $\mathfrak{h}$ def\/ined by $\varepsilon_i\colon H=(h_{kl})\mapsto h_{ii}$. Then $\{\epsilon_1,\dots,\epsilon_m\}$ is a basis of $\mathfrak{h}^\ast$ and $\triangle=\{\pm\epsilon_i\pm\epsilon_j\colon  i,j=1,\dots,m\}$. If we choose $\triangle^+=\{\epsilon_i\pm\epsilon_j\colon 1\le i<j\le m\}$
as positive roots, then the simple roots are $\alpha_i:=\epsilon_i-\epsilon_{i+1}$, $i=1,\dots,m-1$ and $\alpha_m:=\epsilon_{m-1}+\epsilon_m$. The associated fundamental weights are $\omega_i=\epsilon_1+\dots+\epsilon_i$, $i=1,\dots,m-2$, $ \omega_{m-1}=\frac{1}{2}(\epsilon_1+\dots+\epsilon_{m-1}-\epsilon_m)$ and $\omega_m=\frac{1}{2}(\epsilon_1+\dots+\epsilon_m)$. The \textit{lowest form} $\rho$ is equal
to $\frac{1}{2}\sum\limits_{\alpha\in\triangle^+}\alpha=\omega_1+\dots+\omega_m=(m-1,\dots,1,0)$.
If $\lambda=\sum\limits_{i=1}^m\lambda_i\epsilon_i$ where $\lambda_i\in\mathbb{C}$, then we will also write $\lambda=(\lambda_1,\dots,\lambda_m)$.
The simple ref\/lection $s_i\in W_\mathfrak{g}$ associated to $\alpha_i$ acts on $\mathfrak{h}^\ast$ by
\begin{gather}\label{simple reflections}
s_i(\lambda)=(\lambda_1,\dots,\lambda_{i-1},\lambda_{i+1},\lambda_{i},\lambda_{i+2},\dots,\lambda_m),\qquad i=1,\dots,m-1
\end{gather}
and
\begin{gather*}
 s_m(\lambda)=(\lambda_1,\dots,\lambda_{m-2},-\lambda_m,-\lambda_{m-1}).
\end{gather*}

We will be interested in the double f\/ibration diagram
\begin{gather}\label{double fibration diagram I}
\begin{matrix}
{\dynkin\root{}\link\dots\link\noroot{}\link\dots\link\root{}\norootupright{}
\rootdownright{}\enddynkin}\\
\swarrow\ \ \ \ \ \ \ \ \ \ \ \ \ \ \ \ \ \ \ \ \ \ \ \searrow\\
\\
{\dynkin \root{}\link\dots\link\root{}\norootupright{}\rootdownright{}\enddynkin}\ \ \ \ \ \ \ \ \ \
{\dynkin \root{}\link\dots\link\noroot{}\link\dots\link\root{}\rootupright{}\rootdownright{}\enddynkin}
\end{matrix}
\end{gather}
where, going from left to right, the sets of simple roots are $\{\alpha_m\}$, $\{\alpha_k,\alpha_m\}$ and $\{\alpha_k\}$ and the associated gradings are
\begin{gather*}
\mathfrak{r}_{-1}\oplus\mathfrak{r}_0\oplus\mathfrak{r}_1, \qquad \mathfrak{q}_{-3}\oplus\mathfrak{q}_{-2}\oplus\dots\oplus\mathfrak{q}_3\qquad \mathrm{and}\qquad
 \mathfrak{g}_{-2}\oplus\dots\oplus\mathfrak{g}_2,
\end{gather*}
respectively.
With respect to the block decomposition from (\ref{lie algebra g}), we have\footnote{Here we mean $\mathfrak{q}_0$ is the subspace of block diagonal matrices, $\mathfrak{q}_1$ is the subspace of those block matrices where only the matrices $Z_1$, $D$ are non-zero, etc.}
\begin{gather*}%\label{grading on G/P}
\hspace*{17mm}\left(
\begin{matrix}
\mathfrak{q}_0&\mathfrak{q}_1&\mathfrak{q}_2&\mathfrak{q}_3\\
\mathfrak{q}_{-1}&\mathfrak{q}_0&\mathfrak{q}_1&\mathfrak{q}_2\\
\mathfrak{q}_{-2}&\mathfrak{q}_{-1}&\mathfrak{q}_0&\mathfrak{q}_1\\
\mathfrak{q}_{-3}&\mathfrak{q}_{-2}&\mathfrak{q}_{-1}&\mathfrak{q}_0
\end{matrix}
\right) \\
\hspace*{25mm} \swarrow\ \ \ \ \ \ \ \ \ \ \ \ \ \searrow \nonumber\\
 \left(
\begin{matrix}
\mathfrak{r}_0&\mathfrak{r}_0&\mathfrak{r}_1&\mathfrak{r}_1\\
\mathfrak{r}_0&\mathfrak{r}_0&\mathfrak{r}_1&\mathfrak{r}_1\\
\mathfrak{r}_{-1}&\mathfrak{r}_{-1}&\mathfrak{r}_0&\mathfrak{r}_0\\
\mathfrak{r}_{-1}&\mathfrak{r}_{-1}&\mathfrak{r}_0&\mathfrak{r}_0\\
\end{matrix}
\right)\ \ \
\left(
\begin{matrix}
\mathfrak{g}_0&\mathfrak{g}_1&\mathfrak{g}_1&\mathfrak{g}_2\\
\mathfrak{g}_{-1}&\mathfrak{g}_0&\mathfrak{g}_0&\mathfrak{g}_1\\
\mathfrak{g}_{-1}&\mathfrak{g}_0&\mathfrak{g}_0&\mathfrak{g}_1\\
\mathfrak{g}_{-2}&\mathfrak{g}_{-1}&\mathfrak{g}_{-1}&\mathfrak{g}_0\\
\end{matrix}
\right).
\end{gather*}
The associated standard parabolic subalgebras are
\begin{gather*}
 \mathfrak{r}=\mathfrak{r}_0\oplus\mathfrak{r}_1, \qquad \mathfrak{q}=\mathfrak{q}_0\oplus\mathfrak{q}_1\oplus\mathfrak{q}_2\oplus\mathfrak{q}_3\qquad \mathrm{and} \qquad \mathfrak{p}=\mathfrak{g}_0\oplus\mathfrak{g}_1\oplus\mathfrak{g}_2,
\end{gather*}
respectively, and we have the following isomorphisms
\begin{gather*}
 \mathfrak{r}_0\cong M(m,\mathbb{C}), \qquad \mathfrak{q}_0\cong M(k,\mathbb{C})\oplus M(n,\mathbb{C})\qquad \mathrm{and} \qquad \mathfrak{g}_0\cong M(k,\mathbb{C})\oplus\mathfrak{so}(2n,\mathbb{C}).
\end{gather*}

We for brevity put
\begin{gather}
 \mathbb{C}^k:=[e_1,\dots,e_k], \qquad \mathbb{C}^{k\ast}:=[e_1^\ast,\dots,e_k^\ast],\nonumber\\
 \mathbb{C}^n:=[e_{k+1},\dots,e_m],\qquad \mathbb{C}^{n\ast}:=[e_{k+1}^\ast,\dots,e_m^\ast],\nonumber\\
  \mathbb{C}^{2n}:=\mathbb{C}^n\oplus\mathbb{C}^{n\ast}\qquad \mathrm{and}\qquad \mathbb{C}^m:=\mathbb{C}^k\oplus\mathbb{C}^n.\label{subspace of C2m}
\end{gather}
Notice that the bilinear form $h$ induces dualities between $\mathbb{C}^k$ and $\mathbb{C}^{k\ast}$ and between $\mathbb{C}^n$ and $\mathbb{C}^{n\ast}$ which justif\/ies the notation, that $\mathbb{C}^m$ is a maximal, totally isotropic and $\mathfrak{r}_0$-invariant subspace, that $\mathbb{C}^k$, $\mathbb{C}^{k\ast}$, $\mathbb{C}^n$ and $\mathbb{C}^{n\ast}$ are $\mathfrak{q}_0$-invariant, that $\mathbb{C}^k$, $\mathbb{C}^{2n}$ and $\mathbb{C}^{k\ast}$ are $\mathfrak{g}_0$-invariant and f\/inally, that $h|_{\mathbb{C}^{2n}}$ is a non-degenerate, symmetric and $\mathfrak{g}_0$-invariant bilinear form. We will for brevity write only $h$ instead of $h|_{\mathbb{C}^{2n}}$ as it will be always clear from the context what is meant.

Let us now consider the associated nilpotent subalgebras
\begin{gather*}
 \mathfrak{r}_-=\mathfrak{r}_{-1},\qquad \mathfrak{q}_-=\mathfrak{q}_{-3}\oplus\mathfrak{q}_{-2}\oplus\mathfrak{q}_{-1}\qquad \mathrm{and} \qquad \mathfrak{g}_-=\mathfrak{g}_{-2}\oplus\mathfrak{g}_{-1}.
\end{gather*}
By the Jacobi identity, the Lie bracket is equivariant with respect to the adjoint action of the corresponding Levi factor
and by the grading property following equation~(\ref{k-gradation}), it is homogeneous of degree zero. Hence, we can consider the Lie bracket in each homogeneity separately.

The f\/irst algebra $\mathfrak{r}_-$ is abelian and so there is nothing to add.

On the other hand, $\mathfrak{q}_{-}$ is 3-graded and, as $\mathfrak{q}_0$-modules, we have $\mathfrak{q}_{-1}\cong\mathbb{E}\oplus\mathbb{F}$, $ \mathfrak{q}_{-2}\cong\mathbb{C}^{k\ast}\otimes\mathbb{C}^{n\ast}$, $\mathfrak{q}_{-3}\cong\Lambda^2\mathbb{C}^{k\ast}$ where we put $\mathbb{E}:=\mathbb{C}^{k\ast}\otimes\mathbb{C}^n$ and $\mathbb{F}:=\Lambda^2\mathbb{C}^{n\ast}$. Using these isomorphisms, the Lie brackets in homogeneity~$-2$ and~$-3$ are the compositions of the canonical projections
\begin{gather}\label{lie bracket in homogeneity -2}
\Lambda^2\mathfrak{q}_{-1}\rightarrow\mathbb{E}\otimes\mathbb{F}=\big(\mathbb{C}^{k\ast}\otimes\mathbb{C}^n\big) \otimes\Lambda^2\mathbb{C}^{n\ast}\rightarrow\mathbb{C}^{k\ast}\otimes\mathbb{C}^{n\ast}=\mathfrak{q}_{-2}
\end{gather}
and
\begin{gather*}
 \mathfrak{q}_{-1}\otimes\mathfrak{q}_{-2}\rightarrow\mathbb{E}\otimes\mathfrak{q}_{-2}= \big(\mathbb{C}^{k\ast}\otimes\mathbb{C}^n\big)\otimes\big(\mathbb{C}^{k\ast}\otimes\mathbb{C}^{n\ast}\big)\rightarrow\Lambda^2\mathbb{C}^{k\ast} =\mathfrak{q}_{-3},
\end{gather*}
respectively. Here we use the canonical pairing $\mathbb{C}^n\otimes\mathbb{C}^{n\ast}\rightarrow\mathbb{C}$. Notice that $\Lambda^2\mathbb{E}\oplus\Lambda^2\mathbb{F}$ is contained in the kernel of~(\ref{lie bracket in homogeneity -2}).

In order to understand the Lie bracket on $\mathfrak{g}_-$, f\/irst notice that there are isomorphisms $\mathfrak{g}_{-1}\cong \mathbb{C}^{k\ast}\otimes\mathbb{C}^{2n}$ and $ \mathfrak{g}_{-2}\cong\Lambda^2\mathbb{C}^{k\ast}\otimes\mathbb{C}$ of irreducible $\mathfrak{g}_0$-modules where $\mathbb{C}$ is the trivial representation of $\mathfrak{so}(2n,\mathbb{C})$. As $\mathfrak{g}_-$ is 2-graded, the Lie bracket is non-zero only in homogeneity~$-2$. It is given by
\begin{gather*}
\Lambda^2\mathfrak{g}_{-1}=\Lambda^2\big(\mathbb{C}^{k\ast}\otimes\mathbb{C}^{2n}\big)\rightarrow\Lambda^2\mathbb{C}^{k\ast}\otimes S^2\mathbb{C}^{2n}\rightarrow\Lambda^2\mathbb{C}^{k\ast}\otimes\mathbb{C}=\mathfrak{g}_{-2},
\end{gather*}
where in the last map we take the trace with respect to $h$.

In the table below we specify when $\lambda=(\lambda_1,\dots,\lambda_m)\in\mathfrak{h}^\ast$ is dominant for each parabolic subalgebra $\mathfrak{p},\mathfrak{q}$ and $\mathfrak{r}$. We put $\mathbb{N}_0:=\mathbb{N}\cup\{0\}$.

%\captionof{table}{Table dominant weights}
\begin{table}[h!]\centering
 \caption{Dominant weights.}\label{table dominant weights}
 \begin{tabular}{|c|c|}
\hline algebra&dominant and integral weights \\
\hline$\mathfrak{p}$&$\lambda_i-\lambda_{i+1}\in\mathbb{N}_0,\ i\ne k,\ 2\lambda_m\in\mathbb{Z},\ \lambda_{m-1}\ge|\lambda_m|$\\
\hline$\mathfrak{r}$&$\lambda_i-\lambda_{i+1}\in\mathbb{N}_0$\\
\hline$\mathfrak{q}$&$\lambda_i-\lambda_{i+1}\in\mathbb{N}_0,i\ne k$\\
\hline
\end{tabular}
 \end{table}

\subsection[Relative Hasse diagram $W_\mathfrak{r}^\mathfrak{q}$]{Relative Hasse diagram $\boldsymbol{W_\mathfrak{r}^\mathfrak{q}}$}\label{section the relative Weyl group}
Let us f\/irst set notation. By a \textit{partition} we will mean an element of $\mathbb{N}^{k,n}_{++}:=\{(a_1,\dots,a_k)\colon a_i\in\mathbb{Z}$, $n\ge a_1\ge a_2\ge\dots\ge a_k\ge0\}$. For two partitions $a=(a_1,\dots,a_k)$ and $a'=(a_1',\dots,a'_k)$ we write $a\le a'$ if $a_i\le a'_i$ for all $i=1,\dots,k$ and $a<a'$ if $a\le a'$ and $a\ne a'$. If $a<a'$ does not hold, then we write $a\nless a'$. We put
\begin{gather}\label{abbreviations}
 |a|=a_1+\dots+a_k,\qquad d(a):=\max\{i\colon a_i\ge i\},\\
 q(a)=\sum_{i=1}^{k}\max\{ a_i-i,0\}\qquad \mathrm{and}\qquad r(a):=d(a)+q(a).\nonumber	
\end{gather}

To the partition $a$ we associate the \textit{Young diagram} (or the \textit{Ferrers diagram}) $\mathrm{Y}$ consisting of~$k$ left-justif\/ied rows with $a_i$-boxes in the $i$-th row. Let $b_i$ be the number of boxes in the $i$-th column of $\mathrm{Y}$. Then we call $b=(b_1,\dots,b_n)\in\mathbb{N}^{n,k}_{++}$ the \textit{partition conjugated to} $a$ and we say that $a$ is \textit{symmetric} if $a_i=b_i$, $i=1,\dots,k$ and $b_{k+1}=\dots=b_n=0$. As we assume $n\ge k$, the set of symmetric partitions in $\mathbb{N}^{k,n}_{++}$ depends only on $k$, and thus, we denote it for simplicity by $S^k$ and put $S^k_j:=\{a\in S^k\colon r(a)=j\}$.

\begin{Example}\label{ex YDI}\quad
\begin{enumerate}\itemsep=0pt
\item[(1)] The empty partition is by def\/inition always symmetric.

\item[(2)] The Young diagram of $a=(4,3,1,0,0)\in\mathbb{N}^{5,6}_{++}$ is
\begin{gather}\label{YD}
\Yctj
\end{gather}
and we f\/ind that $d(a)=2$, $q(a)=4$ and $r(a)=6$. The conjugated partition is $b=(3,2,2,1,0,0)$ with $d(b)=2$, $q(b)=2$ and $r(b)=4$. The Young diagram of $b$ is
\begin{gather*}
\Ytddj.
\end{gather*}
We see that the partition is not symmetric.
\end{enumerate}
\end{Example}

Notice that $d(a)$ and $q(a)$ are equal to the number of boxes in the associated Young diagram that are on and above the main diagonal, respectively and that a partition is symmetric if and only if its Young diagram is symmetric with respect to the ref\/lection along the main diagonal.

We can now continue by investigating the relative Hasse graph $W_\mathfrak{r}^\mathfrak{q}$.
The group $W_\mathfrak{r}$ is generated by $s_1,\dots,s_{m-1}$ while $W_\mathfrak{q}$ is generated by elements $s_1,\dots,s_{k-1},s_{k+1},\dots,s_{m-1}$. By (\ref{simple reflections}), it follows that $W_\mathfrak{r}$ is the permutation group $S_m$ on $\{1,\dots,m\}$ and that $W_\mathfrak{q}\cong S_k\times S_n$ is the stabilizer of $\{1,\dots,k\}$.
Recall from Section~\ref{section review} that in each left coset of $W_\mathfrak{q}$ in $W_\mathfrak{r}$ there is a unique element of minimal length and that we denote the set of all such distinguished representatives by $W_\mathfrak{r}^\mathfrak{q}$. Moreover, the Bruhat order on $W_\mathfrak{g}$ descends to a partial order on $W_\mathfrak{r}$ and on $W_\mathfrak{q}^\mathfrak{r}$. We will now show that there is an isomorphism $\mathbb{N}^{k,n}_{++}\rightarrow W^\mathfrak{q}_\mathfrak{r}$ of partially ordered sets.

Let $a=(a_1,\dots,a_k)\in\mathbb{N}^{k,n}_{++}$ and $\mathrm{Y}$ be the associated Young diagram. We will call the box in the $i$-th row and the $j$-th column of $Y$ an $(i,j)$\textit{-box} and we write into this box the number $\sharp(i,j):=k-i+j$. Notice that $1\le\sharp(i,j)\le m$. Then the set of boxes in $Y$ is indexed by $\Xi_a:=\{(i,j)\colon i=1,\dots,k,\ j=1,\dots,a_i\}$ and we order this set lexicographically, i.e., $(i,j)<(i',j')$ if $i<i'$ or $i=i'$ and $j<j'$. Then
\begin{gather}\label{definition of permutation}
w_a:=s_{\sharp(\Psi(1))}s_{\sharp(\Psi(2))}\dots s_{\sharp(\Psi(|a|))}\in S_m,
\end{gather}
where $\Psi\colon \{1,2,\dots,|a|\}\rightarrow\Xi_a$ is the unique isomorphism of ordered sets.
 Let us now look at an example.

\begin{Example}
The Young diagram from (\ref{YD}) is f\/illed as
\begin{gather*}% \label{YDI}
\begin{matrix}
 k&k+1&k+2&k+3\\
k-1&k&k+1&\\
k-2&&&\\
\end{matrix}
\end{gather*}
and so $w_a:=s_ks_{k+1}s_{k+2}s_{k+3}s_{k-1}s_ks_{k+1}s_{k-2}$.
\end{Example}

We have the following preliminary observation.

\begin{Lemma}
 Let $a=(a_1,\dots,a_k)$ be as above and $b=(b_1,\dots,b_n)$ be the conjugated partition. Then the permutation $w_a\in S_m$ from~\eqref{definition of permutation} satisfies
 \begin{gather}\label{permutation from YD}
w_a(k-i+1+a_i)=k-i+1\qquad \mathrm{and}\qquad w_a(k+j-b_j)=k+j
\end{gather}
for each $i=1,\dots,k$ and $j=1,\dots,n$.
\end{Lemma}
\begin{proof}
Fix $i=1,\dots,k$.
If $a_i>0$, there is $r_i:=k-i+1$ written in the $(i,1)$-box and $r^i:=k-i+a_i=r_i+a_i-1$ in the $(i,a_i)$-box. We put $r^i-1=r_i:=k-i+1$ if $a_i=0$. Similarly, if $j=1,\dots,n$ and $b_j>0$, then there is $c_j:=k+j-1$ in the $(1,j)$-box and $c^j:=k+j-b_j=c_j-b_j+1$ in the $(b_j,j)$-box. We put $c^j-1=c_j:=k+j-1$ if $b_j=0$. Then it is easy to check that $w_a(r^i+1)=r_i$ and $w_a(c^j)=c_j+1$ which completes the proof.
\end{proof}

Notice that the sets $\{k-i+1+a_i\colon i=1,\dots,k\}$ and $\{ k+j-b_j\colon j=1,\dots,n\}$ are disjoint and that their union is $\{1,2,\dots,m\}$. By~(\ref{permutation from YD}), it follows that
\begin{gather}\label{action on lowest form}
w_a\rho=\rho+(-a_k,\dots,-a_1\,|\,b_1,\dots,b_n),
\end{gather}
where $\rho=(m-1,\dots,1,0)$ is the lowest form of $\mathfrak{g}$ and for clarity, we separate the f\/irst $k$ and last $n$ coef\/f\/icients by~$|$. Comparing this with Table~\ref{table dominant weights}, we see that $w_a\rho$ is $\mathfrak{q}$-dominant. As the same holds for any $\mathfrak{r}$-dominant weight, it follows that $w_a\in W^\mathfrak{q}_\mathfrak{r}$.

\begin{Lemma}\label{lemma relative Weyl group}
The map $\mathbb{N}^{k,n}_{++}\rightarrow W^\mathfrak{q}_\mathfrak{r}$, $a\mapsto w_a$ is an isomorphism of partially ordered sets.
\end{Lemma}
\begin{proof}
The map $a\mapsto w_a$ is by (\ref{action on lowest form}) clearly injective. To show surjectivity, f\/ix
$w\in W^\mathfrak{q}_\mathfrak{r}$. Then the sequence $c_1,\dots,c_k$ where $c_i=w^{-1}(i)$, $i=1,\dots,k$ is increasing. By \cite[Proposition~3.2.16]{CS}, the map $w\in W^\mathfrak{q}_\mathfrak{r}\mapsto w^{-1}\omega_k$ is injective. It follows that $w^{-1}\omega_k$ is uniquely determined by the sequence $c_1,\dots,c_k$. Then $a:=(a_1,\dots,a_k)\in\mathbb{N}^{k,n}_{++}$ where $a_i:=c_{k-i+1}+i-k-1$, $i=1,\dots,k$ and from~(\ref{permutation from YD}), it follows that $w_a^{-1}(k-i+1)=k-i+1+a_i=c_{k-i+1}$, $i=1,\dots,k$. This shows that $w_a^{-1}\omega_k=w^{-1}\omega_k$ and thus, $w=w_a$. Now it remains to show that the map is compatible with the orders.

Assume that $a=(a_1,\dots,a_k)$, $a'=(a'_1,\dots,a'_k)\in\mathbb{N}^{k,n}_{++}$ satisfy $|a'|=|a|+1$ and $a<a'$. Then there is a unique integer $i\le k$ such that $a'_i=a_i+1$ and so $w_{a'}=w_as_{k-i+a'_i}$. By~(\ref{permutation from YD}), we have that $w_a\alpha_{k-i+a'_i}>0$ and thus by \cite[Proposition~3.2.16]{CS}, there is an arrow $w_a\rightarrow w_{a'}$ in $W^\mathfrak{q}_\mathfrak{r}$.

On the other hand, suppose that $a''=(a_1'',\dots,a_k'')\in\mathbb{N}^{k,n}_{++}$ satisf\/ies $a'\not <a''$. In order to complete the proof, it is enough to show that there is no arrow $w_{a'}\rightarrow w_{a''}$. By assumptions, there is $j$ such that $a'_1\le a_1'',\dots,a_{j-1}'\le a''_{j-1}$ and $a'_j>a_j''$. Without loss of generality we may assume that $i=j$. Then $w_{a'}^{-1}(w_{a}\alpha_{k+a_i-i})=s_{k+a_i-i}\alpha_{k+a_i-i}<0$. On the other hand by~(\ref{permutation from YD}), it follows that
$w_{a''}^{-1}(w_a\alpha_{k+a_i-i})>0$. We proved that $\Phi_{w_a'}\not\subset\Phi_{w_{a''}}$ and thus by \cite[Proposition~3.2.17]{CS}, there cannot be any arrow $w_{a'}\rightarrow w_{a''}$.
\end{proof}

We will later need the following two observations. A permutation $w\in S_m$ is $k$-\textit{balanced}, if the following is true: if $w(k-i)>k$ for some $i=0,\dots,k-1$, then $ w(k+i+1)\le k$.

\begin{Lemma}\label{lemma symmetric YD}
The permutation $w_a$ associated to $a\in\mathbb{N}^{k,n}_{++}$ is $k$-balanced if and only if $a\in S^k$.
\end{Lemma}
\begin{proof}
Let $b=(b_1,\dots,b_n)$ be the partition conjugated to $a=(a_1,\dots,a_k)$. First notice that if $w_a(k-i)=k+j>k$, then by~(\ref{permutation from YD}) we have $i=b_j-j$.

If $a\in S^k$, then $w_a(k+i+1)=w_a(k-j+b_j+1)=w_a(k-j+a_j+1)=k-j+1\le k$ and so $a$ is $k$-balanced.

If $a\not\in S^k$, then there is $j$ such that $a_1=b_1,\dots,a_{j-1}= b_{j-1}$ and $a_j\ne b_j$. It follows that $b_j\ge j$ and so $i:=b_j-j\ge0$. Then $w_a(k-i)=w_a(k-b_j+j)=k+j>k$. If $a_j>b_j$, then $w_a(k+i+1)=k+b_j+1>k$. If $a_j<b_j$, then $w_a(k+i+1)=k+b_j>k$.
\end{proof}

Recall from \cite{BE} that given $w\in S_m$, there exists a minimal integer $\ell(w)$, called the \textit{length} of~$w$, such that~$w$ can be expressed as a product of $\ell(w)$ simple ref\/lections $s_1,\dots,s_m$. It is well known that $\ell(w)$ is equal to the number of pairs $1\le i<j\le m$ such that $w(i)>w(j)$.

\begin{Lemma}\label{lemma length of permutation}
Let $w_a\in S_m$. Then $\ell(w_a)=|a|$.
\end{Lemma}
\begin{proof}
By the def\/inition of $w_a$, it follows that $\ell(w_a)\le |a|$. On the other hand, if $a<a'$, then $w_a\rightarrow w_{a'}$ and thus also $\ell(w_a)<\ell(w_{a'})$. By induction on $|a|$, we have that $\ell(w_a)\ge|a|$.
\end{proof}

\section{Geometric structures attached to (\ref{double fibration diagram I})}\label{section geometry behind PT}
In Section \ref{section geometry behind PT} we will consider dif\/ferent geometric structures associated to~(\ref{double fibration diagram I}). Namely, we will consider in Section~\ref{section complex grassmannian} the associated homogeneous spaces, in Section~\ref{section filtration of TM and TCS} the f\/iltrations of tangent bundles of these parabolic geometries and in Section \ref{section projections} the projections $\eta$ and $\tau$.

\subsection{Homogeneous spaces}\label{section complex grassmannian}

A connected and simply connected Lie group $\mathrm{G}$ with Lie algebra $\mathfrak{g}$ is isomorphic to $\mathrm{Spin}(2m,\mathbb{C})$.
Let $\mathrm{R}$, $\mathrm{Q}$ and $\mathrm{P}$ be the parabolic subgroups of~$\mathrm{G}$ with Lie algebras $\mathfrak{r}$, $\mathfrak{q}$ and~$\mathfrak{p}$ that are associated to $\{\alpha_m\}$, $\{\alpha_k,\alpha_m\}$ and~$\{\alpha_k\}$, respectively, as explained in Section~\ref{section review}. We for brevity put $TS:=\mathrm{G}/\mathrm{R}$, $CS:=\mathrm{G}/\mathrm{Q}$ and $M:=\mathrm{G}/\mathrm{P}$. Recall from Section~\ref{section PT} that we call $TS$ the twistor space and~$CS$ the correspondence space.

\textit{The twistor space $TS$}. Let us f\/irst recall (see \cite[Section 6]{GW}) some well known facts about spinors.
Recall from (\ref{subspace of C2m}) that $\mathbb{W}:=\mathbb{C}^{m}$ is a maximal totally isotropic subspace of $\mathbb{C}^{2m}$. We can (via $h$) identify the dual space $\mathbb{W}^\ast$ with the subspace $[e_1^\ast,\dots,e_m^\ast]$. Put $\mathbb{S}:=\bigoplus_{i=0}^m\Lambda^i\mathbb{W}^\ast$. There is a canonical linear map $\mathbb{C}^{2m}\rightarrow\operatorname{End}(\mathbb{S})$ which is determined by $w\cdot\psi=i_w\psi$ and $w^\ast\cdot\psi=w^\ast\wedge\psi$ where $w\in\mathbb{W}$, $w^\ast\in\mathbb{W}^\ast$, $\psi\in\mathbb{S}$ and $i_w$ stands for the contraction by $w$. If $\psi\in\mathbb{S}$, then we put $T_\psi:=\{v\in\mathbb{C}^{2m}\colon v\cdot\psi=0\}$. If $\psi\ne0$, then~$T_\psi$ is a totally isotropic subspace and we call $\psi$ a~\textit{pure spinor} if $\dim T_\psi=m$ (which is equivalent to saying that $T_\psi$ is a~maximal totally isotropic subspace).

The standard linear isomorphism $\Lambda^2\mathbb{C}^{2m}\cong\mathfrak{g}$ gives an injective linear map $\mathfrak{g}\rightarrow\operatorname{End}(\mathbb{S})$. It is straightforward to verify that the map is a homomorphism of Lie algebras where the commutator in the associative algebra $\operatorname{End}(\mathbb{S})$ is the standard one. Hence, $\mathfrak{g}$ is a Lie subalgebra of $\operatorname{End}(\mathbb{S})$ and it turns out that~$\mathbb{S}$ is no longer irreducible under $\mathfrak{g}$ but it decomposes as $\mathbb{S}_+\oplus\mathbb{S}_-$ where $\mathbb{S}_+:=\bigoplus_{i=0}^m\Lambda^{2i}\mathbb{W}^\ast$ and $\mathbb{S}_-:=\bigoplus_{i=0}^m\Lambda^{2i+1}\mathbb{W}^\ast$. Then $\mathbb{S}_+$ and $\mathbb{S}_-$ are irreducible non-isomorphic complex spinor representations of $\mathfrak{g}$ with highest weights $\omega_m$ and $\omega_{m-1}$, respectively. It is well known that any pure spinor belongs to $\mathbb{S}_+$ or to $\mathbb{S}_-$ (which explains why the Grassmannian of maximal totally isotropic subspaces in~$\mathbb{C}^{2m}$ has two connected components).

Now we can easily describe the twistor space. The spinor $1\in\mathbb{S}_+$ is annihilated by all positive roots in $\mathfrak{g}$ and hence, it is a highest weight vector. Recall from Section~\ref{section review} that the line spanned by~1 is invariant under~$\mathrm{R}$ and since $T_1=\mathbb{W}$, we f\/ind that~$\mathrm{R}$ is the stabilizer of $\mathbb{W}$ inside $\mathrm{G}$. As~$\mathrm{G}$ is connected, we conclude that~$TS$ is the connected component of~$\mathbb{W}$ in the Grassmannian of maximal totally isotropic subspaces in~$\mathbb{C}^{2m}$.

\textit{The isotropic Grassmannian $M$.}
An irreducible $\mathfrak{g}$-module with highest weight $\omega_k$ is isomorphic to $\Lambda^k\mathbb{C}^{2m}$. Then $e_1\wedge e_2\wedge\dots\wedge e_k$ is clearly a highest weight vector and the corresponding point in $\mathbb P(\Lambda^k\mathbb{C}^{2m})$ can be viewed as the totally isotropic subspace $x_0:=\mathbb{C}^k$. We see that $M$ is the Grassmannian of totally isotropic $k$-dimensional subspaces in $\mathbb{C}^{2m}$. We denote by $\textbf{p}\colon \mathrm{G}\rightarrow M$ the canonical projection.

\textit{The correspondence space $CS$}.
The correspondence space $CS$ is the generalized f\/lag manifold of nested subspaces $\{(z,x)\colon z\in TS,\ x\in M,\ x\subset z\}$ and~$\mathrm{Q}$ is the stabilizer of $(\mathbb{W},x_0)$. Let $\textbf{q}\colon \mathrm{G}\rightarrow CS$ be the canonical projection.

\subsection[Filtrations of the tangent bundles of $M$ and $CS$]{Filtrations of the tangent bundles of $\boldsymbol{M}$ and $\boldsymbol{CS}$}\label{section filtration of TM and TCS} Recall from
 Section \ref{section review} that the $|2|$-grading $\mathfrak{g}=\mathfrak{g}_{-2}\oplus\mathfrak{g}_{-1}\oplus\dots\oplus\mathfrak{g}_2$ associated to $\{\alpha_k\}$ determines a 2-step f\/iltration $\{0\}=F_{0}^M\subset F_{-1}^M\subset F_{-2}^M=TM$ of the tangent bundle of $M$ where $\{0\}$ is the zero section. We put $G_{i}^M:=F_{i}^M/F_{i+1}^M$, $i=-2,-1$ so that the associated graded bundle $gr(TM)=G_{-2}^M\oplus G_{-1}^M$ is a locally trivial bundle of graded nilpotent Lie algebras with typical f\/iber $\mathfrak{g}_-$. Dually, there is a f\/iltration $T^\ast M=F_1^M\supset F_2^M\supset F_3^M=\{0\}$ where $F_i^M\cong (F^M_{-i+1})^\perp$. We put $G_i^M:=F_i^M/F_{i+1}^M$ so that $G_i^M\cong(G_{-i}^M)^\ast$. There are linear isomorphisms
\begin{gather}\label{linear isomorphisms over x0}
 \mathfrak{g}_{i}\cong \big(G^M_i\big)_{x_0},\qquad i=-2,-1,1,2.
\end{gather}

Recall from Section \ref{section wdo} that $\mathfrak{gr}^r_{x_0}$ denotes the vector space of weighted $r$-jets of germs of holomorhic functions at $x_0$ whose weighted $(r-1)$-jet vanishes. Then the isomorphisms from~(\ref{isom of graded jets over origin}) are	
\begin{gather*}
\mathfrak{gr}^{1}_{x_0}\cong \mathfrak{g}_1,\qquad \mathfrak{gr}^{2}_{x_0}\cong S^2\mathfrak{g}_1\oplus\mathfrak{g}_2,\qquad \mathfrak{gr}^{3}_{x_0}\cong S^3\mathfrak{g}_1\oplus\mathfrak{g}_1\otimes\mathfrak{g}_2, \qquad \dots
\end{gather*}
for small $r$ and in general
\begin{gather}\label{weighted jets at x_0}
\mathfrak{gr}^{r}_{x_0}\cong\bigoplus_{i+2j=r}S^i\mathfrak{g}_1\otimes S^j\mathfrak{g}_2.
\end{gather}

The $|3|$-grading $\mathfrak{g}=\bigoplus_{i=-3}^3\mathfrak{q}_i$ determined by $\{\alpha_k,\alpha_m\}$ induces a 3-step f\/iltration $TCS= F_{-3}^{CS}\supset F_{-2}^{CS}\supset F_{-1}^{CS}\supset F_0^{CS}=\{0\}$. We put $G_i^{CS}:= F_i^{CS}/F^{CS}_{i+1}$ so that $gr(TCS)=\bigoplus_{i=-3}^{-1}G_i^{CS}$ is a locally trivial vector bundle of graded nilpotent Lie algebras with typical f\/iber $\mathfrak{q}_-$.
Dually, we get a f\/iltration $T^\ast CS=F_1^{CS}\supset F_2^{CS}\supset F_3^{CS}\supset F_4^{CS}=\{0\}$ where $F_i^{CS}:=(F^{CS}_{-i+1})^\perp$. The associated graded vector bundle is $gr(T^\ast CS)=\bigoplus_{i=1}^3G_i^{CS}$ where we put $G_i^{CS}:=F_{i}^{CS}/F^{CS}_{i+1}$. Then as above, $G_i^{CS}\cong (G_{-i}^{CS})^\ast$.

The $\mathrm{Q}$-invariant subspaces $\mathbb{E}\oplus\mathfrak{q}$ and $\mathbb{F}\oplus\mathfrak{q}$ give a f\/iner f\/iltration of the tangent bundle, namely $F_{-1}^{CS}=E^{CS}\oplus F^{CS}$. Since the Lie bracket $\Lambda^2\mathfrak{q}_{-1}\rightarrow\mathfrak{q}_{-2}$ vanishes on $\Lambda^2\mathbb{E}\oplus\Lambda^2\mathbb{F}$, it follows that $E^{CS}$ and $F^{CS}$ are integrable distributions. This can be deduced also from the short exact sequences
\begin{gather}\label{E and F as kernels of tangent maps}
0\rightarrow E^{CS}\rightarrow TCS\xrightarrow{T\eta}T(TS)\rightarrow0\qquad \mathrm{and}\qquad 0\rightarrow F^{CS}\rightarrow TCS\xrightarrow{T\tau}TM\rightarrow0,
\end{gather}
i.e., $E^{CS}=\ker(T\eta)$ and $F^{CS}=\ker(T\tau)$. Notice that $(T\tau)^{-1}(F_{-1}^M)=F^{CS}_{-2}$.

\subsection[Projections $\tau$ and $\eta$]{Projections $\boldsymbol{\tau}$ and $\boldsymbol{\eta}$}\label{section projections}

Recall from (\ref{subspace of C2m}) that $\mathbb{C}^{2n}:=[e_{k+1},\dots,e_m,e_{k+1}^\ast,\dots,e_m^\ast]$ and $\mathbb{C}^n:=[e_{k+1},\dots,e_m]$, i.e., we view~$\mathbb{C}^{2n}$ and~$\mathbb{C}^n$ as subspaces of~$\mathbb{C}^{2m}$. On~$\mathbb{C}^{2n}$ we consider the non-degenerate bilinear form~$h|_{\mathbb{C}^{2n}}$ which we for brevity denote by~$h$. Then $\mathbb{C}^n$ is a maximal totally isotropic subspace of~$\mathbb{C}^{2n}$.

The f\/ibers of $\tau$ and $\eta$ are homogeneous spaces of parabolic geometries which (see~\cite{BE}) can be recovered from the Dynkin diagrams given in~(\ref{double fibration diagram I}).

\begin{Lemma}\label{lemma fibers}\quad
\begin{enumerate}\itemsep=0pt
 \item[$(a)$] The fibers of $\tau$ are biholomorphic to the Grassmannian of $k$-dimensional subspaces in $\mathbb{C}^{n+k}$.
 \item [$(b)$] The fibers of $\eta$ are biholomorphic to the connected component $\mathrm{Gr}^+_h(n,n)$ of $\mathbb{C}^n$ in the Grassmannian of maximal totally isotropic subspaces in~$\mathbb{C}^{2n}$.
\end{enumerate}
\end{Lemma}

\begin{proof}As the f\/ibers over distinct points are biholomorphic, it suf\/f\/ices to look at the f\/ibers of~$\eta$ and~$\tau$ over $\mathbb{W}$ and~$x_0$, respectively.

(a) By def\/inition, $\eta^{-1}(\mathbb{W})$ is the set of $k$-dimensional totally isotropic subspaces in $\mathbb{W}$. As $\mathbb{W}$ is already totally isotropic, the f\/irst claim follows.

(b) Notice that $x_0^\bot=x_0\oplus \mathbb{C}^{2n}$. Then it is easy to see that $y\in\mathrm{Gr}^+_h(n,n)\mapsto(x_0\oplus y,x_0)\in\tau^{-1}(x_0)$ is a~biholomorphism.
\end{proof}

We will use the following notation. Assume that $X\in M(2m,k,\mathbb{C})$ and $Y\in M(2m,n,\mathbb{C})$ have maximal rank. Then we denote by $[X]$ the $k$-dimensional subspace of $\mathbb{C}^{2m}$ that is spanned by the columns of the matrix and by $[X|Y]$ the f\/lag of nested subspaces $[X]\subset[X]\oplus[Y]$.

It is straightforward to verify that
\begin{gather}\label{affine subset of M}
(\textbf{p}\circ\exp)\colon \ \mathfrak{g}_- \rightarrow M,\\
(\textbf{p}\circ\exp)\left(
\begin{matrix}
0&0&0&0\\
X_1&0&0&0\\
X_2&0&0&0\\
Y&-X_2^T&-X_1^T&0\\
\end{matrix}
\right) =
\left[
\begin{matrix}
1_k\\
X_1\\
X_2\\
Y-\frac{1}{2}(X_1^TX_2+X^T_2X_1)\\
\end{matrix}
\right].\nonumber
\end{gather}
We see that $\mathcal{X}:=\textbf{p}\circ\exp(\mathfrak{g}_-)$ is an open, dense and af\/f\/ine subset of $M$ and that any $(z,x)\in\tau^{-1}(\mathcal{X})$ can be represented by
\begin{gather}\label{coordinates on tau^{-1}(U)}
\left[
\begin{array}{c|c}
1_k&0\\
X_1&A\\
X_2&B\\
Y-\frac{1}{2}(X_1^TX_2+X_2^TX_1)&C\\
\end{array}
\right],
\end{gather}
where $A,B\in M(n,\mathbb{C})$, $C\in M(k,n,\mathbb{C})$ are such that
$\left[
\begin{matrix}
A\\
B\\
\end{matrix}
\right]\in\mathrm{Gr}^+_{h}(n,n)$ and $C=-(X_1^TB+X_2^TA)$. We immediately get the following observation.

\begin{Lemma}\label{lemma set tau^{-1}(U)}
The set $\tau^{-1}(\mathcal{X})$ is biholomorphic to $\mathcal{X}\times\tau^{-1}(x_0)$. The restriction of $\tau$ to this set is then the projection onto the first factor.
\end{Lemma}

The set $\tau^{-1}(\mathcal{X})$ is not af\/f\/ine as dif\/ferent choices of $A$ and $B$ might lead to the same element in $\tau^{-1}(\mathcal{X})$. Let $\mathcal{Y}$ be the subset of $\tau^{-1}(\mathcal{X})$ of those nested f\/lags $x\subset z$ which can be represented by a matrix as above with $A$ regular. In that case we may assume $A=1_n$ which uniquely pins down~$B$. It is straightforward to f\/ind that $B=-B^T$ and conversely, any skew-symmetric $n\times n$ matrix determines a totally isotropic $n$-dimensional subspace in~$\mathbb{C}^{2n}$. We see that $\mathcal{Y}$ is an open and af\/f\/ine set which is biholomorphic to $\mathfrak{g}_-\times A(n,\mathbb{C})$. In order to write down also $\eta$ as a canonical projection $\mathbb{C}^{{m\choose2}+nk}\rightarrow\mathbb{C}^{m\choose2}$, it will be convenient to choose a dif\/ferent coordinate system on $\mathcal{Y}$.

\begin{Lemma}\label{lemma projection eta}
Let $\mathcal{Y}$ be as above and put $\mathcal{Z}:=\eta(\mathcal{Y})$. Then $\mathcal{Y}$ and $\mathcal{Z}$ are open affine sets and there is a commutative diagram of holomorphic maps
\begin{gather}\label{other coordinates on CS}\begin{split} &
 \xymatrix{\mathcal{Y}\ar[d]^{\eta|_\mathcal{Y}}\ar[r]&A(m,\mathbb{C})\times M(n,k,\mathbb{C})\ar[d]^{pr_1}\\
\mathcal{Z}\ar[r]&A(m,\mathbb{C}),}\end{split}
\end{gather}
 where $pr_1$ is the canonical projection and the horizontal arrows are biholomorphisms.
\end{Lemma}

\begin{proof}
Let $(z,x)$ be the nested f\/lag corresponding to (\ref{coordinates on tau^{-1}(U)}) where $A=1_n$ so that $B=-B^T$. Put for brevity $Y':=Y-\frac{1}{2}(X_1^TX_2+X_2^TX_1)$ and $C:=-X_2^T-X_1^TB$.
The map in the f\/irst row in (\ref{other coordinates on CS}) is $(z,x)\mapsto(W,Z)$ where
\begin{gather*}\label{coordinates on W}
W=\left(
\begin{matrix}
 W_1&W_2\\
 W_0&-W_1^T
\end{matrix}
\right)
\end{gather*}
and
$ Z:=X_1$, $W_1:=X_2-BX_1$, $W_0:=Y'-CX_1$, $W_2:=B$.
Then $W=-W^T$ and the map $\mathcal{Y}\rightarrow A(m,\mathbb{C})\times M(n,k,\mathbb{C})$ is clearly a biholomorphism.
In order to have a geometric interpretation of the map, consider the following. Using Gaussian elimination on the columns of the matrix~(\ref{coordinates on tau^{-1}(U)}), we can eliminate the $X_1$-block and get a new matrix	
\begin{gather*}
 \left(
\begin{matrix}
1_k&0\\
0&1_n\\
X_2-BX_1&B\\
Y'-CX_1&C
\end{matrix}
\right)=\left(
\begin{matrix}
1_m\\
W
\end{matrix}
\right).
\end{gather*}
The columns of the matrix span the same totally isotropic subspace $z$ as the original matrix. Moreover, it is clear that $z$ admits a unique basis of this form. From this we easily see that $\mathcal{Z}$ is indeed an open af\/f\/ine subset of~$TS$ which is biholomorphic to~$A(m,\mathbb{C})$. In these coordinate systems, the restriction of $\eta$ is the projection onto the f\/irst factor.
\end{proof}

\section[The Penrose transform for the $k$-Dirac complexes]{The Penrose transform for the $\boldsymbol{k}$-Dirac complexes}\label{section computing PT}
In Section \ref{section computing PT} we will consider the relative BGG sequence associated to a particular $\mathfrak{r}$-dominant and integral weight as explained in Section \ref{section PT}. More explicitly, we will def\/ine in Section \ref{section relative double complex} for each $p,q\ge0$ a sheaf of relative $(p,q)$-forms and we get a Dolbeault-like double complex. Then we will show (see Section~\ref{section relative de Rham}) that this double complex contains a relative holomorphic de Rham complex. Then in Section~\ref{section relative twisted double complex} we will twist each sheaf of relative $(p,q)$-forms as well as the holomorhic de Rham complex by a certain pullback sheaf. Using some elementary representation theory, we will turn (see Section~\ref{section relative BGG sequence}) the twisted relative de Rham complex into the relative BGG sequence. In Section~\ref{section direct images} we will compute direct images of sheaves in the relative BGG sequence.

We will use the following notation. We denote by $\mathcal{O}_\mathfrak{q}$ and $\mathcal{E}^{p,q}_\mathfrak{q}$ the structure sheaf and the sheaf of smooth $(p,q)$-forms, respectively, over~$CS$. We denote the corresponding sheaves over~$TS$ by the subscript~$\mathfrak{r}$. If $W$ is a~holomorphic vector bundle over $CS$, then we denote by $\mathcal{O}_\mathfrak{q}(W)$ the sheaf of holomorphic sections of~$W$ and by $\mathcal{E}^{p,q}_\mathfrak{q}(W)$ the sheaf of smooth $(p,q)$-forms with values in $W$. We for brevity put $\mathcal{E}_\ast:=\mathcal{E}^{0,0}_\ast$ and $\mathcal{E}^{p,q}_\ast(\mathcal{U}):=\Gamma(\mathcal{U},\mathcal{E}_\ast^{p,q})$ where $\mathcal{U}$ is an open set and $\ast=\mathfrak{q}$ or $\mathfrak{r}$. Moreover, we put $\eta^\ast\mathcal{E}^{p,q}_\mathfrak{r}:= \mathcal{E}_\mathfrak{q}\otimes_{\eta^{-1}\mathcal{E}_\mathfrak{r}}\mathcal{E}^{p,q}_\mathfrak{r}$ where we use that $\eta^{-1}\mathcal{E}_\mathfrak{r}$ is naturally a sub-sheaf of~$\mathcal{E}_\mathfrak{q}$.

\subsection{Double complex of relative forms}\label{section relative double complex}
Recall from Lemma \ref{lemma projection eta} that $\mathcal{Y}$ is biholomorphic to $A(m,\mathbb{C})\times M(n,k,\mathbb{C})$, that $\eta(\mathcal{Y})=\mathcal{Z}$ is biholomorphic to $A(m,\mathbb{C})$ and that the canonical map $\mathcal{Y}\rightarrow\mathcal{Z}$ is the projection onto the f\/irst factor. In this way we can use matrix coef\/f\/icients on $M(n,k,\mathbb{C})$ as coordinates on the f\/ibers of $\eta$. We will write $Z=(z_{\alpha i})\in M(n,k,\mathbb{C})$ and
if $I=((\alpha_1,i_1),\dots,(\alpha_p,i_p))$ is a multi-index where $(\alpha_j,i_j)\in I(n,k):=\{(\alpha,i)\colon \alpha=1,\dots,n,\ i=1,\dots,k\}$, $j=1,\dots,p$, then we put $dz_I:=d z_{(\alpha_1,i_1)}\wedge\dots\wedge d z_{(\alpha_p,i_p)}$ and $|I|=p$.

We call
\begin{gather*}
 \mathcal{E}^{p+1,q}_\eta:=\mathcal{E}^{p+1,q}_\mathfrak{q}/(\eta^\ast\mathcal{E}^{1,0}_\mathfrak{r}\wedge\mathcal{E}_\mathfrak{q}^{p,q}),\qquad p,q\ge0
\end{gather*}
the \textit{sheaf of relative $(p+1,q)$-forms}. By the $\mathrm{G}$-action, it is clearly enough to understand the space of sections of this sheaf over the open set $\mathcal{Y}$ from Section~\ref{section projections}.
Given $\omega\in\mathcal{E}^{p,q}_\eta(\mathcal{Y})$, it is easy to see that there is a unique $(p,q)$-form cohomologous to $\omega$ which can be written in the form
\begin{gather}\label{relative form in coordinates}
\sideset{}{'}\sum_{|I|=p} dz_I\wedge \omega_I,
\end{gather}
where $\Sigma'$ denote that the summation is performed only over strictly increasing multi-indeces\footnote{We order the set $I(n,k)$ lexicographically, i.e., $(\alpha,i)<(\alpha',i')$ if $\alpha<\alpha'$ or $\alpha=\alpha'$ and $i<i'$.} and each $\omega_I\in\mathcal{E}^{0,q}_\mathfrak{q}(\mathcal{Y})$.

As $\eta$ is holomorphic, $\partial$ and $\bar\partial$ commute with the pullback map $\eta^\ast$. We see that $\partial(\eta^\ast\mathcal{E}^{1,0}_\mathfrak{r}\wedge\mathcal{E}_\mathfrak{q}^{p,q})\subset \eta^\ast\mathcal{E}^{1,0}_\mathfrak{r}\wedge\mathcal{E}_\mathfrak{q}^{p+1,q}$ and $\bar\partial(\eta^\ast\mathcal{E}^{1,0}_\mathfrak{r}\wedge\mathcal{E}_\mathfrak{q}^{p,q})\subset \eta^\ast\mathcal{E}^{1,0}_\mathfrak{r}\wedge\mathcal{E}_\mathfrak{q}^{p,q+1}$ and thus, $\partial$ and $\bar\partial$ descend to dif\/ferential operators
\begin{gather*}%\label{operators in double complex I}
\partial_\eta\colon \ \mathcal{E}^{p,q}_\eta\rightarrow\mathcal{E}^{p+1,q}_\eta\qquad \mathrm{and}\qquad \bar\partial\colon \ \mathcal{E}^{p,q}_\eta\rightarrow\mathcal{E}^{p,q+1}_\eta,
\end{gather*}
respectively. From the def\/initions it easily follows that:
\begin{gather}\label{properties of relative dbar}
 \partial_\eta(\omega\wedge\omega')=(\partial_\eta\omega)\wedge\omega'+(-1)^{p+q}\omega\wedge\partial_\eta\omega',\\
 \partial_\eta f=\sum_{(\alpha,i)\in I(n,k)}\frac{\partial f}{\partial z_{\alpha i}}dz_{\alpha i},\nonumber
\end{gather}
where $\omega\in\mathcal{E}^{p,q}_\eta(\mathcal{Y})$, $\omega'\in\mathcal{E}^{p',q'}_\eta(\mathcal{Y})$ and $f\in\mathcal{E}(\mathcal{Y})$. Recall from Section \ref{section filtration of TM and TCS} that $\partial_{z_{\alpha i}}\in\Gamma(E^{1,0}|_{\mathcal{Y}})$ and thus, $\partial_\eta f(x)$, $ x\in\mathcal{Y}$ depends only on the f\/irst weighted jet~$\mathfrak{j}^1_xf$ of $f$ at~$x$ (see Section~\ref{section wdo}).

Recall from (\ref{E and F as kernels of tangent maps}) that the distribution $E^{CS}$ is equal to $\ker(T\eta)$.

\begin{Proposition}\label{thm double complex relative forms}\quad
\begin{enumerate}\itemsep=0pt
 \item[$(i)$] The sheaf $\mathcal{E}^{p,q}_\eta$ is naturally isomorphic to the sheaf $\mathcal{E}^{0,q}_\mathfrak{q}(\Lambda^{p}E^{CS\ast})$ of smooth $(0,q)$-forms with values in the vector bundle $\Lambda^p E^{CS\ast}$ and $(\mathcal{E}^{p,\ast}_\eta,\bar\partial)$ is a resolution of $\mathcal{O}_\mathfrak{q}(\Lambda^p E^{CS\ast})$ by fine sheaves.
 \item [$(ii)$] $\partial_\eta$ is a linear $\mathrm{G}$-invariant differential operator of weighted order one and the sequence of sheaves $(\mathcal{E}^{\ast,q}_\eta,\partial_\eta)$, $q\ge0$ is exact. \label{lemma local exactness of relative d}
 \item [$(iii)$]The data $(\mathcal{E}^{p,q}_\eta,(-1)^p\bar\partial,\partial_\eta)$ define a double complex of fine sheaves with exact rows and columns.
\end{enumerate}
\end{Proposition}

\begin{proof} $(i)$ By def\/inition, the sequence of vector bundles $0\rightarrow E^{CS\bot}\rightarrow T^\ast CS\rightarrow E^{CS\ast}\rightarrow0$ is short exact. Hence, also the sequence $0\rightarrow E^{CS\bot}\wedge\Lambda^{p}T^\ast CS\rightarrow\Lambda^{p+1} T^\ast CS\rightarrow \Lambda^{p+1}E^{CS\ast}\rightarrow0$, $p\ge1$ is short exact. In view of the isomorphism $\mathcal{E}^{p+1,q}_\mathfrak{q}\cong\mathcal{E}^{0,q}_\mathfrak{q}(\Lambda^{p+1}T^\ast CS)$, it is enough to show that $\eta^\ast\mathcal{E}^{1,0}_\mathfrak{r}\wedge\mathcal{E}_\mathfrak{q}^{p,q}$ is isomorphic to $\mathcal{E}^{0,q}_\mathfrak{q}(E^{CS\bot}\wedge\Lambda^{p}T^\ast CS)\cong\mathcal{E}_\mathfrak{q}(E^{CS\bot})\wedge\mathcal{E}^{0,q}_\mathfrak{q}(\Lambda^{p}T^\ast CS)$. Now $\eta^\ast\mathcal{E}^{1,0}_\mathfrak{r}$ is a sub-sheaf of $\mathcal{E}_\mathfrak{q}^{1,0}=\mathcal{E}_\mathfrak{q}(T^\ast CS)$ and since $\ker(T\eta)=E^{CS}$, it is contained in~$\mathcal{E}_\mathfrak{q}(E^{CS\bot})$. Using that $\ker(T\eta)=E^{CS}$ again, it is easy to see that the map $\eta^\ast\mathcal{E}^{1,0}_\mathfrak{r}\rightarrow\mathcal{E}_\mathfrak{q}(E^{CS\bot})$ induces an isomorphism of stalks at any point. Hence, $\eta^\ast\mathcal{E}^{1,0}_\mathfrak{r}\cong\mathcal{E}_\mathfrak{q}(E^{CS\bot})$ and the proof of the f\/irst claim is complete. The second claim is clear.

(ii) It is clear that $\partial_\eta$ is $\mathbb{C}$-linear. It is $\mathrm{G}$-invariant as $\partial$ commutes with the pullback of any holomorphic map and since $\eta$ is $\mathrm{G}$-equivariant. As we already observed above that $\partial_\eta f(x)$ depends only on $\mathfrak{j}^1_xf$ when $x\in\mathcal{Y}$, the $\mathrm{G}$-invariance of $\partial_\eta$ shows that the same holds on $CS$ and thus, $\partial_\eta$ is a dif\/ferential operator of weighted order one. It remains to check the exactness of the complex and using the $\mathrm{G}$-invariance, it is enough to do this at~$x\in\mathcal{Y}$.
By Lemma~\ref{lemma projection eta}, $\mathcal{Y}$ is biholomorphic to $\mathbb{C}^\ell$ where $\ell={m\choose2}+nk$. Hence, we can view the standard coordinates $w_1,\dots,w_\ell$ on $\mathbb{C}^\ell$ as coordinates on $\mathcal{Y}$. If $J=(j_1,\dots,j_q)$ where $j_1,\dots,j_q\in\{1,\dots,\ell\}$, then we put $d\bar w_J=d\bar w_{j_1}\wedge\cdots\wedge d\bar w_{j_q}$ and $|J|=q$. Let $\omega=\sideset{}{'}\sum\limits_{|I|=p}dz_I\wedge\omega_I\in\mathcal{E}^{p,q}_\eta(\mathcal{Y})$ be the relative form as in~(\ref{relative form in coordinates}). Then there are unique functions $f_{I,J}\in\mathcal{E}_\mathfrak{q}(\mathcal{Y})$ so that $\omega_I=\sideset{}{'}\sum\limits_{|J|=q}f_{I,J}d\bar w_J$. Assume that $\partial_\eta\omega=0$ on some open neighborhood $\mathcal{U}_x$ of $x$. This is equivalent to saying that for each increasing multi-index $J\colon \partial_\eta\sigma_J=0$ on $\mathcal{U}_x$ where $\sigma_J:=\sideset{}{'}\sum_{I}f_{I,J}dz_I$.
Now using the same arguments as in the proof of the Dolbeault lemma, see \cite[Theorem~2.3.3]{H}, we can for each $J$ f\/ind a $(p-1,0)$-form $\phi_J$ such that $\partial_\eta\phi_J=\sigma_J$ on some open neighborhood of~$x$. Then $\partial_\eta\big(\sideset{}{'}\sum_J\phi_J\wedge d\bar w _J\big)=\sideset{}{'}\sum_J\sigma_J\wedge d\bar w_J=\omega$ on some open neighborhood of~$x$ and the proof is complete.

(iii) This follows from $[\bar\partial,\partial_\eta]=0$ and the observations made above.
\end{proof}

\subsection{Relative de Rham complex}\label{section relative de Rham}
 By def\/inition, $\Omega^\ast_\eta:=\mathcal{E}^{\ast,0}_\eta\cap\ker\bar\partial$ is a~sheaf of holomorphic sections. Since $[\bar\partial,\partial_\eta]=0$, there is a~complex of sheaves $(\Omega^\ast_\eta,\partial_\eta)$ and we call it the \textit{relative de Rham complex}.
\begin{Proposition}\label{thm relative de Rham}\quad
\begin{enumerate}\itemsep=0pt
 \item [$(i)$] The relative de Rham complex is an exact sequence of sheaves which resolves the sheaf $\eta^{-1}\mathcal{O}_\mathfrak{r}$.
\item [$(ii)$] The relative de Rham complex induces for each $r:=\ell+j\ge0$ a long exact sequence of vector bundles
\begin{gather}
\mathfrak{gr}^{\ell+j}\xrightarrow{\mathfrak{gr} \partial_\eta}E^{CS\ast}\otimes\mathfrak{gr}^{\ell+j-1} \rightarrow\cdots\nonumber\\
\hphantom{\mathfrak{gr}^{\ell+j}}{} \rightarrow\Lambda^jE^{CS\ast}\otimes\mathfrak{gr}^{\ell}  \xrightarrow{\mathfrak{gr}\partial_\eta}\Lambda^{j+1}E^{CS\ast}\otimes\mathfrak{gr}^{\ell+j-1} \rightarrow\cdots. \label{first exact formal complex of relative de Rham}
\end{gather}
Let $s_0>0$, $s_1,s_2,s_3\ge0$ be integers such that $s_0+s_1+2s_2+3s_3=r$. Then the sequence~\eqref{first exact formal complex of relative de Rham} contains a long exact subsequence
\begin{gather}\label{second exact formal complex of relative de Rham}
 0\rightarrow {S}^{s_0}E^{CS\ast}\otimes S^{s_1,s_2,s_3}\rightarrow E^{CS\ast}\otimes {S}^{s_0-1}E^{CS\ast}\otimes S^{s_1,s_2,s_3}\rightarrow\cdots \\
\hphantom{0} \rightarrow\Lambda^jE^{CS\ast}\otimes {S}^{s_0-j}E^{CS\ast}\otimes S^{s_1,s_2,s_3}\!\rightarrow \Lambda^{j+1}E^{CS\ast}\otimes {S}^{s_0-j-1}E^{CS\ast}\otimes S^{s_1,s_2,s_3}\!\rightarrow\cdots,\nonumber
\end{gather}
where ${S}^{s_1,s_2,s_3}:={S}^{s_1}F^{CS\ast}\otimes {S}^{s_2}G^{CS}_2\otimes {S}^{s_3}G^{CS}_3$.
\item [$(iii)$] The kernel of the first map in \eqref{first exact formal complex of relative de Rham} is $\bigoplus_{s_1+2s_2+3s_3=r}{S}^{s_1,s_2,s_3}$.
\end{enumerate}
\end{Proposition}

\begin{proof} (i) Since $[\partial_\eta,\bar\partial]=0$, the relative de Rham complex is a sub-complex of the zero-th row $(\mathcal{E}_\mathfrak{q}^{\ast,0},\partial_\eta)$ of the double complex from Proposition~\ref{thm double complex I}. By diagram chasing and using the exactness of columns and rows in the double complex, one easily proves the exactness of the relative de Rham complex. By~(\ref{properties of relative dbar}), it easily follows that $\eta^{-1}\mathcal{O}_\mathfrak{r}=\mathcal{E}^{0,0}_\eta\cap \ker(\partial_\eta)\cap\ker(\bar\partial)$.

(ii) The standard de Rham complex induces the Spencer complex (see \cite{Sp}) which is known to be exact. As the complex $(\Omega^\ast_\eta,\partial_\eta)$ is just a relative version of the (holomorphic) de Rham complex and $\partial_\eta$ satisf\/ies the usual properties of $\partial$, it is clear that the relative de Rham complex induces for each ${s_0}>0$, $s_1$, $s_2$ and $s_3$ the long exact sequence~(\ref{second exact formal complex of relative de Rham}).
The sequence~(\ref{first exact formal complex of relative de Rham}) is the direct sum of all such sequences as ${s_0}$, $s_1$, $s_2,$ and $s_3$ ranges over all quadruples of non-negative integers satisfying $r={s_0}+s_1+2s_2+3s_3$.

(iii) This readily follows from the part (ii).
\end{proof}

\subsection{Twisted relative de Rham complex}\label{section relative twisted double complex}
The weight $\lambda:=(1-2n)\omega_m$ is $\mathfrak{g}$-integral and $\mathfrak{r}$-dominant. Hence, there is an irreducible $\mathrm{R}$-module $\mathbb{W}_\lambda$ with lowest weight $-\lambda$. Since $\mathfrak{r}$ is associated to $\{\alpha_m\}$, it follows that $\dim\mathbb{W}_\lambda=1$ and so $\mathbb{W}_\lambda$ is also an irreducible $\mathrm{Q}$-module. We will denote by $\mathcal{E}_\mathfrak{q}(\lambda)$ and $\mathcal{O}_\mathfrak{q}(\lambda)$ the sheaves of smooth and holomorphic sections of $W_\lambda^{CS}:=\mathrm{G}\times_\mathrm{Q}\mathbb{W}_\lambda$, respectively.
If $W$ is a vector bundle over $CS$, then we denote $W(\lambda):=W\otimes W_\lambda^{CS}$, i.e., we twist $W$ by tensoring with the line bundle~$W_\lambda^{CS}$. It is not hard to see that $\eta^\ast\mathcal{E}_\mathfrak{r}(\lambda)\cong\mathcal{E}_\mathfrak{q}(\lambda)$ and $\eta^\ast\mathcal{O}_\mathfrak{r}(\lambda)=\mathcal{O}_\mathfrak{q}(\lambda)$
where we denote by the subscript $\mathfrak{r}$ the corresponding sheaves over~$TS$.

We call $\mathcal{E}^{p,q}_\eta(\lambda):=\mathcal{E}^{p,q}_\eta\otimes_{\eta^{-1}\mathcal{E}_\mathfrak{r}}\eta^{-1}\mathcal{E}_\mathfrak{r}(\lambda)$ the \textit{sheaf of twisted relative $(p,q)$-forms}. Consider the following sequence of canonical isomorphisms:
\begin{gather}
 \mathcal{E}^{p,q}_\eta(\lambda)\rightarrow\mathcal{E}^{p,q}_\eta\otimes_{\mathcal{E}_\mathfrak{q}} \mathcal{E}_\mathfrak{q}\otimes_{\eta^{-1}\mathcal{E}_\mathfrak{r}}\eta^{-1}\mathcal{E}_\mathfrak{r} (\lambda)\rightarrow\mathcal{E}^{p,q}_\eta\otimes_{\mathcal{E}_\mathfrak{q}}\eta^\ast\mathcal{E}_\mathfrak{r}(\lambda)\nonumber\\
\hphantom{\mathcal{E}^{p,q}_\eta(\lambda)}{} \rightarrow\mathcal{E}^{p,q}_\eta\otimes_{\mathcal{E}_\mathfrak{q}}\mathcal{E}_\mathfrak{q}(\lambda) \rightarrow\mathcal{E}^{0,q}_\mathfrak{q}\otimes_{\mathcal{E}_\mathfrak{q}}\big(\mathcal{E}^{p,0}_\eta \otimes_{\mathcal{E}_\mathfrak{q}}\mathcal{E}_\mathfrak{q}(\lambda)\big)
\rightarrow\mathcal{E}^{0,q}_\mathfrak{q}\big(\Lambda^pE^{CS\ast}(\lambda)\big).\label{formula18-18}
\end{gather}
We see that $\mathcal{E}^{p,q}_\eta(\lambda)$ is isomorphic to the sheaf of smooth $(0,q)$-forms with values in $\Lambda^pE^{CS\ast}(\lambda)$. Hence, the Dolbeault dif\/ferential induces a dif\/ferential $\bar\partial\colon \mathcal{E}^{p,q}_\eta(\lambda)\rightarrow\mathcal{E}^{p,q+1}_\eta(\lambda)$ and a complex $(\mathcal{E}^{p,\ast}_\eta(\lambda),\bar\partial)$.

A section of $\mathcal{E}^{p,q}_\eta(\lambda)$ is by def\/inition a f\/inite sum of decomposable elements $\omega\otimes v$ where $\omega$ and $v$ are sections of $\mathcal{E}^{p,q}_\eta$ and $\eta^{-1}\mathcal{E}_\mathfrak{r}(\lambda)$, respectively. As any section of $\eta^{-1}\mathcal{E}_\mathfrak{r}$ as well as transition functions between sections of $\eta^{-1}\mathcal{E}_\mathfrak{r}(\lambda)$ belong to $\ker(\partial_\eta)$, it follows that there is a unique linear dif\/ferential operator
\begin{gather*}\label{twisted relative de Rham operator}
\mathcal{E}^{p,q}_\eta(\lambda)\rightarrow\mathcal{E}^{p+1,q}_\eta(\lambda),
\end{gather*}
which satisf\/ies $\omega\otimes v\mapsto\partial_\eta\omega\otimes v$. We denote the operator also by $\partial_\eta$ as there is no risk of confusion. It is clear that $\partial_\eta$ is a linear $\mathrm{G}$-invariant dif\/ferential operator of weighted order one.

\begin{Proposition}\label{thm double complex I}
Let $p,q\ge0$ be integers.
\begin{enumerate}\itemsep=0pt
 \item[$(i)$] The sequence of sheaves $(\mathcal{E}^{(p,\ast)}_\eta(\lambda),\bar\partial)$ is exact.
\item [$(ii)$] The sequence of sheaves $(\mathcal{E}^{(\ast,q)}_\eta(\lambda),\partial_\eta)$ is exact.
\item [$(iii)$] There is a double complex $(\mathcal{E}_\eta^{p,q}(\lambda),\partial_\eta,(-1)^p\bar{\partial})$ of fine sheaves with exact rows and columns.
 \end{enumerate}
\end{Proposition}

\begin{proof}
(i) By construction, the sequence is a Dolbeault complex and the claim follows.

(ii) The exactness follows immediately from Proposition \ref{thm double complex relative forms}(ii).

(iii) We need to verify that $[\bar\partial,\partial_\eta]=0$. To see this, notice that a section of $\mathcal{E}^{p,q}_\eta(\lambda)$ can be locally written as a f\/inite sum of elements as above with $v$ holomorphic. The claim then easily follows from Proposition~\ref{thm double complex relative forms}(iii).
\end{proof}

Put $\Omega^{\ast}_\eta(\lambda):=\mathcal{E}^{\ast,0}_\eta(\lambda)\cap\ker(\bar\partial)$. The complex $(\mathcal{E}^{\ast,0}_\eta(\lambda),\partial_\eta)$ contains a sub-complex $(\Omega^\ast_\eta(\lambda),\partial_\eta)$ which we call the \textit{twisted relative de Rham complex}. As in Proposition~\ref{thm relative de Rham}, one can easily see that $(\Omega_\eta^\ast(\lambda),\partial_\eta)$ is an exact sequence of sheaves of holomorhic sections. Following the proof of Proposition~\ref{thm relative de Rham}, we obtain the following:

\begin{Proposition}\label{thm relative twisted de Rham}
The relative de Rham complex $(\Omega_\eta^\ast(\lambda),\partial_\eta)$ induces for each $r\ge0$ a long exact sequence of vector bundles
\begin{gather}\label{first exact formal complex of twisted relative de Rham}
\big(\Lambda^\bullet E^{CS\ast}\otimes\mathfrak{gr}^{r-\bullet}(\lambda),\mathfrak{gr}\partial_\eta\big).
\end{gather}
 Let $s_0>0$, $s_1,s_2,s_3\ge0$ be integers such that $s_0+s_1+2s_2+3s_3=r$. Then the sequence~\eqref{first exact formal complex of twisted relative de Rham} contains a long exact subsequence
\begin{gather}\label{second exact formal complex of twisted relative de Rham}
\big(\Lambda^\bullet E^{CS\ast}\otimes {S}^{s_0-\bullet}E^{CS\ast}\otimes {S}^{s_1,s_2,s_3}(\lambda),\mathfrak{gr}\partial_\eta\big),
\end{gather}
where ${S}^{s_1,s_2,s_3}$ is defined in Proposition~{\rm \ref{thm relative de Rham}}.
The kernel of the first map in~\eqref{first exact formal complex of twisted relative de Rham} is the bundle $\bigoplus_{s_1+2s_2+3s_3=r}{S}^{s_1,s_2,s_3}(\lambda)$.
\end{Proposition}

\subsection{Relative BGG sequence}\label{section relative BGG sequence}
We know that $\Omega^p_\eta(\lambda)$ is isomorphic to the sheaf of holomorphic sections of $\Lambda^pE^{CS\ast}(\lambda)=\mathrm{G}\times_\mathrm{Q}(\Lambda^p\mathbb{E}^{\ast}\otimes\mathbb{W}_\lambda)$. The $\mathrm{Q}$-module $\Lambda^p\mathbb{E}^{\ast}$ is not irreducible. Decomposing this module into irreducible $\mathrm{Q}$-modules, we obtain from the relative twisted de Rham complex a relative BGG sequence and this will be crucial in the construction of the $k$-Dirac complexes.
We will use notation from Section~\ref{section the relative Weyl group}.

\begin{Proposition}\label{thm relative bgg sequence}%(\textbf{The relative BGG sequence})
Let $a\in\mathbb{N}^{k,n}_{++}$ and $w_a\in W^\mathfrak{q}_\mathfrak{r}$ be as in Section~{\rm \ref{section the relative Weyl group}}. Then
\begin{gather}\label{decomposition of skew symmetric powers}
\Lambda^p\mathbb{E}^\ast\otimes\mathbb{W}_\lambda= \!\!\bigoplus_{a\in\mathbb{N}^{k,n}_{++}:|a|=p} \!\!\mathbb{W}_{\lambda_a}\qquad \text{and  thus}\qquad  \Omega^{p}_\eta(\lambda)= \!\!\bigoplus_{a\in\mathbb{N}^{n,k}_{++}:|a|=p} \!\!\mathcal{O}_\mathfrak{q}(\lambda_a),
\end{gather}
where $\mathbb{W}_{\lambda_a}$ is an irreducible $\mathrm{Q}$-module with lowest weight $-\lambda_a:=-w_a.\lambda$.

There is a linear $\mathrm{G}$-invariant differential operator
\begin{gather}\label{dif op in BGG}
\partial_{a'}^a\colon \ \mathcal{O}_\mathfrak{q}(\lambda_a)\rightarrow\Omega^p_\eta(\lambda)\xrightarrow{\partial_\eta}\Omega^p_\eta(\lambda)\rightarrow\mathcal{O}_\mathfrak{q}(\lambda_{a'})
\end{gather}
where the first map is the canonical inclusion and the last map is the canonical projection. If $a\nless a'$, then $\partial_{a'}^a=0$.
 \end{Proposition}

\begin{proof} Recall from Section \ref{section lie algebra} that the semi-simple part $\mathfrak{r}_0^{ss}$ of $\mathfrak{r}_0$ is isomorphic to $\mathfrak{sl}(m,\mathbb{C})$ and that $\mathfrak{r}_0^{ss}\cap\mathfrak{q}$ is a parabolic subalgebra of~$\mathfrak{r}_{0}^{ss}$.
The direct sum decomposition from~(\ref{decomposition of skew symmetric powers}) then follows at once from the Kostant's version of the Bott--Borel--Weyl theorem (see \cite[Theorem~3.3.5]{CS}) applied to $\mathbb{W}_\lambda$ and
$(\mathfrak{r}_0^{ss},\mathfrak{r}_0^{ss}\cap\mathfrak{q})$ and the identity $\ell(w_a)=|a|$ from Lemma~\ref{lemma length of permutation}.
Recall from \cite[Section~8.7]{BE} that the graph of the relative BGG sequence coincides with the relative Hasse graph~$W^\mathfrak{q}_\mathfrak{r}$. The last claim then follows from Lemma~\ref{lemma relative Weyl group}.
\end{proof}

\begin{Remark}\label{remark weight}
Let $a=(a_1,\dots,a_k)\in\mathbb{N}^{k,n}_{++}$ and $b=(b_1,\dots,b_n)\in\mathbb{N}^{n,k}_{++}$ be the conjugated partition. In order to compute $\lambda_a$ from Proposition~\ref{thm relative bgg sequence}, notice that
\begin{gather}\label{shifted lambda}
\lambda+\rho=\bigg(\frac{2k-1}{2},\dots,\frac{3}{2},\frac{1}{2}\,\bigg|\,-\frac{1}{2},-\frac{3}{2},\dots,\frac{-2n+1}{2}\bigg).
\end{gather}
Since $w_a(\omega_m)=\omega_m$, we have $w_a(\lambda)=\lambda$ and thus $\lambda_a=w_a(\lambda+\rho)-\rho=\lambda+w_a\rho-\rho$. By~(\ref{action on lowest form}), it follows that
\begin{gather*}%\label{action on weights}
\lambda_a=\lambda+(-a_k,\dots,-a_1\,|\, b_1,b_2,\dots,b_n).
\end{gather*}
\end{Remark}

\subsection{Direct image of the relative BGG sequence}\label{section direct images}

Recall from \cite[Section 5.3]{BE} that given a $\mathfrak{g}$-integral and $\mathfrak{q}$-dominant weight $\nu$, there is at most one $\mathfrak{p}$-dominant weight in the $W_\mathfrak{p}^\mathfrak{q}$-orbit of $\nu$. If there is no $\mathfrak{p}$-dominant weight, then all direct images of $\mathcal{O}_\mathfrak{q}(\nu)$ vanish. If there is a $\mathfrak{p}$-dominant weight, say $\mu=w.\nu$ where $w\in W^\mathfrak{q}_\mathfrak{p}$, then $\tau^{\ell(w)}_\ast\mathcal{O}_\mathfrak{q}(\nu)\cong\mathcal{O}_\mathfrak{p}(\mu)$ is the unique non-zero direct image of $\mathcal{O}_\mathfrak{q}(\nu)$.

\begin{Proposition}\label{thm direct images}
Let $n\ge k\ge 2$ and $a=(a_1,\dots,a_k)\in\mathbb{N}^{k,n}_{++}$. Put $\mu^\pm:=\frac{1}{2}(-2n+1,\dots,-2n+1\,|\,1,\dots,1,\pm1)$ and
$\mu_a:=\mu^\ast-(a_k,\dots,a_1\,|\,0,\dots,0)$ where $\ast=+$ if $d(a)\equiv n \mod 2$ and $\ast=-$ otherwise.
Then
\begin{gather*}
\tau_\ast^q(\mathcal{O}_\mathfrak{q}(\lambda_a))
=
 \begin{cases}
\mathcal{O}_\mathfrak{p}(\mu_a),&\mathrm{if} \ a\in S^k,\ q=\ell(a):={n\choose2}-q(a),\\
\{0\},& \mathrm{otherwise}.
\end{cases}
\end{gather*}
\end{Proposition}

\begin{proof}By def\/inition, each $w\in W_\mathfrak{p}^\mathfrak{q}$ f\/ixes the f\/irst $k$ coef\/f\/icients of $\lambda_a$ and so it is enough to look at the last $n$ coef\/f\/icients. By Remark~\ref{remark weight}, it follows
\begin{gather*}% \label{action of the weight}
w_a(\lambda+\rho)=\lambda_a+\rho=\bigg(\dots,i-a_i-\frac{1}{2},\dots,\frac{1}{2}-a_1\,\bigg|\, b_1-\frac{1}{2},\dots,b_j-j+\frac{1}{2},\dots\bigg),
\end{gather*}
where $b=(b_1,\dots,b_n)$ is the conjugated partition.
Put $c:=(c_1,\dots,c_n)$ where $c_j:=\big|b_j-j+\frac{1}{2}\big|$. By~(\ref{simple reflections}) and Table~\ref{table dominant weights}, if $c_i= c_j$ for some $i\ne j$, then there cannot be a~$\mathfrak{p}$-dominant weight in the $W^\mathfrak{q}_\mathfrak{p}$-orbit of~$\lambda_a$. If $a\not\in S^k$, then by Lemma~\ref{lemma symmetric YD} there is $s\in\{0,\dots,k-1\}$ such that\footnote{Notice that at this point we need that $n\ge k$.} $w_a(k-s)=k+i\ge k$ and $w_a(k+s+1)=k+j\ge k$ for some distinct positive integers~$i$ and~$j$. By~(\ref{shifted lambda}), it follows that $b_i-i+\frac{1}{2}=\frac{1}{2}+s$, $b_j-j-\frac{1}{2}=-\frac{1}{2}-s$ and thus $c_i=c_j$. Hence, all direct images of $\mathcal{O}_\mathfrak{q}(\lambda_a)$ are zero.

We may now suppose that $a\in S^k$. By the def\/inition of $d(a)$, we have $b_{d(a)}-d(a)+\frac{1}{2}>0>b_{d(a)+1}-d(a)-\frac{1}{2}$. By Lemma \ref{lemma symmetric YD} and~(\ref{shifted lambda}), it follows that $(c_1,\dots,c_n)$ is a permutation of $\big(\frac{2n-1}{2},\dots,\frac{3}{2},\frac{1}{2}\big)$. As $\lambda_a$ is $\mathfrak{q}$-dominant, we know that the sequence $\big(b_1-\frac{1}{2},\cdots,b_j-j+\frac{1}{2},\dots\big)$ is decreasing. Thus for each integer $i=1,\dots,d(a)$, the set $\{j\colon i<j\le n,\ c_i>c_j\}$ contains precisely $b_i-i$ distinct elements. Altogether, there are precisely $\sum\limits_{\ell=1}^{d(a)}(b_\ell-\ell)=q(a)$ pairs $i<j$ such that $c_i>c_j$. Equivalently, there are $\ell(a)$ pairs $i<j$ such that $c_i<c_j$. It follows that the length of the permutation that maps $(c_1,\dots,c_n)$ to $\big(\frac{2n-1}{2},\dots,\frac{3}{2},\frac{1}{2}\big)$ is precisely $\ell(a)$. Now it is easy to see (recall~(\ref{simple reflections})) that there is $w\in
W^\mathfrak{q}_\mathfrak{p}$ such that $w.\lambda_a$ is $\mathfrak{p}$-dominant and $\ell(w)=\ell(a)$. As there are $n-d(a)$ negative numbers in the sequence $\big(b_1-\frac{1}{2},\dots,b_j-j+\frac{1}{2},\dots\big)$, the last claim about the sign of the last coef\/f\/icient of $w.\lambda_a$ also follows. This completes the proof.\end{proof}

\begin{Remark}\label{remark PT for other Dirac}
In Proposition~\ref{thm direct images} we recovered the $W^\mathfrak{p}$-orbit of the singular weight $\mu^+$ if $n$ is even and of $\mu^-$ if $n$ is odd which was computed in \cite{F}. There is an automorphism of $\mathfrak{g}$ which swaps $\alpha_m$ and $\alpha_{m-1}$ and hence, it swaps also the associated parabolic subalgebras. If we cross in~(\ref{double fibration diagram I}) the simple root $\alpha_{m-1}$ instead of $\alpha_m$, take $(1-2n)\omega_{m-1}$ as $\lambda$ and follow the computations given above, we will get the~$W^\mathfrak{p}$-orbit of $\mu^+$ if~$n$ is odd and of $\mu^-$ if $n$ is even. As also all other arguments presented in this paper work for the other case, we will obtain the other ``half'' of the $k$-Dirac complex from \cite{TS} as mentioned in Introduction.
\end{Remark}

\subsection{Double complex of relative forms II}\label{section relative double complex II}
The direct sum decomposition from Proposition \ref{thm relative bgg sequence} together with the isomorphism in~\eqref{formula18-18} gives a direct sum decomposition $\mathcal{E}^{p,q}_\eta(\lambda)=\bigoplus_{a\in\mathbb{N}^{k,n}_{++}:|a|=p}\mathcal{E}^{0,q}_\mathfrak{q}(\lambda_a)$.
Let $a,a'\in\mathbb{N}^{k,n}_{++}$ be such that $p=|a|=|a'|-1$. Then there is a linear dif\/ferential operator
\begin{gather}\label{relative differential in double complex}
\partial_{a'}^a\colon \ \mathcal{E}^{0,q}_\mathfrak{q}(\lambda_a)\rightarrow\mathcal{E}^{p,q}_\eta(\lambda)\xrightarrow{\partial_\eta}\mathcal{E}^{p+1,q}_\eta(\lambda)\rightarrow\mathcal{E}^{0,q}_\mathfrak{q}(\lambda_{a'})
\end{gather}
where the f\/irst map is the canonical inclusion and the last map is the canonical projection as in (\ref{dif op in BGG}). We denote the dif\/ferential operator by $\partial_{a'}^a$ as in (\ref{dif op in BGG}) as there is no risk of confusion. Recall from Proposition \ref{thm relative bgg sequence} that $\partial^a_{a'}=0$ if $a\nless a'$.

Suppose that $\mathcal U$ is an open, contractible and Stein subset of $M$. We put $\mathcal{E}^{p,q}_\eta(\tau^{-1}(\mathcal U),\lambda):=\Gamma(\tau^{-1}(\mathcal U),\mathcal{E}^{p,q}_\eta(\lambda))$, i.e., this is the space of sections of the sheaf $\mathcal{E}^{p,q}_\eta(\lambda)$ over $\tau^{-1}(\mathcal U)$. Then there is a double complex
\begin{gather}\label{double complex II}\begin{split}
\xymatrix{&&&\\
\cdots\ar[r]& \mathcal{E}^{p,q+1}_\eta(\tau^{-1}(\mathcal U),\lambda)\ar[r]^{D''}\ar[u] &\mathcal{E}^{p+1,q+1}_\eta(\tau^{-1}(\mathcal U),\lambda) \ar[u]\ar[r]&\cdots\\
\cdots\ar[r]&\mathcal{E}^{p,q}_\eta(\tau^{-1}(\mathcal U),\lambda)\ar[r]^{D''}\ar[u]^{D'}&\mathcal{E}^{p+1,q}_\eta(\tau^{-1}(\mathcal U),\lambda)\ar[u]^{D'}\ar[r]&\cdots,\\
&\ar[u]&\ar[u]&}\end{split}
\end{gather}
where $D'=(-1)^p\bar\partial$ and $D''=\partial_\eta$. Put $T^i(\mathcal U):=\bigoplus_{p+q=i}\mathcal{E}^{p,q}_\eta(\tau^{-1}(\mathcal U),\lambda)$. We obtain three complexes $(T^{\ast}(\mathcal U),D')$, $(T^{\ast}(\mathcal U),D'')$ and $(T^\ast(\mathcal U),D'+D'')$.

\begin{Lemma}\label{lemma sections over stein set}
Let $a\in\mathbb{N}^{k,n}_{++}$ and $\mathcal U$ be the open, contractible and Stein subset of $M$ as above. Then
\begin{gather}\label{sections over stein set}
H^{q}\big(\tau^{-1}(\mathcal U),\mathcal{O}_\mathfrak{q}(\lambda_a)\big)=
\begin{cases}
\Gamma(\mathcal U,\mathcal{O}_\mathfrak{p}(\mu_a)),& a\in S^k, \ q=\ell(a),\\
\{0\}, & \mathrm{otherwise},
\end{cases}
\end{gather}
and thus also
\begin{gather*}
H^{{n\choose2}+j}(T^\ast(\mathcal U),D')=\bigoplus_{a\in S^k_j}H^{\ell(a)}\big(\tau^{-1}(\mathcal U),\mathcal{O}_\mathfrak{q}(\lambda_a)\big).
\end{gather*}
\end{Lemma}
\begin{proof}The f\/irst claim follows from Proposition \ref{thm direct images} and application of the Leray spectral sequence as explained in \cite{BE}. For the second claim, recall from \cite[Theorem~3.20]{W} that the sheaf cohomology is equal to the Dolbeault cohomology, i.e., there is an isomorphism
\begin{gather}\label{computing sheaf cohomology groups}
H^{q}\big(\tau^{-1}(\mathcal U),\mathcal{O}_\mathfrak{q}(\lambda_a)\big)\cong H^{q}\big(\mathcal{E}^{0,\ast}_\mathfrak{q}\big(\tau^{-1}(\mathcal U),\lambda_a\big),\bar\partial\big).
\end{gather}
The cohomology group appears on the $(|a|+\ell(a))=(d(a)+2q(a)+{n\choose2}-q(a))=({n\choose2}+r(a))$-th diagonal of the double complex. Here, see Proposition \ref{thm direct images}, we use that $\ell(a)={n\choose2}-q(a)$, the notation from~(\ref{abbreviations}) and $S^k_j=\{a\in S^k\colon r(a)=j\}$.
\end{proof}

\section[$k$-Dirac complexes]{$\boldsymbol{k}$-Dirac complexes}\label{section complex}

In Section \ref{section complex} we will give the def\/inition of dif\/ferential operators in the $k$-Dirac complexes. It will be clear from the construction that the operators are linear, local and $\mathrm{G}$-invariant. Later in Lemma~\ref{lemma diff op on graded jets} we will show that each operator is indeed a dif\/ferential operator and we give an upper bound on its weighted order. The operators naturally form a sequence
and we will prove in Theorem~\ref{theorem complex} that they form a complex which we call the $k$-Dirac complex.

Recall from Section \ref{section relative double complex II} that $\mathcal{E}^{0,q}_\mathfrak{q}(\tau^{-1}(\mathcal U),\lambda_a)$ is the space of sections of the sheaf $\mathcal{E}^{0,q}_\mathfrak{q}(\lambda_a)$ over~$\tau^{-1}(\mathcal U)$. If $\alpha\in\mathcal{E}^{0,q}_\mathfrak{q}(\tau^{-1}(\mathcal U),\lambda_a)$ is $\bar\partial$-closed, then we will denote by $[\alpha]\in H^{q}(\tau^{-1}(\mathcal U),\mathcal{O}_\mathfrak{q}(\lambda_a))$ the corresponding cohomology class.

\begin{Lemma}\label{lemma diff op}
Let $j\ge0$, $a\in S^k_j$, $a'\in S^k_{j+1}$ be such that $a<a'$ and $\mathcal{U}$ be the Stein set as above. Then there is a~linear, local and $\mathrm{G}$-invariant operator
\begin{gather*}%\label{diff op}
D_{a'}^a\colon \ \Gamma(\mathcal U,\mathcal{O}_\mathfrak{p}(\mu_a))\rightarrow\Gamma(\mathcal U,\mathcal{O}_\mathfrak{p}(\mu_{a'})).
\end{gather*}
\end{Lemma}
\begin{proof}
Let us for a moment put $\mathcal{V}:=\tau^{-1}(\mathcal U)$. Using the isomorphisms from~(\ref{sections over stein set}), it is enough to def\/ine a map $H^{\ell(a)}(\mathcal{V},\mathcal{O}_\mathfrak{q}(\lambda_a))\rightarrow H^{\ell(a')}(\mathcal{V},\mathcal{O}_\mathfrak{q}(\lambda_{a'}))$ which has the right properties. By assumption, we have $|a'|-|a|\in\{1,2\}$. Let us f\/irst consider $|a'|-|a|=1$. Then $q:=\ell(a')=\ell(a)$ and by (\ref{relative differential in double complex}), we have the map $\partial_{a'}^a\colon \mathcal{E}^{0,q}_\mathfrak{q}(\mathcal{V},\lambda_a)\rightarrow\mathcal{E}^{0,q}_\mathfrak{q}(\mathcal{V},\lambda_{a'})$ in
the double complex (\ref{double complex II}). The induced map on cohomology is $D_{a'}^a$.

If $|a'|-|a|=2$, then $q:=\ell(a)=\ell(a')+1$ and we f\/ind that there are precisely two non-symmetric partitions $b,c\in\mathbb{N}^{k,n}_{++}$ such that $a<b<a'$ and $ a<c<a'$. Then there is a diagram
\begin{gather}\label{def of second order operator}\begin{split} &
\xymatrix{\mathcal{E}^{0,q}_\mathfrak{q}(\mathcal{V},\lambda_a)\ar[r]^{\!\!\!\!\!\!\!\!\!\!\!\!\!\!\!(\partial^a_b,\partial_c^a)}&\mathcal{E}^{0,q}_\mathfrak{q}(\mathcal{V},\lambda_b)\oplus\mathcal{E}^{0,q}_\mathfrak{q}(\mathcal{V},\lambda_c)&\\
&\mathcal{E}^{0,q-1}_\mathfrak{q}(\mathcal{V},\lambda_b)\oplus\mathcal{E}^{0,q-1}_\mathfrak{q}(\mathcal{V},\lambda_c)\ar[u]^{(-1)^p\bar\partial}\ar[r]^{ \ \ \ \ \ \ \ \ \ \partial_{a'}^b+\partial_{a'}^c}&\mathcal{E}^{0,q-1}_\mathfrak{q}(\mathcal{V},\lambda_{a'}),}\end{split}
\end{gather}
which lives in the double complex (\ref{double complex II}). Let $\alpha\in\mathcal{E}^{0,q}_\mathfrak{q}(\mathcal{V},\lambda_a)$ be $\bar\partial$-closed. Then $\partial_{b}^a\alpha$ and $\partial_{c}^a\alpha$ are also $\bar\partial$-closed and thus by Lemma \ref{lemma sections over stein set} and the isomorphism~(\ref{computing sheaf cohomology groups}), we can f\/ind $\beta$ and $\gamma$ such that $\partial_{b}^a\alpha=(-1)^p\bar\partial\beta$ and $\partial_c^a\alpha=(-1)^p\bar\partial\gamma$ where $p=|a|+1$. Since the relative BGG sequence is a complex, we have
\begin{gather*}
 \bar\partial\big(\partial_{a'}^b\beta+\partial_{a'}^c\gamma\big)=(-1)^p(\partial_{a'}^b\partial_b^a+\partial_{a'}^c\partial_c^a)\alpha=0,
\end{gather*}
which shows that $\partial_{a'}^b\beta+\partial_{a'}^c\gamma$ is a cocycle. Of course this elements depends on choices but we claim that $[\partial_{a'}^b\beta+\partial_{a'}^c\gamma]$ depends only on $[\alpha]$. It is easy to see that $[\partial_{a'}^b\beta+\partial_{a'}^c\gamma]$ does not depend on the choices of $\beta$ and $\gamma$. If $[\alpha]=0$, say $\alpha=(-1)^{p-1}\bar\partial\varrho$, then
we may put $\beta=-\partial_b^a\varrho$ and $\gamma=-\partial_c^a\varrho$ and thus $\partial_{a'}^b\beta+\partial_{a'}^c\gamma=-(\partial_{a'}^b \partial^a_b+\partial^c_{a'} \partial_c^a)\varrho=0$. Hence, we can put $D_{a'}^a[\alpha]:=[\partial_{a'}^b\beta+\partial_{a'}^c\gamma]$.

From the construction is clear that $D_{a'}^a$ is linear. The locality follows from the fact that $D_{a'}^a$ is compatible with restrictions to smaller Stein subsets of $\mathcal{U}$. As the operators in the double complex (\ref{double complex II}) are $\mathrm{G}$-invariant, it is easy to verify that each operator $D_{a'}^a$ is $\mathrm{G}$-invariant.
\end{proof}

Put $\mathcal{O}_j:=\bigoplus_{a\in S^k_{j}}\mathcal{O}_\mathfrak{p}(\mu_a)$ and $\mathcal{O}_j(\mathcal{U}):=\Gamma(\mathcal{U},\mathcal{O}_j)$. If $a\in S^k_j$ and $s\in\mathcal{O}_j(\mathcal{U})$, then we denote by $s_a$ the $a$-th component of $s$ so that we may write $s=(s_a)_{a\in S^k_j}$.
We call the following complex~(\ref{sequence of diff ope}) the $k$-\textit{Dirac complex}.

\begin{Theorem}\label{theorem complex}
With the notation set above, there is a complex
\begin{gather}\label{sequence of diff ope}
\mathcal{O}_0(\mathcal{U})\xrightarrow{D_0}\mathcal{O}_1(\mathcal{U})\rightarrow\cdots \rightarrow\mathcal{O}_j(\mathcal{U})\xrightarrow{D_j}\mathcal{O}_{j+1}(\mathcal{U})\rightarrow\cdots
\end{gather}
of linear $\mathrm{G}$-invariant operators where
\begin{gather*}
 (D_js)_{a'}=\sum_{a<a'}D_{a'}^as_a.
\end{gather*}
\end{Theorem}
\begin{proof}
Let $a,a'\in S^k$ be such that $a<a'$, $r(a)=r(a')-2$. We need to verify that $\sum\limits_{a''\in S^k\colon a<a''<a'}D_{a'}^{a''}D_{a''}^{a}=0$. Observe that $|a'|-|a|\in\{3,4\}$. Let us f\/irst assume that $|a'|-3=|a|$. Then there are at most two symmetric partitions $a''$ such that $a<a''<a'$. If there is only one such symmetric partition $a''$, then, since the relative BGG sequence is a~complex, it follows easily that $D_{a'}^{a''}D_{a''}^{a}=0$. So we can assume that there are two symmetric partitions, say $a''_1$, $a''_2$. Consider for example
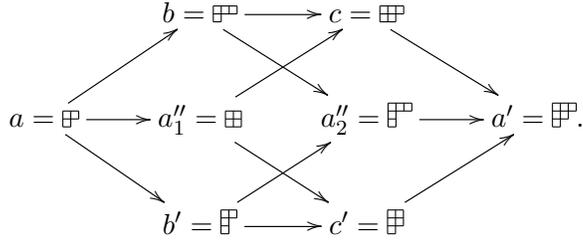
\begin{gather*}\xymatrix{&b=\sYtj\ar[rd]\ar[r]&c=\sYtd\ar[rd]&\\
a=\sYdj\ar[ru]\ar[r]\ar[rd]&a''_1=\sYdd\ar[ru]\ar[rd]&a''_2=\sYtjj\ar[r]&a'=\sYtdj.\\
&b'=\sYdjj\ar[r]\ar[ru]&c'=\sYddj\ar[ru].&}
\end{gather*}
Then we can f\/ind $\beta$ and $\beta'$ so that $\partial_b^a\alpha=(-1)^{p}\bar\partial\beta$ and $\partial_{b'}^a\alpha=(-1)^{p}\bar\partial\beta'$ where $p=|b|=|b'|$. Then $[D_{a_2''}^a\alpha]=[\partial_{a''_2}^b\beta+\partial_{a''_2}^{b'}\beta']$ which implies
\begin{gather*}
(-1)^{p+1}\bar\partial\big(\partial_c^b\beta\big)=(-1)^{p+1}\partial_c^b\bar\partial\beta=-\partial_c^b \partial_b^a\alpha=\partial_{c}^{a_1''} \partial_{a_1''}^a\alpha
\end{gather*}
and similarly $(-1)^{p+1}\bar\partial(\partial_{c'}^{b'}\beta')=\partial_{c'}^{a_1''}\partial_{a''_1}^a\alpha$. Hence, we conclude that \begin{gather*} D^{a_1''}_{a'} D^{a}_{a_1''}[\alpha]=\big[\partial_{a'}^{c}\partial_c^b\beta+\partial_{a'}^{c'}\partial_{c'}^{b'}\beta'\big]\end{gather*} and thus
\begin{gather*}
D_{a'}^{a_2''} D_{a_2''}^a[\alpha]=\big[\partial_{a'}^{a_2''}\big(\partial_{a_2''}^b\beta+\partial_{a_2''}^{b'}\beta'\big)\big] =-\big[\partial_{a'}^{c}\partial_c^b\beta+\partial_{a'}^{c'}\partial_{c'}^{b'}\beta'\big]=-D^{a_1''}_{a'} D^{a}_{a_1''}[\alpha].
\end{gather*}
This completes the proof when $|a'|=|a|+3$ and now we may assume $|a'|=|a|+4$.

We put $A'':=\{a''\in S^k|\ a<a''<a'\}$, $B:=\{b\in\mathbb{N}^{k,n}_{++}\,|\, \exists\, a''\in A''\colon a<b<a''\}$, $B':=\{b'\in\mathbb{N}^{k,n}_{++}\,|\, \exists\, a''\in A''\colon a''<b'<a'\}$ and f\/inally $C:=\{c\in\mathbb{N}^{k,n}_{++}\setminus S^k\,|\, \exists \, b\in B,\ \exists\, b'\in B'\colon b'<c<b''\}$.
 Consider for example the diagram
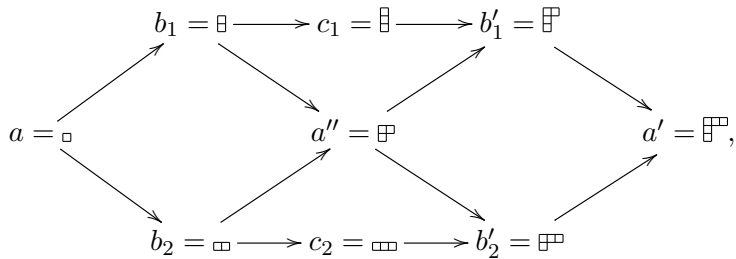
\begin{gather*}\xymatrix{
&b_1=\sYjj\ar[r]\ar[rd]&c_1=\sYjjj\ar[r]&b'_1=\sYdjj\ar[rd]\\
a=\sYj\ar[rd]\ar[ru]&&a''=\sYdj\ar[rd]\ar[ru]&&a'=\sYtjj,\\
&b_2=\sYd\ar[r]\ar[ru]&c_2=\sYt\ar[r]&b_2'=\sYtj\ar[ru]&}
\end{gather*}
where $A''=\{a''\}$, $B=\{b_1,b_2\}$, $B'=\{b_1',b_2'\}$ and $C=\{c_1,c_2\}$. As above, the set $A''$ contains at most two elements but we will not need that.

Now we can proceed as above. There are $\beta_i$ such that $(-1)^{p}\bar\partial\beta_i=\partial_{b_i}^a\alpha$ where $p=|b_i|=|a|+1$ and so $[D_{a''_j}^{a}\alpha]=\big[\sum\limits_{b_i\in B}\partial_{a''_j}^{b_i}\beta_i\big]$ for every $a''_j\in A''$. As the relative BGG sequence is a~complex, we have for each $c_\ell\in C$:
\begin{gather*} \bar\partial\bigg(\sum_{b_i\in B}\partial_{c_\ell}^{b_i}\beta_i\bigg)=\sum_{b_i\in B}\partial_{c_\ell}^{b_i}\bar\partial\beta_i=(-1)^{p}\sum_{b_i\in B}\partial_{c_\ell}^{b_i}\partial_{b_i}^a\alpha=0.
\end{gather*}
As above, there is $\gamma_\ell$ such that $(-1)^{p+1}\bar\partial\gamma_\ell=\sum\limits_{b_i\in B}\partial_{c_\ell}^{b_i}\beta_i$. Then for $b'_s\in B'$:
\begin{gather*}
(-1)^{p+2}\bar\partial\bigg(\sum_{c_\ell\in C}\partial_{b'_s}^{c_\ell}\gamma_\ell\bigg) =-\sum_{c_\ell\in C}\partial_{b'_s}^{c_\ell}((-1)^{p+1}\bar\partial\gamma_\ell)
 =-\sum_{c_\ell\in C,\ b_i\in B}\partial_{b'_s}^{c_\ell}\partial_{c_\ell}^{b_i}\beta_i\\
\hphantom{(-1)^{p+2}\bar\partial\bigg(\sum_{c_\ell\in C}\partial_{b'_s}^{c_\ell}\gamma_\ell\bigg)}{}
 =\sum_{b_i\in B,\ a''_j\in A''_j\colon b_i<a''_j<b_{s}'}\partial_{b'_s}^{a''_j}\partial_{a''_j}^{b_i}\beta_i
 =\sum_{a''_j\in A\colon a''_j<b'_s}\partial_{b'_s}^{a''_j}\bigg(\sum_{b_i\in B}\partial_{a''_j}^{b_i}\beta_i\bigg)\\
 \hphantom{(-1)^{p+2}\bar\partial\bigg(\sum_{c_\ell\in C}\partial_{b'_s}^{c_\ell}\gamma_\ell\bigg)}{}
 =\sum_{a''_j\in A\colon a''_j<b'_s}\partial_{b'_s}^{a''_j}(D_{a''_j}^a\alpha).
\end{gather*}
This implies that $\sum\limits_{a''_j\in A''}D_{a'}^{a''}D_{a''}^a([\alpha])$ is the cohomology class of
 \begin{gather*}
\sum_{a''_j\in A''}\bigg(\sum_{b'_s\in B'\colon a''_j<b_s'}\partial_{a'}^{b'_s}\bigg(\sum_{c_\ell\in C}\partial_{b'_s}^{c_\ell}\gamma_\ell\bigg)\bigg) =
\sum_{b'_s\in B'}\sum_{c_\ell\in C}\partial_{a'}^{b'_s}\partial_{b'_s}^{c_\ell}\gamma_\ell \\
\hphantom{\sum_{a''_j\in A''}\bigg(\sum_{b'_s\in B'\colon a''_j<b_s'}\partial_{a'}^{b'_s}\bigg(\sum_{c_\ell\in C}\partial_{b'_s}^{c_\ell}\gamma_\ell\bigg)\bigg)}{}
 =\sum_{c_\ell\in C}\sum_{b'_s\in B'}\partial_{a'}^{b'_s}\partial_{b'_s}^{c_\ell}\gamma_\ell=\sum_{c_\ell\in C}0=0.
\end{gather*}
In the f\/irst equality we use the fact that given $b'_s\in B'$, there is only one $a''_j\in A''$ such that $a''_j<b'_s$ and in the third equality we use that the relative BGG sequence is a complex once more.
\end{proof}

\section[Formal exactness of $k$-Dirac complexes]{Formal exactness of $\boldsymbol{k}$-Dirac complexes}\label{section formal exactness}

We will proceed in Section \ref{section formal exactness} as follows. In Section~\ref{section formal neighborhood of X_0} we will recall the def\/inition of the normal bundle of the analytic subvariety $X_0:=\tau^{-1}(x_0)$ and give the def\/inition of the weighted formal neighborhood of $X_0$. In Section~\ref{section formal double complex} we will consider the double complex of twisted relative forms from Section~\ref{section computing PT} and restrict it to the weighted formal neighborhood of~$X_0$. In Section \ref{section les of weighted jets} we will prove that the operators def\/ined in Section \ref{section complex} are dif\/ferential operators and f\/inally, in Theorem~\ref{theorem formal exactness} we will prove that the $k$-Dirac complexes are formally exact.

\subsection[Formal neighborhood of $\tau^{-1}(x_0)$]{Formal neighborhood of $\boldsymbol{\tau^{-1}(x_0)}$}\label{section formal neighborhood of X_0}
Let us f\/irst recall notation from Section \ref{section filtration of TM and TCS}. There is the 2-step f\/iltration $\{0\}=F_0^M\subset F_{-1}^M\subset F_{-2}^M=TM$ and the 3-step f\/iltration $\{0\}=F_0^{CS}\subset F^{CS}_{-1}\subset F^{CS}_{-2}\subset F^{CS}_{-3}=TM$. Moreover, $F_{-1}^{CS}$ decomposes as $E^{CS}\oplus F^{CS}$ where $E^{CS}=\ker(T\eta)$ and $F^{CS}=\ker(T\tau)$. From this it follows that~$E^{CS}$ and~$F^{CS}$ are integrable distributions. Dually, there are f\/iltrations $T^\ast M=F_1^M\supset F_2^M\supset F_3^M=\{0\}$ and $T^\ast CS=F_1^{CS}\supset F_2^{CS}\supset F_3^{CS}\supset F_4^{CS}=\{0\}$ where $F_i^M$ is the annihilator of $F_{-i+1}^M$ and similarly for $F_i^{CS}$. We put $G_i^M:=F_i^M/F_{i+1}^M$ and $G_i^{CS}:=F_i^{CS}/F_{i+1}^{CS}$ so that
\begin{gather*}
 gr(TM)=G_{-2}^M\oplus G_{-1}^M,\qquad gr(T^\ast M)= G_1^M\oplus G_2^M,\\
 gr(TCS)= G_{-3}^{CS}\oplus G_{-2}^{CS}\oplus G_{-1}^{CS},\qquad gr(T^\ast CS)=G_1^{CS}\oplus G_2^{CS}\oplus G_3^{CS},\\
 G_i^M\cong \big(G_{-i}^M\big)^\ast,\quad i=1,2 \qquad \mathrm{and} \qquad G_i^{CS}\cong\big(G_{-i}^{CS}\big)^\ast,\quad i=1,2,3.
\end{gather*}

Let us now brief\/ly recall Section~\ref{section ideal sheaf}. If $X$ is an analytic subvariety of a complex manifold $Y$, then the normal bundle $N_X$ of $X$ in $Y$ is the quotient $(TY|_X)/TX$ and the co-normal bundle $N^\ast_X$ is the annihilator of $TX$ inside $T^\ast X$.
In particular, the origin $x_0$ can be viewed as an analytic subvariety of $M$ with local def\/ining equation $X_1=0$, $X_2=0$ and $ Y=0$ where the matrices are those as in~(\ref{affine subset of M}). For each $i\ge1$ there is the associated ($i$-th power of the) ideal sheaf $\mathcal{I}_{x_0}^i$. This is a sheaf of $\mathcal{O}_M$-modules such that
\begin{gather*}
 (\mathcal{I}_{x_0})^i_x=
 \begin{cases}
 (\mathcal{O}_M)_x, & x\ne x_0,\\
 \mathcal{F}^i_{x_0}, & x=x_0,
 \end{cases}
\end{gather*}
where $\mathcal{O}_M$ is the structure sheaf on $M$, $\mathcal{F}^i_{x}=\{f\in(\mathcal{O}_M)_x\colon j^i_xf=0\}$ and the subscript $x$ stands for the stalk at $x\in M$ of the corresponding sheaf.

Also recall from Section~\ref{section wdo} the def\/inition of weighted jets. For each $i\ge0$, there is a short exact sequence of vector spaces
\begin{gather}\label{ses of weighted jets}
0\rightarrow\mathfrak{F}^{i+1}_{x_0}\rightarrow\mathfrak{F}^i_{x_0}\rightarrow\mathfrak{gr}^{i+1}_{x_0}\rightarrow0,
\end{gather}
where $\mathfrak{F}^i_{x_0}:=\{f\in\mathcal{O}_{x_0}\colon \mathfrak{j}^i_{x_0}f=0\}$. We will view~(\ref{ses of weighted jets}) also as a short exact sequence of sheaves over~$\{x_0\}$.

Put $X_0:=\tau^{-1}(x_0)$. Recall from Lemma~\ref{lemma fibers} that $X_0$ is complex manifold which is biholomorphic to the connected component $\mathrm{Gr}^+_h(n,n)$ of $\mathbb{C}^{n}$ in the Grassmannian of maximal totally isotropic subspaces in~$\mathbb{C}^{2n}$.

\begin{Remark}\label{remark notation over X_0}
If $V^{CS}$ is a holomorphic vector bundle over $CS$, we will for brevity put $V:=V^{CS}|_{X_0}$. We also put $\tau_0:=\tau|_{X_0}$.
\end{Remark}

\begin{Lemma}\label{lemma analytic subvariety X_0}\quad
\begin{enumerate}\itemsep=0pt
 \item [$(i)$] $X_0$ is a closed analytic subvariety of $CS$ and there is an isomorphism of sheaves $\mathcal{I}_{X_0}\cong\tau^\ast\mathcal{I}_{x_0}$.
\item [$(ii)$] There is an isomorphism of vector bundles\footnote{Here we use notation set in Remark \ref{remark notation over X_0}.} $TX_0\cong F$.
\item [$(iii)$] The normal bundle $N$ of $X_0$ in $CS$ is isomorphic to $\tau_0^{\ast} T_{x_0}M$. In particular, $N$ is a trivial holomorphic vector bundle.
\end{enumerate}
\end{Lemma}

\begin{proof}\label{lemma X0}
(i) By Lemma \ref{lemma set tau^{-1}(U)}, $\tau^{-1}(\mathcal{X})=\mathcal{X}\times X_0$ where $\mathcal{X}=(\textbf{p}\circ\exp)(\mathfrak{g}_-)$. From this the claim easily follows.

(ii) As $X_0=\tau^{-1}(x_0)$, it is clear that $TX_0=\ker(T\tau)|_{X_0}$. But we know that $\ker(T\tau)=F$ and the claim follows.

(iii) By def\/inition, $\tau^\ast_0 T_{x_0}M=\{(x,v)\,|\, x\in X_0,\ v\in T_{x_0}M\}$. Hence, there is an obvious projection $TCS|_{X_0}\rightarrow\tau^\ast_0 T_{x_0}M$, $(x,v)\mapsto (x,T_x\tau(v))$ which descends to an isomorphism $N\rightarrow\tau^\ast_0 T_{x_0}M$.
\end{proof}

Recall now the linear isomorphisms $\mathfrak{g}_{i}\cong (G_i^M)_{x_0}$, $i=-2,-1,1,2$, from~(\ref{linear isomorphisms over x0}). In particular, we can view $\mathfrak{g}_i$ as the f\/iber of $G_i^M$ over $\{x_0\}$ and thus also as a vector bundle over $\{x_0\}$. We use this point of view in the following def\/inition.

\begin{Definition}\label{df weighted co-normal neigh}
Put $N^\ast_i:=\tau^\ast_0 \mathfrak{g}_i$, $i=1,2$ and $ \mathfrak{S}^\ell N^\ast:=\tau^\ast_0\mathfrak{gr}^{\ell}$, $\ell=0,1,2,\dots$.
\end{Definition}

Notice that $N^\ast_i$, $i=1,2$ and $\mathfrak{S}^\ell N^\ast$, $\ell\ge0$ are by def\/inition trivial holomorphic vector bundles over~$X_0$. Recall from the end of Section~\ref{section wdo} that $gr^\ell_{x_0}$ is the subspace of~$\mathfrak{gr}^\ell_{x_0}$ that is isomorphic to~$ S^\ell\mathfrak{g}_1$.

\begin{Lemma}
The co-normal bundle $N^\ast $ of $X_0$ in $CS$ is isomorphic to $\tau^\ast_0 T^\ast_{x_0}M$ and the bundle~$N^\ast_2$ is isomorphic to~$G_3$. There are short exact sequences of vector bundles
\begin{gather}\label{ses with conormal bdle}
0\rightarrow N^\ast_2\rightarrow N^\ast\rightarrow N^\ast_1\rightarrow0\qquad \text{and}\qquad 0\rightarrow G_2\rightarrow N^\ast_1\rightarrow E^{\ast}\rightarrow0
\end{gather}
over $X_0$. Moreover, for each $\ell\ge0$ there are isomorphisms of vector bundles
\begin{gather}\label{weighted formal neighborhood}
\mathfrak{S}^\ell N^\ast=\bigoplus_{\ell_1+2\ell_2=\ell } {S}^{\ell_1}N^\ast_1\otimes {S}^{\ell_2}N^\ast_2 \qquad \text{and}\qquad
 S^\ell N^\ast_1=\tau^\ast gr^\ell_{x_0}.
\end{gather}
\end{Lemma}

\begin{proof} There is a canonical injective vector bundle map $\tau^\ast_0 T^\ast_{x_0}M\rightarrow T^\ast CS$ and a moment of thought shows that its image is contained in $N^\ast$. By comparing dimensions of both vector bundles, we have $\tau^\ast_0 T^\ast_{x_0}M\cong N^\ast$ and thus the f\/irst claim. It is clear that $N_2^\ast=\tau^\ast_0\mathfrak{g}_2$ is the annihilator of $F_{-2}=(T\tau)^{-1}(\mathfrak{g}_{-1})$ and since $G_3=F_{-2}^\perp$, the second claim follows.

The f\/irst sequence in (\ref{ses with conormal bdle}) is the pullback of the short exact sequence $0\rightarrow \mathfrak{g}_2\rightarrow T^\ast_{x_0}M\rightarrow \mathfrak{g}_1\rightarrow0$ and thus, it is short exact. The exactness of the latter sequence follows from the exactness of $0\rightarrow G^{CS}_2\rightarrow F^\perp/G^{CS}_3\rightarrow E^{CS\ast}\rightarrow0$ and the isomorphisms $N^\ast\cong F^\perp$, $G_3\cong N_2^\ast$ and $N_1^\ast\cong N^\ast/N_2^\ast$.

The isomorphisms in (\ref{weighted formal neighborhood}) follow immediately from def\/initions and the isomorphism (\ref{weighted jets at x_0}).
\end{proof}

We know that $\mathfrak{S}^\ell N^\ast$ is a trivial holomorphic vector bundle over the compact base $X_0$. It follows that any global holomorphic section of $\mathfrak{S}^\ell N^\ast$ is a constant $\mathfrak{gr}^{\ell}$-valued function on $X_0$ and that $\mathfrak{S}^\ell N^\ast$ is trivialized by such sections. The same is obviously true also for $S^\ell N^\ast_1$. Let us formulate this as lemma.

\begin{Lemma}\label{lemma global hol sections over X_0}
The holomorphic vector bundles $\mathfrak{S}^\ell N^\ast$ and $ S^\ell N^\ast_1$ are trivial and there are canonical isomorphisms $\mathfrak{gr}^{\ell}_{x_0}\rightarrow\Gamma(\mathcal{O}(\mathfrak{S}^\ell N^\ast))$ and $gr^\ell_{x_0}\rightarrow\Gamma(\mathcal{O}({S}^\ell N^\ast_1))$ of finite-dimensional vector spaces.
\end{Lemma}

Let us f\/inish this section by recalling the concept of formal neighborhoods (see \cite{B,WW}). Let $\iota_0\colon X_0\hookrightarrow CS$ be the inclusion. Then $\mathcal{F}_{X_0}:=\iota^{-1}_0\mathcal{I}_{X_0}$ is a sheaf of $\mathcal{O}_{X_0}$-modules whose stalk at $x\in X_0$ is the space of germs of holomorphic functions which are def\/ined on some open neighborhood $\mathcal{V}$ of $x$ in $CS$ and which vanish on $\mathcal{V}\cap X_0$. Let us now view the vector space $\mathcal{F}_{x_0}=\{f\in(\mathcal{O}_M)_{x_0}\colon f(x_0)=0\}$ also as a sheaf over $\{x_0\}$. Then (recall from Lemma~\ref{lemma analytic subvariety X_0}) it is easy to see that $\mathcal{F}_{X_0}=\tau^\ast_0\mathcal{F}_{x_0}$. Observe that $\Gamma(\mathcal{F}_{X_0})$ is the space of equivalence classes of holomorphic functions which are def\/ined on an open neighborhood of $X_0$ in $CS$ where two such functions belong to the same equivalence class if they agree on some possibly smaller open neighborhood of~$X_0$.

The inf\/inite-dimensional vector spaces from (\ref{ses of weighted jets}) form a decreasing f\/iltration $\cdots\subset\mathfrak{F}^{i+1}_{x_0}\subset\mathfrak{F}^{i}_{x_0}\subset\cdots$ of $\mathcal{F}_{x_0}=\mathfrak{F}^0_{x_0}$. Then $\mathfrak{F}^i_{X_0}:=\tau^\ast\mathfrak{F}^i_{x_0}$ is a sheaf of $\mathcal{O}_{X_0}$-modules which is naturally a~sub-sheaf of $\mathcal{F}_{X_0}$. This induces a f\/iltration $\dots\subset\mathfrak{F}^{i+1}_{X_0}\subset\mathfrak{F}^{i}_{X_0}\subset\cdots$ of $\mathcal{F}_{X_0}=\mathfrak{F}^0_{X_0}$. Arguing as in Section \ref{section ideal sheaf}, one can show that for each $i\ge0$ there is a short exact sequence of sheaves
\begin{gather*}
0\rightarrow\mathfrak{F}^{i+1}_{X_0}\rightarrow\mathfrak{F}^i_{X_0}\rightarrow\mathcal{O}\big(\mathfrak{S}^{i+1}N^\ast\big)\rightarrow0
\end{gather*}
and thus, the graded sheaf associated to the f\/iltration $\mathcal{F}_{X_0}$ is isomorphic to $\bigoplus_{i\ge1}\mathcal{O}(\mathfrak{S}^iN^\ast)$. Using the analogy with the classical formal neighborhood, we will call the pair $(X,\mathcal{O}^{(i)}_X)$ where $\mathcal{O}_X^{(i)}:=\iota_0^{-1}\mathcal{O}_{CS}/\mathfrak{F}^{i+1}_{X_0}$ the $i$\textit{-th weighted formal neighborhood of} $X_0$. Notice that the f\/iltration $\{\mathfrak{F}^i_{X_0}\colon i=0,1,2,\dots\}$ descends to a f\/iltration of $\mathcal{O}_X^{(i)}$ and that the associated graded sheaf is isomorphic to $\bigoplus_{j\ge0}^{i+1}\mathcal{O}(\mathfrak{S}^jN^\ast)$.

\subsection[The double complex on the formal neighborhood of $\tau^{-1}(x_0)$]{The double complex on the formal neighborhood of $\boldsymbol{\tau^{-1}(x_0)}$}\label{section formal double complex}
Recall from Section~\ref{section relative BGG sequence} that for each $a\in\mathbb{N}^{k,n}_{++}$ there is a $\mathrm{Q}$-dominant and integral weight~$\lambda_a$, an irreducible $\mathrm{Q}$-module $\mathbb{W}_{\lambda_a}$ with lowest weight $-\lambda_a$ and an associated vector bundle $W_{\lambda_a}^{CS}=\mathrm{G}\times_\mathrm{Q}\mathbb{W}_{\lambda_a}$. We will denote by $W_{\lambda_a}$ the restriction of $W_{\lambda_a}^{CS}$ to $X_0$, by $\mathcal{O}(\lambda_a)$ the sheaf of holomorphic sections of $W_{\lambda_a}$, by $\mathcal{E}^{p,q}$ the sheaf of smooth $(p,q)$-forms over $X_0$ and by $\mathcal{E}^{p,q}(\lambda_a)$ the sheaf of $(p,q)$-forms with values in $W_{\lambda_a}$. If $V$ is another vector bundle over $X_0$, then we denote by $V(\lambda_a)$ the tensor product of~$V$ with~$W_{\lambda_a}$. We will use the notation set in~(\ref{notation weighted jets over origin}) and~(\ref{isom of graded jets over origin}).

\begin{Lemma}\label{lemma les of formal neighborhoods}
There is for each $r:=\ell+j\ge0$ a long exact sequence a vector bundles over $X_0$:
\begin{gather}\label{les of formal neighs I}
\mathfrak{S}^{r}N^\ast(\lambda)\xrightarrow{\mathfrak{d}_0} E^\ast\otimes \mathfrak{S}^{r-1}N^\ast(\lambda)\xrightarrow{\mathfrak{d}_1}\Lambda^2E^\ast\otimes \mathfrak{S}^{r-2}N^\ast(\lambda)\xrightarrow{\mathfrak{d}_2}\cdots.
\end{gather}
This sequence contains a long exact subsequence
\begin{gather}\label{les of formal neighs II}
{S}^{r}N^\ast_1 (\lambda)\xrightarrow{\delta_0}E^\ast\otimes {S}^{r-1}N^\ast_1(\lambda)\xrightarrow{\delta_1}\Lambda^2E^\ast\otimes {S}^{r-2}N^\ast_1(\lambda)\xrightarrow{\delta_2}\cdots.
\end{gather}
\end{Lemma}
\begin{proof}
In order to obtain the sequence (\ref{les of formal neighs I}), take the direct sum of all long exact sequences from~(\ref{second exact formal complex of twisted relative de Rham}) indexed by $s_0$, $s_1$, $s_2$ and $s_3$ where $s_0+s_2+2s_3=\ell+j$, $s_1=0$ and restrict it to~$X_0$.
The subsequence~(\ref{les of formal neighs II}) is obtained similarly, we only add one more condition $s_3=0$.
\end{proof}

Recall that each long exact sequence from (\ref{second exact formal complex of twisted relative de Rham}) is induced by the relative twisted de Rham complex by restricting to weighted jets. Hence, also (\ref{les of formal neighs I}) and (\ref{les of formal neighs II}) are naturally induced by this complex.

\begin{Remark}\label{remark sheaves of vector valued forms over X_0 I}
Let $\mathcal{E}^{0,q}(\Lambda^j E^\ast\otimes\mathfrak{S}^\ell N^\ast(\lambda))$ be the sheaf of smooth $(0,q)$-forms with values in the corresponding vector bundle over $X_0$. The vector bundle map $\mathfrak{d}_j$ induces a map of sheaves
\begin{gather}\label{horizontal differential in first double complex with FN}
\mathcal{E}^{0,q}\big(\Lambda^j E^\ast\otimes\mathfrak{S}^\ell N^\ast(\lambda)\big)\rightarrow \mathcal{E}^{0,q}\big(\Lambda^{j+1} E^\ast\otimes\mathfrak{S}^{\ell-1}N^\ast(\lambda)\big),
\end{gather}
which we also denote by $\mathfrak{d}_j$ as there is no risk of confusion.

Recall from (\ref{decomposition of skew symmetric powers}) that $\Lambda^j\mathbb{E}^\ast\otimes\mathbb{W}_\lambda=\bigoplus_{a\in\mathbb{N}^{k,n}_{++}\colon |a|=j}\mathbb{W}_{\lambda_a}$ which gives direct sum decomposition
$\mathcal{E}^{0,q}(\Lambda^j E^\ast\otimes\mathfrak{S}^\ell N^\ast(\lambda))=\bigoplus_{a\in\mathbb{N}^{k,n}_{++}\colon |a|=j}\mathcal{E}^{0,q}(\mathfrak{S}^{\ell}N^\ast(\lambda_a))$.
We see that if $a, a'\in\mathbb{N}^{k,n}_{++}$ are such that $|a|=|a'|-1=j$, then $\mathfrak{d}_j$ induces
\begin{gather}\label{graded differential I}
\mathfrak{d}_{a'}^a\colon \ \mathcal{E}^{0,q}\big(\mathfrak{S}^{\ell}N^\ast(\lambda_a)\big)\rightarrow\mathcal{E}^{0,q}\big(\mathfrak{S}^{\ell-1}N^\ast(\lambda_{a'})\big)
\end{gather}
in the same way $\partial_\eta$ induces in (\ref{dif op in BGG}) the operator $\partial_{a'}^{a}$ in the relative BGG sequence. By Proposition \ref{thm relative bgg sequence}, $\mathfrak{d}_{a'}^{a}=0$ if $a\nless a'$.
\end{Remark}

\begin{Remark}\label{remark sheaves of vector valued forms over X_0 II}
Replacing (\ref{les of formal neighs I}) by (\ref{les of formal neighs II}) in Remark \ref{remark sheaves of vector valued forms over X_0 I}, we get a map of sheaves
\begin{gather}\label{horizontal differential in second double complex with FN}
\delta_j\colon \ \mathcal{E}^{0,q}\big(\Lambda^j E^\ast\otimes S^\ell N_1^\ast(\lambda)\big)\rightarrow \mathcal{E}^{0,q}\big(\Lambda^{j+1} E^\ast\otimes S^{\ell-1}N_1^\ast(\lambda)\big).
\end{gather}
 If $a,a'$ are as above, then there is a map
\begin{gather*}%\label{graded differential II}
\delta_{a'}^a\colon \ \mathcal{E}^{0,q}\big(S^{\ell}N^\ast_1(\lambda_a)\big)\rightarrow\mathcal{E}^{0,q}\big(S^{\ell-1}N^\ast_1(\lambda_{a'})\big),
\end{gather*}
which is induced in the same way $\mathfrak{d}_j$ induces $\mathfrak{d}_{a'}^a$.
\end{Remark}

Even though the proof of Lemma \ref{lemma sheaf cohomology groups over origin} is trivial, it will be crucial later on.

\begin{Lemma}\label{lemma sheaf cohomology groups over origin}
Let $a\in\mathbb{N}^{k,n}_{++}$. Then
\begin{gather}\label{first isom}
(\tau_0)^q_\ast\big(\mathcal{O}\big(\mathfrak{S}^\ell N^\ast (\lambda_a)\big)\big)=H^q\big(X_0,\mathcal{O}\big(\mathfrak{S}^\ell N^\ast (\lambda_a)\big)\big)=
\begin{cases}
\mathfrak{gr}^{\ell}\mathbb{V}_{\mu_a}, \\
\{0\}
\end{cases}
\end{gather}
and
\begin{gather}\label{second isom}
(\tau_0)^q_\ast\big(\mathcal{O}({S}^\ell N^\ast_1(\lambda_a))\big)=H^q\big(X_0,\mathcal{O}({S}^\ell N_1^\ast (\lambda_a))\big)=
\begin{cases}
gr^\ell\mathbb{V}_{\mu_a}, \\
\{0\},
\end{cases}
\end{gather}
where\footnote{As above, we identify a sheaf over $\{x_0\}$ with its stalk.} in \eqref{first isom} and \eqref{second isom} the first possibility holds if and only if $a\in S^k$ and $q=\ell(a)$.
\end{Lemma}

\begin{proof}The f\/irst equality in (\ref{first isom}) is just the def\/inition of $(\tau_0)_\ast^q$. The sheaf cohomology group in the middle is equal to the cohomology of the Dolbeault complex. In view of Lemma~\ref{lemma global hol sections over X_0}, $\Gamma(\mathcal{E}^{0,q}(\mathfrak{S}^\ell N^\ast(\lambda_a)))\cong\mathfrak{gr}^{\ell}_{x_0}\otimes\Gamma(\mathcal{E}^{0,q}(\lambda_a))$ and thus, the sheaf cohomology group is isomorphic to $\mathfrak{gr}^{\ell}_{x_0}\otimes H^q(X_0,\mathcal{O}(\lambda_a))$. By the Bott--Borel--Weil theorem, $H^q(X_0,\mathcal{O}(\lambda_a))\cong\mathbb{V}_{\mu_a}$ if $a\in S^k$, $q=\ell(a)$ and vanishes otherwise. The second equality in~(\ref{first isom}) then follows from the isomorphism $\mathfrak{gr}^{\ell}_{x_0}\otimes\mathbb{V}_{\mu_a}\rightarrow\mathfrak{gr}^\ell\mathbb{V}_{\mu_a}$ from (\ref{isom of graded jets over origin}).

The isomorphism in (\ref{second isom}) is proved similarly. We only use the other isomorphism
\begin{gather*} \Gamma\big(\mathcal{O}\big({S}^\ell N^\ast_1\big)\big)\rightarrow gr^\ell_{x_0}
\end{gather*} from Lemma \ref{lemma global hol sections over X_0} and the isomorphism $gr^{\ell}_{x_0}\otimes\mathbb{V}_{\mu_a}\rightarrow gr^\ell\mathbb{V}_{\mu_a}$.
\end{proof}

There is for each non-negative integer a certain double complex whose horizontal dif\/ferential is~(\ref{horizontal differential in first double complex with FN}) and the vertical dif\/ferential is (up to sign) the Dolbeault dif\/ferential. This is the double complex from Proposition \ref{thm double complex I} restricted to the weighted formal neighborhood of $X_0$.

\begin{Proposition}\label{thm first double complex on formal neigh}
Let $r\ge0$ be an integer. Then there is a double complex $(\mathfrak{E}^{p,q}(r),d',d'')$ where:
\begin{itemize}\itemsep=0pt
 \item $\mathfrak{E}^{p,q}(r)=\Gamma(\mathcal{E}^{0,q}(\Lambda^p E^\ast\otimes\mathfrak{S}^{r-p}N^\ast(\lambda)))$,
 \item the vertical differential $d'$ is $(-1)^p\bar\partial$ where $\bar\partial$ is the standard Dolbeault differential and
 \item the horizontal differential $d''$ is $\mathfrak{d}_p$ from \eqref{horizontal differential in first double complex with FN}.
\end{itemize}
Moreover, we claim that:
\begin{enumerate}\itemsep=0pt
 \item [$(i)$] $H^j(T^\ast(r),d'+d'')=0$ if $j>{n\choose2}$ where $T^i(r):=\bigoplus_{p+q=i}\mathfrak{E}^{p,q}(r)$;
\item [$(ii)$] the first page of the spectral sequence associated to the filtration by columns is
\begin{gather*}%\label{spaces on first page}
 \mathfrak{E}^{p,q}_1(r)=\bigoplus_{a\in S^k\colon |a|=p,\ \ell(a)=q}\mathfrak{gr}^{r-p}\mathbb{V}_{\mu_a};
\end{gather*}
\item [$(iii)$] the spectral sequence degenerates on the second page.
\end{enumerate}
\end{Proposition}

\begin{proof}Recall from the proof of Proposition \ref{thm relative de Rham} that $\mathfrak{gr}\partial_\eta$ is induced from $\partial_\eta$ by passing to weighted jets (as explained at the end of Section~\ref{section wdo}) and, see Lemma~\ref{lemma les of formal neighborhoods}, that $\mathfrak{d}=\mathfrak{d}_p$ is the restriction of the map $\mathfrak{gr}\partial_\eta$ to the sub-complex~(\ref{les of formal neighs I}). Since $[\partial_\eta,\bar\partial]=0$, we have that $[\mathfrak{d},\bar\partial]=0$ and thus also $d'd''=-d''d'$. This shows the f\/irst claim.

(i) The rows of the double complex are exact as the sequence~(\ref{les of formal neighs I}) is exact. Since $\dim X_0={n\choose2}$, it follows that $\mathfrak{E}^{p,q}_1(r)=0$ whenever $q>{n\choose2}$. This proves the claim.

(ii) By def\/inition, $\mathfrak{E}^{p,q}_1(r)$ is the $d'$-cohomology group in the $p$-th row and $q$-th column. The claim then follows from the direct sum decomposition from Remark~\ref{remark sheaves of vector valued forms over X_0 I} and Lemma~\ref{lemma sheaf cohomology groups over origin}.

(iii) The space $\mathfrak{gr}^{r-p}\mathbb{V}_{\mu_a}$ lives on the $|a|$-th vertical line and $\ell(a)=({n\choose 2}-q(a))$-th horizontal line of the f\/irst page of the spectral sequence and thus, on the $(|a|+{n\choose2}-q(a))=(2q(a)+d(a)+{n\choose2}-q(a))=(r(a)+{n\choose2})$-th diagonal. Choose $a'\in S^k$ such that $\mathfrak{gr}^{r-|a'|}\mathbb{V}_{\mu_{a'}}$ lives on the next diagonal and $a<a'$. This means that $r(a')=r(a)+1$ and so $q(a)=q(a')$ or $q(a')=q(a)+1$. In the f\/irst case, $\mathfrak{gr}^{r-|a'|}\mathbb{V}_{\mu_{a'}}$ lives on the $\ell(a)$-th row. In the second case, it lives on the $(\ell(a)-1)$-th row. As $\mathfrak{d}_{a'}^a=0$ if $a\nless a'$, it follows from def\/inition that the dif\/ferential on the $i$-th page is zero if $i>2$.
\end{proof}

If we use the exactness of (\ref{les of formal neighs II}) instead of (\ref{les of formal neighs I}) and use the isomorphism (\ref{second isom}) instead of~(\ref{first isom}), the proof of Proposition~\ref{thm first double complex on formal neigh} gives the following.

\begin{Proposition}\label{thm second double complex on formal neigh}
The double complex from Proposition {\rm \ref{thm first double complex on formal neigh}} contains a double complex \linebreak $(F^{p,q}(r),d',d'')$ where 	 $F^{p,q}(r):=\Gamma(\mathcal{E}^{0,q}(\Lambda^pE^\ast\otimes{S}^{r-p}N^\ast_1(\lambda)))$.
Moreover we claim that:
\begin{enumerate}\itemsep=0pt
 \item [$(i)$] $H^j(T^\ast(r),d'+d'')=0$ if $j>{n\choose2}$ where $T^i(r):=\bigoplus_{p+q=i}F^{p,q}(r)$;
\item [$(ii)$]
the first page of the spectral sequence associated to the filtration by columns is
\begin{gather*}
F^{p,q}_1(r):=\bigoplus_{a\in S^k\colon  |a|=p,\ \ell(a)=q}gr^{r-p}\mathbb{V}_{\mu_a};
\end{gather*}
\item [$(iii)$] the spectral sequence degenerates on the second page.
\end{enumerate}
\end{Proposition}

\subsection{Long exact sequence of weighted jets}\label{section les of weighted jets}
Let $a\in S^k$ and $\mathbb{V}_{\mu_a}$ be an irreducible $\mathrm{P}$-module with lowest weight $-\mu_a$, see Proposition \ref{thm direct images}.
Now we are ready to show that the linear operators def\/ined in Lemma \ref{lemma diff op} are dif\/ferential operators and we give an upper bound on their weighted order.

\begin{Lemma}\label{lemma diff op on graded jets}
Let $a,a'\in S^k$ be such that $a<a'$ and $ r(a')=r(a)+1$. Then the operator $D^a_{a'}$ from Lemma~{\rm \ref{lemma diff op}} is a differential operator of weighted order at most $s:=|a'|-|a|$.

Hence, $D^{a}_{a'}$ induces for each $i\ge 0$ a linear map
\begin{gather}\label{map of weighted jets I}
\mathfrak{gr} D_{a'}^a\colon \ \mathfrak{gr}^i\mathbb{V}_{\mu_a}\rightarrow\mathfrak{gr}^{i-s}\mathbb{V}_{\mu_{a'}},
\end{gather}
which restricts to a linear map
\begin{gather}\label{map of weighted jets II}
gr D_{a'}^a\colon \ gr^i\mathbb{V}_{\mu_a}\rightarrow gr^{i-s}\mathbb{V}_{\mu_{a'}}.
\end{gather}
\end{Lemma}
\begin{proof}
Let us make a few preliminary observations.
Let $v\in\mathcal{O}_\mathfrak{p}(\mu_a)_{x_0}$. By the $\mathrm{G}$-invariance of $D_{a'}^a$, it is obviously enough to show that $(D_a^{a'}v)(x_0)$ depends only on $\mathfrak{j}_{x_0}^sv$. We may assume that $v$ is def\/ined on the Stein set $\mathcal{U}$ from Section~\ref{section complex} and so we can view $v$ as a cohomology class $[\alpha]=H^{\ell(a)}(\tau^{-1}(\mathcal{U}),\mathcal{O}_\mathfrak{q}(\lambda_a))$.
A choice of Weyl structure (see~\cite{CS}) and the isomorphisms~(\ref{first isom}) give for each integer $i\ge0$ isomorphisms
\begin{gather*}
 \mathfrak{J}^i\mathbb{V}_{\mu_a}\rightarrow\bigoplus_{j=0}^i\mathfrak{gr}^j\mathbb{V}_{\mu_a}\rightarrow \bigoplus_{j=0}^iH^{\ell(a)}\big(X_0,\mathfrak{S}^jN^\ast(\lambda_a)\big).
\end{gather*}
Hence, the Taylor series of $v$ at $x_0$ determines an inf\/inite\footnote{We will at this point avoid discussion about the convergence of the sum as we will not need it.}
 sum $\sum\limits_{j=0}^\infty [v_j]$ where each $[v_j]$ belong to $H^{\ell(a)}(X_0,\mathfrak{S}^j(\lambda_a))$.

Now we can proceed with the proof. By assumption, $s\in\{1,2\}$. If $s=1$, then $\ell(a)=\ell(a')$. By def\/inition, $D_{a'}^av$ corresponds to $[\partial_{a'}^a\alpha]\in H^{\ell(a')}(\tau^{-1}(U),\mathcal{O}_\mathfrak{q}(\lambda_{a'}))$ and $\mathfrak{j}^i_{x_0}(D_{a'}^av)$ can be viewed as $\sum\limits_{j=0}^i[(\mathfrak{d}_{a'}^{a}) v_{j+1}]$. But since $[\mathfrak{d}_{a'}^a (v_j)]\in H^{\ell(a)}(X_0,\mathfrak{S}^{j-1}(\lambda_a))$, it is clear that $D_{a'}^a(v)(x_0)=0$ if $\mathfrak{j}^1_{x_0}v=0$. This completes the proof when $s=1$.

Notice that the linear map $\mathfrak{gr} D_{a'}^a$ f\/its into a commutative diagram
\begin{gather}\label{com diagram with delta}\begin{split}
\xymatrix{
\Gamma(\mathcal{E}^{0,\ell(a)}\big(\mathfrak{S}^iN^\ast(\lambda_a))\big)\cap\ker\bar\partial\ar[d]\ar[r]^{\mathfrak{d}_{a'}^a}& \Gamma\big(\mathcal{E}^{0,\ell(a')}\big(\mathfrak{S}^{i-1}N^\ast(\lambda_{a'})\big)\big)\cap\ker\bar\partial\ar[d]\\
H^{\ell(a)}\big(X_0,\mathfrak{S}^iN^\ast(\lambda_a)\big)\ar[r]&H^{\ell(a')}\big(X_0,\mathfrak{S}^{i-1}N^\ast(\lambda_{a'})\big)\\
\mathfrak{gr}^i\mathbb{V}_{\mu_a}\ar[u]\ar[r]^{\mathfrak{gr} D_{a'}^a}&\mathfrak{gr}^{i-1}\mathbb{V}_{\mu_{a'},}\ar[u] }\end{split}
\end{gather}
where the lower vertical arrows are the isomorphisms from Lemma~\ref{lemma sheaf cohomology groups over origin}, the upper vertical arrows are the canonical projections and the map $\mathfrak{d}_{a'}^a$ is the one from~(\ref{graded differential I}).

Let us now assume $s=2$. In view of the diagram (\ref{def of second order operator}), we have to replace in (\ref{com diagram with delta}) the map~$\mathfrak{d}_{a'}^a$ by the diagram
\begin{gather*}
\xymatrix{\Gamma(\mathcal{E}^{0,q}(\mathfrak{S}^iN^\ast(\lambda_a)))\cap \operatorname{Ker}(\bar\partial)\ar[r]^{(\mathfrak{d}_b^{a})\oplus(\mathfrak{d}_c^{a})}&\ \ \ \ \Gamma(\mathcal{E}^{0,q}(\mathfrak{S}^{i-1}N^\ast({\lambda_b}\oplus {\lambda_c})))\\
\Gamma(\mathcal{E}^{0,q-1}(\mathfrak{S}^{i-1}N^\ast({\lambda_b}\oplus {\lambda_c})))\ar[ur]^{\bar\partial}\ar[r]^{(\mathfrak{d}_{a'}^{b})+(\mathfrak{d}_{a'}^{c})}&\ \ \ \ \ \Gamma(\mathcal{E}^{0,q-1}(\mathfrak{S}^{i-2}N^\ast(\lambda_{a'})))\cap \operatorname{Ker}(\bar\partial),}
\end{gather*}
where we for brevity put $\mathfrak{S}^\bullet N^\ast(\lambda_b\oplus\lambda_c):=\mathfrak{S}^\bullet N^\ast(\lambda_b)\oplus\mathfrak{S}^\bullet N^\ast(\lambda_c)$. Following the same line of arguments as in the case $s=1$, we easily f\/ind that $D_{a'}^a(v)(x_0)=0$ whenever $\mathfrak{j}^2_{x_0}v=0$.

In order to prove the claim about $gr D_{a'}^a$, we need to replace everywhere $\mathfrak{d}_{a'}^{a}$ by its restric\-tion~$\delta_{a'}^a$ and use~(\ref{second isom}) instead of~(\ref{first isom}).
\end{proof}

In order to get rid of the factor $s$ in (\ref{map of weighted jets I}) and (\ref{map of weighted jets II}), we shift the gradings by introducing $\mathfrak{gr}^i\mathbb{V}_{\mu_a}[\uparrow]:=\mathfrak{gr}^{i-q(a)}\mathbb{V}_{\mu_a}$ and $gr^i\mathbb{V}_{\mu_a}[\uparrow]:= gr^{i-q(a)}\mathbb{V}_{\mu_a}$. We can now rewrite the maps from~(\ref{map of weighted jets I}) and~(\ref{map of weighted jets II}) as
\begin{gather*}%\label{dif op on shifted weighted jets}
\mathfrak{gr} D_{a'}^a\colon \ \mathfrak{gr}^\ell\mathbb{V}_{\mu_a}[\uparrow]\rightarrow\mathfrak{gr}^{\ell-1}\mathbb{V}_{\mu_{a'}}[\uparrow]\qquad \mathrm{and}\qquad  grD_{a'}^a\colon \ gr^\ell\mathbb{V}_{\mu_a}[\uparrow]\rightarrow gr^{\ell-1}\mathbb{V}_{\mu_{a'}}[\uparrow],
\end{gather*}
respectively, where $\ell\ge0$ is the corresponding integer. We also put
\begin{gather*}
 \mathfrak{gr}^\ell\mathbb{V}_{j,i}[\uparrow]= \!\!\bigoplus_{a\in S^k_j\colon  q(a)=i} \!\!\mathfrak{gr}^\ell\mathbb{V}_{\mu_a}[\uparrow],\qquad \mathfrak{gr}^\ell\mathbb{V}_j[\uparrow]=\bigoplus_{i=0}^{j}\mathfrak{gr}^\ell\mathbb{V}_{j,i}[\uparrow]
\end{gather*}
and
\begin{gather*}
 gr^\ell\mathbb{V}_{j,i}[\uparrow]= \!\!\bigoplus_{a\in S^k_j\colon  q(a)=i}\!\!  gr^\ell\mathbb{V}_{\mu_a}[\uparrow],\qquad gr^\ell\mathbb{V}_j[\uparrow]=\bigoplus_{i=0}^{j}gr^\ell\mathbb{V}_{j,i}[\uparrow].
\end{gather*}
We view $\mathfrak{gr} D_{a'}^a$ also as a map $\mathfrak{gr}^\ell\mathbb{V}_j[\uparrow]\rightarrow\mathfrak{gr}^{\ell-1}\mathbb{V}_{j+1}[\uparrow]$ by extending it from $\mathfrak{gr}^\ell\mathbb{V}_{\mu_a}[\uparrow]$ by zero to all the other summands. We put
\begin{gather*}%\label{dif op on graded jets}
 \mathfrak{gr} D_j:= \!\!\sum_{a\in S^k_j,\ a'\in S^k_{j+1}\colon a<a'} \!\! \mathfrak{gr} D_{a'}^a\colon \ \mathfrak{gr}^\ell\mathbb{V}_j[\uparrow]\rightarrow\mathfrak{gr}^{\ell-1}\mathbb{V}_{j+1}[\uparrow]
\end{gather*}
and
\begin{gather*}
 \mathfrak{gr}( D_j)_{i'}^i\colon \ \mathfrak{gr}^\ell\mathbb{V}_{j,i}[\uparrow]\rightarrow\mathfrak{gr}^\ell\mathbb{V}_j[\uparrow]\xrightarrow{\mathfrak{gr} D_j}\mathfrak{gr}^{\ell-1}\mathbb{V}_{j+1}[\uparrow]\rightarrow \mathfrak{gr}^{\ell-1}\mathbb{V}_{j+1,i'}[\uparrow], %\label{dif op on graded jets I}
\end{gather*}
where the f\/irst map is the canonical inclusion and the last map is the canonical projection.

Recall from Section \ref{section the relative Weyl group} that if $a<a'$, $a\in S^k_j$, $a'\in S^k_{j+1}$, then $q(a')\le q(a)+1$. This implies that $\mathfrak{gr}(D_j)_{i'}^{i}=0$ if $i\ne i'$ or $i'\ne i+1$. Then $\mathfrak{gr} D_j$ is
 \begin{gather*}%\label{graded map in components}
\begin{matrix}
\left(\begin{matrix}
\mathfrak{gr}^\ell\mathbb{V}_{j,0}[\uparrow]\\
\oplus \\
\mathfrak{gr}^\ell\mathbb{V}_{j,1}[\uparrow]\\
\oplus\\
\dots\\
\end{matrix}\right)
\begin{matrix}
\longrightarrow\\
\searrow\\
\longrightarrow\\
\searrow\\
\dots
\end{matrix}
\left(
\begin{matrix}
 \mathfrak{gr}^{\ell-1}\mathbb{V}_{j+1,0}[\uparrow]\\
 \oplus\\
 \mathfrak{gr}^{\ell-1}\mathbb{V}_{j+1,1}[\uparrow]\\
 \oplus\\
 \dots\\
\end{matrix}
\right)
\end{matrix},
 \end{gather*}
where the horizontal arrows and the diagonal arrows are $\mathfrak{gr}(D_j)^i_i$ and $\mathfrak{gr}(D_j)^i_{i+1}$, respectively.

We similarly def\/ine linear maps $gr D_j\colon gr^\ell\mathbb{V}_j[\uparrow]\rightarrow gr^{\ell-1}\mathbb{V}_{j+1}[\uparrow]$ and $gr (D_j)_{i'}^{i}\colon gr^\ell\mathbb{V}_{j,i}[\uparrow]\rightarrow gr^{\ell-1}\mathbb{V}_{j+1,i'}[\uparrow]$.

\begin{Remark}\label{remark dif op and ss}
Notice that
\begin{gather*}
\mathfrak{gr}^\ell\mathbb{V}_{j,i}[\uparrow] =\bigoplus_{a\in S^k_j\colon  q(a)=i}\mathfrak{gr}^\ell\mathbb{V}_{\mu_a}[\uparrow]=\bigoplus_{a\in S^k\colon  r(a)=j,\ q(a)=i}\mathfrak{gr}^{\ell-q(a)}\mathbb{V}_{\mu_a}\\
\hphantom{\mathfrak{gr}^\ell\mathbb{V}_{j,i}[\uparrow]}{}
=\bigoplus_{a\in S^k\colon  |a|=i+j,\ \ell(a)={n\choose2}-i}\mathfrak{gr}^{\ell-i}\mathbb{V}_{\mu_a}=\mathfrak{E}_1^{i+j,{n\choose2}-i}(\ell+j).
\end{gather*}
Put $p:=i+j$, $q:={n\choose2}-i$ and $r:=\ell+j$. Then we can view $\mathfrak{gr}(D_j)_{i}^i$ and $\mathfrak{gr}(D_j)_{i+1}^i$ as maps
\begin{gather*}%\label{graded maps as differential in ss I}
 \mathfrak{E}^{p,q}_1(r)\rightarrow \mathfrak{E}_1^{p+1,q}(r)\qquad \mathrm{and}\qquad \mathfrak{E}^{p,q}_1(r)\rightarrow \mathfrak{E}_1^{p+2,q-1}(r),
\end{gather*}
respectively. By the def\/inition of $\mathfrak{gr}(D_j)_{i}^i$ from Lemma~\ref{lemma diff op on graded jets}, it follows that we can view it as the dif\/ferential~$d_1$ on the f\/irst page of the spectral sequence from Proposition~\ref{thm first double complex on formal neigh}.

Suppose that $v\in \mathfrak{E}^{p,q}_1(s)$ satisf\/ies $d_1(v)=0$. Then we can apply the dif\/ferential $d_2$ living on the second page to $v+\operatorname{im}(d_1)$ and, comparing this with the def\/inition of~$\mathfrak{gr}(D_j)_{i+1}^i$ from Lemma~\ref{lemma diff op on graded jets}, we f\/ind that
\begin{gather} \label{d_2 as differential operator}
d_2(v+\operatorname{im}(d_1))=\mathfrak{gr}(D_j)_{i+1}^i(v)+\operatorname{im}(d_1).
\end{gather}

Similarly we f\/ind that $gr^\ell\mathbb{V}_{j,i}=F_1^{p,q}(s)$ where $p$, $q$ and $s$ are as above. Moreover we can view $gr (D_j)_{i}^i$ and $gr (D_j)_{i+1}^i$ as maps
\begin{gather*}% \label{graded maps as differential in ss II}
 F_1^{p,q}(s)\rightarrow F_1^{p+1,q}(s)\qquad \mathrm{and}\qquad F_1^{p,q}(s)\rightarrow F_1^{p+2,q-1}(s),
\end{gather*}
respectively. As the double complex from Proposition \ref{thm second double complex on formal neigh} is a sub-complex of the double complex from Proposition~\ref{thm first double complex on formal neigh} and $gr (D_j)_{i'}^i$ is the restriction of~$\mathfrak{gr} (D_j)_{i'}^i$ to the corresponding subspace, we see that $gr (D_j)_{i}^i$ coincides with the dif\/ferential on the f\/irst page of the spectral sequence from Proposition~\ref{thm second double complex on formal neigh} and that $gr(D_j)^i_{i+1}$ is related to the dif\/ferential on the second page just as $\mathfrak{gr}(D_j)^i_{i+1}$ is related to $d_2$.
\end{Remark}

The exactness of the complex (\ref{les of graded jetsI}) for each $\ell+j\ge0$ implies (see \cite{S}) the exactness of the $k$-Dirac complex at the level of inf\/inite weighted jets at any f\/ixed point. Following \cite{Sp}, we say that the $k$-Dirac complex is formally exact. Notice that for application in \cite{S}, the exactness of the sub-complex~(\ref{les of graded jets II}) for each $\ell+j\ge0$ is a crucial point in the proof of the local exactness of the descended complex and thus, in constructing the resolution of the $k$-Dirac operator.

\begin{Theorem}\label{theorem formal exactness}
The $k$-Dirac complex induces for each $\ell+j\ge0$ a long exact sequence
\begin{gather}\label{les of graded jetsI}
\mathfrak{gr}^{\ell+j}\mathbb{V}_0[\uparrow]\xrightarrow{\mathfrak{gr} D_0}\mathfrak{gr}^{\ell+j-1}\mathbb{V}_1[\uparrow]\rightarrow\dots\rightarrow\mathfrak{gr}^\ell\mathbb{V}_j[\uparrow]\xrightarrow{\mathfrak{gr} D_j}\mathfrak{gr}^{\ell-1}\mathbb{V}_{j+1}[\uparrow]\rightarrow\cdots
\end{gather}
 of finite-dimensional vector spaces. The complex contains a sub-complex
\begin{gather}\label{les of graded jets II}
gr^{\ell+j}\mathbb{V}_0[\uparrow]\xrightarrow{gr D_0}gr^{\ell+j-1}\mathbb{V}_1[\uparrow]\rightarrow\dots\rightarrow gr^\ell\mathbb{V}_j[\uparrow]\xrightarrow{gr D_j}gr^{\ell-1}\mathbb{V}_{j+1}[\uparrow]\rightarrow\cdots,
\end{gather}
which is also exact.
\end{Theorem}
\begin{proof}
Let $v\in\mathfrak{gr}^\ell\mathbb{V}_j[\uparrow]$, $j\ge1$ be such that $\mathfrak{gr} D_j(v)=0$. Write $v=(v_0,\dots,v_j)$ with respect to the decomposition given above, i.e., $v_i\in\mathfrak{gr}^\ell\mathbb{V}_{j,i}[\uparrow]$. Assume that $v_0=v_1=\dots=v_{i-1}=0$ and that $v_i\ne0$. We have that $\mathfrak{gr} (D_{i}^i)(v_i)=0$ and $\mathfrak{gr} (D_{i+1}^i)(v_i)+\mathfrak{gr} (D_{i+1}^{i+1})(v_{i+1})=0$. If we view $v_i$ as an element of $\mathfrak{E}_1^{p,q}(s)$ as in Remark~\ref{remark dif op and ss}, we see that $d_1(v_i)=0$ and by~(\ref{d_2 as differential operator}), we f\/ind that $d_2(v_i)=0$.
By Proposition~\ref{thm first double complex on formal neigh}, the spectral sequence $\mathfrak{E}^{p,q}(r)$ collapses on the second page and by part~(i), we have that $\ker(d_2)=\operatorname{im}(d_2)$ beyond the ${n\choose2}$-th diagonal. By Remark \ref{remark dif op and ss} again, $\mathfrak{gr}^\ell\mathbb{V}_{j,i}[\uparrow]$ lives on the $({n\choose2}+j)$-th diagonal. We see that there are $t_{i-1}\in\mathfrak{gr}^{\ell+1}\mathbb{V}_{j-1,i-1}[\uparrow]$ and $t_{i}\in\mathfrak{gr}^{\ell+1}\mathbb{V}_{j-1,i}[\uparrow]$ such that $\mathfrak{gr}(D_{i-1}^{i-1})(t_{i-1})=0$ and $\mathfrak{gr}(D_{i}^{i-1})(t_{i-1})+\mathfrak{gr}(D_{i}^{i})(t_{i})=v_i$. Hence, we can kill the lowest non-zero component of $v$ and repeating this argument f\/initely many times, we see that there is $t\in\mathfrak{gr}^{\ell+1}\mathbb{V}_{j-1}[\uparrow]$ such that $v=\mathfrak{gr} D_{j-1}(t)$.

The proof of the exactness of the second sequence (\ref{les of graded jets II}) proceeds similarly. We only replace $\mathfrak{gr}^\ell\mathbb{V}_{j}[\uparrow]$ by $gr^\ell\mathbb{V}_{j}[\uparrow]$, $\mathfrak{gr}^\ell\mathbb{V}_{j,i}[\uparrow]$ by $gr^\ell\mathbb{V}_{j,i}[\uparrow]$, use that the second spectral sequence from Proposition~\ref{thm second double complex on formal neigh} has the same key properties as the spectral sequence from Proposition~\ref{thm first double complex on formal neigh} and the end of Remark~\ref{remark dif op and ss}.
\end{proof}

\subsection*{Acknowledgements}

The author is grateful to Vladim\'{\i}r Sou\v{c}ek for his support and many useful conversations. The author would also like to thank to Luk\'a\v{s} Krump for the possibility of using his package for the Young diagrams. The author wishes to thank to the unknown referees for many helpful suggestions which considerably improved this article. The research was partially supported by the grant 17-01171S of the Grant Agency of the Czech Republic.

\pdfbookmark[1]{References}{ref}
\LastPageEnding


\begin{thebibliography}{99}
\footnotesize\itemsep=0pt

\bibitem{B}
Baston R.J., Quaternionic complexes, \href{https://doi.org/10.1016/0393-0440(92)90042-Y}{\textit{J.~Geom. Phys.}} \textbf{8} (1992),
 29--52.

\bibitem{BE}
Baston R.J., Eastwood M.G., The {P}enrose transform. Its interaction with
 representation theory, \textit{Oxford Mathematical Monographs, Oxford Science
 Publications}, The Clarendon Press, Oxford University Press, New York, 1989.

\bibitem{BAS}
Bure\v{s} J., Damiano A., Sabadini I., Explicit resolutions for the complex of
 several {F}ueter operators, \href{https://doi.org/10.1016/j.geomphys.2006.06.001}{\textit{J.~Geom. Phys.}} \textbf{57} (2007),
 765--775.

\bibitem{BS}
Bure\v{s} J., Sou\v{c}ek V., Complexes of invariant operators in several
 quaternionic variables, \href{https://doi.org/10.1080/17476930500482473}{\textit{Complex Var. Elliptic Equ.}} \textbf{51}
 (2006), 463--485.

 \bibitem{CSaI}
\v{C}ap A., Sala\v{c} T., Parabolic conformally symplectic structures~{I}:
 def\/inition and distinguished connections, \href{https://doi.org/10.1515/forum-2017-0018}{\textit{Forum Math.}}, {t}o appear,
 \href{https://arxiv.org/abs/1605.01161}{arXiv:1605.01161}.

\bibitem{CSaII}
\v{C}ap A., Sala\v{c} T., Parabolic conformally symplectic structures II:
 parabolic contactif\/ication, \href{https://doi.org/10.1007/s10231-017-0719-3}{\textit{Ann. Mat. Pura Appl.}}, {t}o appear,
 \href{https://arxiv.org/abs/1605.01897}{arXiv:1605.01897}.

\bibitem{CSaIII}
\v{C}ap A., Sala\v{c} T., Parabolic conformally symplectic structures~{III}:
 invariant dif\/ferential operators and complexes, \href{https://arxiv.org/abs/1701.01306}{arXiv:1701.01306}.

\bibitem{CS}
\v{C}ap A., Slov\'ak J., Parabolic geometries. {I}.~Background and general
 theory, \href{https://doi.org/10.1090/surv/154}{\textit{Mathematical Surveys and Monographs}}, Vol.~154, Amer. Math.
 Soc., Providence, RI, 2009.

\bibitem{CaSS}
\v{C}ap A., Slov\'ak J., Sou\v{c}ek V., Bernstein--{G}elfand--{G}elfand
 sequences, \href{https://doi.org/10.2307/3062111}{\textit{Ann. of Math.}} \textbf{154} (2001), 97--113,
 \href{https://arxiv.org/abs/math.DG/0001164}{math.DG/0001164}.

\bibitem{CSo}
\v{C}ap A., Sou\v{c}ek V., Relative {BGG} sequences: {I}.~{A}lgebra,
 \href{https://doi.org/10.1016/j.jalgebra.2016.06.007}{\textit{J.~Algebra}} \textbf{463} (2016), 188--210, \href{https://arxiv.org/abs/1510.03331}{arXiv:1510.03331}.

\bibitem{CSoI}
\v{C}ap A., Sou\v{c}ek V., Relative {BGG} sequences: {II}.~{BGG} machinery and
 invariant operators, \href{https://doi.org/10.1016/j.aim.2017.09.016}{\textit{Adv. Math.}} \textbf{320} (2017), 1009--1062,
 \href{https://arxiv.org/abs/1510.03986}{arXiv:1510.03986}.

\bibitem{CSSS}
Colombo F., Sabadini I., Sommen F., Struppa D.C., Analysis of {D}irac systems
 and computational algebra, \href{https://doi.org/10.1007/978-0-8176-8166-1}{\textit{Progress in Mathematical Physics}},
 Vol.~39, Birkh\"auser Boston, Inc., Boston, MA, 2004.

\bibitem{CSS}
Colombo F., Sou\v{c}ek V., Struppa D.C., Invariant resolutions for several
 {F}ueter operators, \href{https://doi.org/10.1016/j.geomphys.2005.06.009}{\textit{J.~Geom. Phys.}} \textbf{56} (2006), 1175--1191.

\bibitem{F}
Franek P., Generalized {D}olbeault sequences in parabolic geometry,
 \textit{J.~Lie Theory} \textbf{18} (2008), 757--774, \href{https://arxiv.org/abs/0710.0093}{arXiv:0710.0093}.

\bibitem{GW}
Goodman R., Wallach N.R., Symmetry, representations, and invariants,
 \href{https://doi.org/10.1007/978-0-387-79852-3}{\textit{Graduate Texts in Mathematics}}, Vol.~255, Springer, Dordrecht, 2009.

\bibitem{H}
H\"ormander L., An introduction to complex analysis in several variables, D.~Van Nostrand Co., Inc., Princeton, N.J.~-- Toronto, Ont.~-- London, 1966.

\bibitem{K}
Krump L., A resolution for the {D}irac operator in four variables in dimension~6, \href{https://doi.org/10.1007/s00006-009-0169-0}{\textit{Adv. Appl. Clifford Algebr.}} \textbf{19} (2009), 365--374.

\bibitem{Mo}
Morimoto T., Lie algebras, geometric structures and dif\/ferential equations on
 f\/iltered manifolds, in Lie Groups, Geometric Structures and Dif\/ferential
 Equations~-- One Hundred Years After {S}ophus {L}ie ({K}yoto/{N}ara, 1999),
 \textit{Adv. Stud. Pure Math.}, Vol.~37, Math. Soc. Japan, Tokyo, 2002,
 205--252.

\bibitem{Na}
Nacinovich M., Complex analysis and complexes of dif\/ferential operators, in
 Complex Analysis ({T}rieste, 1980), \href{https://doi.org/10.1007/bfb0061877}{\textit{Lecture Notes in Math.}}, Vol.~950, Springer, Berlin~-- New York, 1982, 105--195.

\bibitem{SSSL}
Sabadini I., Struppa D.C., Sommen F., Van~Lancker P., Complexes of {D}irac
 operators in {C}lif\/ford algebras, \href{https://doi.org/10.1007/s002090100297}{\textit{Math.~Z.}} \textbf{239} (2002),
 293--320.

\bibitem{TSI}
Sala\v{c} T., {$k$}-{D}irac operator and the {C}artan--{K}\"ahler theorem,
 \href{https://doi.org/10.5817/AM2013-5-333}{\textit{Arch. Math. (Brno)}} \textbf{49} (2013), 333--346, \href{https://arxiv.org/abs/1304.0956}{arXiv:1304.0956}.

\bibitem{TS}
Sala\v{c} T., {$k$}-{D}irac operator and parabolic geometries, \href{https://doi.org/10.1007/s11785-013-0292-8}{\textit{Complex
 Anal. Oper. Theory}} \textbf{8} (2014), 383--408, \href{https://arxiv.org/abs/1201.0355}{arXiv:1201.0355}.

\bibitem{TSII}
Sala\v{c} T., {$k$}-{D}irac operator and the {C}artan--{K}\"ahler theorem for
 weighted dif\/ferential operators, \href{https://doi.org/10.1016/j.difgeo.2016.09.004}{\textit{Differential Geom. Appl.}}
 \textbf{49} (2016), 351--371, \href{https://arxiv.org/abs/1601.08077}{arXiv:1601.08077}.

\bibitem{S}
Sala\v{c} T., Resolution of the {$k$}-{D}irac operator, \href{https://doi.org/10.1007/s00006-018-0830-6}{\textit{Adv. Appl.
 Clifford Algebr.}} \textbf{28} (2018), 28:3, 19~pages, \href{https://arxiv.org/abs/1705.10168}{arXiv:1705.10168}.

\bibitem{Sp}
Spencer D.C., Overdetermined systems of linear partial dif\/ferential equations,
 \href{https://doi.org/10.1090/S0002-9904-1969-12129-4}{\textit{Bull. Amer. Math. Soc.}} \textbf{75} (1969), 179--239.

\bibitem{WW}
Ward R.S., Wells Jr. R.O., Twistor geometry and f\/ield theory, \href{https://doi.org/10.1017/CBO9780511524493}{\textit{Cambridge
 Monographs on Mathematical Physics}}, Cambridge University Press, Cambridge,
 1990.

\bibitem{W}
Wells Jr. R.O., Dif\/ferential analysis on complex manifolds, \href{https://doi.org/10.1007/978-0-387-73892-5}{\textit{Graduate
 Texts in Mathematics}}, Vol.~65, 2nd~ed., Springer-Verlag, New York~-- Berlin,
 1980.

\end{thebibliography}
\end{document}